\documentclass[11pt]{article}
\usepackage{graphicx}
\usepackage{amsmath}
\usepackage{amssymb}
\usepackage{theorem}

\numberwithin{equation}{section}

\newtheorem{thm}{Theorem}[section]
\newtheorem{lemma}[thm]{Lemma}
\newtheorem{prop}[thm]{Proposition}
\newtheorem{cor}[thm]{Corollary}
{\theorembodyfont{\rmfamily}
\newtheorem{defn}[thm]{Definition}
\newtheorem{example}[thm]{Example}

\newtheorem{rmk}[thm]{Remark}
}

\newcommand{\qed}{\hfill \mbox{\raggedright \rule{.07in}{.1in}}}
 
\newenvironment{proof}{\vspace{1ex}\noindent{\bf
Proof}\hspace{0.5em}}{\hfill\qed\vspace{1ex}}
\newenvironment{pfof}[1]{\vspace{1ex}\noindent{\bf Proof of
#1}\hspace{0.5em}}{\hfill\qed\vspace{1ex}}

\newcommand{\R}{{\mathbb R}}
\newcommand{\Z}{{\mathbb Z}}

\newcommand{\C}{{\mathbb C}}
\renewcommand{\H}{{\mathbb H}}

 \newcommand{\A}{{\mathbb A}}

\newcommand{\bv}{\bar\varphi}
 \newcommand{\supp}{\operatorname{supp}}
 \newcommand{\sgn}{\operatorname{sgn}}

\newcommand{\BIG}{\displaystyle}

\renewcommand{\Re}{\operatorname{Re}}
\renewcommand{\Im}{\operatorname{Im}}

\title{Operator renewal theory for continuous time dynamical systems with finite and infinite measure}

\author{
Ian Melbourne 
\thanks{Mathematics Institute, University of Warwick, Coventry, CV4 7AL, UK 
\newline i.melbourne@warwick.ac.uk}
\and 
Dalia Terhesiu
\thanks{
Faculty of Mathematics, University of Vienna, 1090 Vienna, AUSTRIA 
daliaterhesiu@gmail.com.
\newline Current address: Mathematics Department, University of Exeter, EX4 4QF, UK}
}

\date{9 April 2014}

\begin{document}

\maketitle

\begin{abstract}

We develop operator renewal theory for flows and apply this to 
obtain results on mixing and rates of mixing for a large class of finite and infinite measure semiflows.
Examples of systems covered by our results include suspensions over parabolic rational maps of the complex plane, and nonuniformly expanding semiflows with indifferent periodic orbits.   

In the finite measure case, the emphasis is on obtaining sharp rates of decorrelations, extending results of Gou\"ezel and Sarig from the discrete time setting to continuous time.  In the infinite measure case, the primary question is to prove results on mixing itself, extending our results in the discrete time setting.
In some cases, we obtain also higher order asymptotics and rates of mixing.
\end{abstract}

\paragraph{AMS Subject Classifications:} 
37A25, 37A40, 37A50, 60K05

\vspace{-3ex}
\paragraph{Keywords:} Infinite ergodic theory, continuous time operator renewal equation, mixing rates.

\section{Introduction}
\label{sec-intro}

This paper is concerned with mixing for continuous time dynamical systems.
To set the background for our results, we begin by discussing developments for discrete time.

Much recent research has centered around the statistical properties of smooth dynamical systems with strong hyperbolicity (expansion/contraction) properties.
Results such as exponential decay of correlations and statistical limit laws are by 
now classical  for uniformly hyperbolic diffeomorphisms~\cite{Bowen75,Ruelle78,Sinai72}.
In particular, if $f:X\to X$ is uniformly hyperbolic and $\mu$ is an equilibrium measure corresponding to a H\"older potential, then the correlation function
$\int_X v \,w\circ f^n\,d\mu-\int_X v\,d\mu\int_X w\,d\mu$ decays exponentially quickly as $n\to\infty$ for H\"older observables $v,w:X\to\R$.
 
Young~\cite{Young98} extended this result to a large class of nonuniformly hyperbolic systems, including planar dispersing billiards, and also established polynomial decay of correlations for systems that are more slowly mixing~\cite{Young99}.
The results were then shown to be optimal by Sarig~\cite{Sarig02} and Gou\"ezel~\cite{Gouezel-sharp}.
Turning to the infinite measure case, the fundamental difference is that 
$\lim_{n\to\infty}\int_X v \,w\circ f^n\,d\mu=0$ for reasonably well-behaved observables $v,w$.   Hence there arises the preliminary problem of showing that
$a_n \int_X v \,w\circ f^n\,d\mu\to
\int_X v\,d\mu\int_X w\,d\mu$ for a suitable normalising sequence $a_n\to\infty$ and
for sufficiently well-behaved $v,w$.   The secondary problem is to estimate the speed of convergence (rate of mixing).
Definitive results on the preliminary problem, and first results on the rate of mixing, were obtained recently in~\cite{MT12}.

In the continuous time situation, decay of correlations is
less well understood.   Exponential decay of correlations has been proved only for a very thin set of Anosov flows (those that possess a contact structure or have exceptionally smooth stable and unstable foliations), see~\cite{Dolgopyat98a,Liverani04}.   On the other hand, superpolynomial decay of correlations holds
for ``typical'' uniformly hyperbolic flows~\cite{Dolgopyat98b,FMT07} for observables that are sufficiently regular.
The typical set of flows includes those with a pair of periodic points whose ratio of periods is Diophantine~\cite{Dolgopyat98b} and also includes an open and dense set of flows~\cite{FMT07}.
Results on superpolynomial decay were
extended by~\cite{M07} to nonuniformly hyperbolic flows whose Poincar\'e map is within the class considered in~\cite{Young98}.
For flows whose Poincar\'e map lies in the class considered in~\cite{Young99},
it was shown in~\cite{M09} that typically polynomial decay holds 
for sufficiently regular observables.

In the current paper, we develop a continuous time operator renewal theory, and thereby obtain results on sharp lower bounds for finite measure semiflows with polynomial decay of correlations, and mixing (as well as higher order asymptotics) for infinite measure semiflows, extending
the discrete time results of~\cite{Gouezel-sharp,Sarig02,MT12}.
Our results hold typically in the same sense as discussed above (so it suffices
that there exists a pair of periods with Diophantine ratio, see hypothesis (A2) and Remark~\ref{rmk-A2} below).

\subsection{Illustrative examples}
To describe the main results, we consider (mainly for convenience) the family of
Pomeau-Manneville intermittent maps~\cite{PomeauManneville80}
 considered by~\cite{LiveraniSaussolVaienti99}, and their suspensions in the continuous time case.  Specifically, define the interval maps
$f:X\to X$, $X=[0,1]$,
\begin{align} \label{eq-LSV}
f(x)=\begin{cases} x(1+2^\gamma x^\gamma), & 0<x<\frac12 \\ 2x-1 & \frac12<x<1
\end{cases},
\end{align}
where $\gamma>0$.
There is a unique (up to scaling) $\sigma$-finite absolutely
continuous invariant measure and this measure is finite if and only if $\gamma<1$.
Such maps have a uniformly expanding (or Gibbs-Markov, see Section~\ref{sec-main} for precise definitions) first return map to the set $Y=[\frac12,1]$. 
Set $\beta=1/\gamma$ and 
\begin{align} \label{eq-xixi}
\xi_\beta(t)  =\begin{cases}
t^{-\beta} & \beta>2 \\ 
(\log t)t^{-2} & \beta=2 \\ 
t^{-(2\beta-2)} & 1<\beta<2 \end{cases}, \quad
\xi_{\beta,\epsilon}(t) = 
\begin{cases}  t^{-(\beta-\epsilon)}, &  \beta\ge 2 \\
t^{-(2\beta-2)}, & 1<\beta<2 \end{cases}.
\end{align}

\vspace{-2ex}
\paragraph{Discrete time, finite measure}
For $\gamma\in(0,1)$, it follows 
from~\cite{Young99,Hu04} that correlations decay like $n^{-(\beta-1)}$:
\[
\int_X v\,w\circ f^n-\int_X v\int_X w
=O(n^{-(\beta-1)}),
\]
for all $v:X\to\R$ H\"older and $w:X\to\R$ lying in $L^\infty$.
This decay rate is sharp~\cite{Gouezel-sharp,Sarig02}:
\[
\int v\,w\circ f^n-\int_X v\int_X w
=c\int_X v\int_X w\; n^{-(\beta-1)}
+O(\xi_{\beta}(n)),
\]
for all observables $v,w$ supported in a compact subset of $(0,1]$ with $v$ H\"older and $w$ in $L^\infty$.   Here $c$ is a positive constant depending only on $\gamma$.

\vspace{-2ex}
\paragraph{Discrete time, infinite measure}
For $\gamma\in(1,2)$, we showed~\cite{MT12} that there is a constant $c>0$
(depending only on $\gamma$) such that
\[
n^{1-\beta} \int_X v \,w\circ f^n\to c \int_X v\int_X w,
\]
 for 
all observables $v,w$ supported in a compact subset of $(0,1]$ with $v$ H\"older and $w$ in $L^\infty$.  
The same result holds for $\gamma=1$ with $n^{1-\beta}$ replaced by $\log n$.
For $\gamma\ge2$, such results cannot hold in  the generality considered in~\cite{MT12} but, using the extra structure of the maps~\eqref{eq-LSV}, 
the corresponding results were obtained by~\cite{Gouezel11} for all $\gamma>1$.
In addition, rates of mixing and higher asymptotics for $\gamma\in(1,2)$ were obtained in~\cite{MT12}, improved upon in~\cite{Terhesiu-app}.
and extended to the case $\gamma\ge2$ in~\cite{Terhesiu-sub}.

\vspace{-2ex}
\paragraph{Continuous time, finite measure}
Now suppose that $\varphi_X:X\to\R^+$ is a H\"older roof function bounded away from zero. 
We form the suspension flow $f_t:\Lambda\to\Lambda$ in the usual way
(see Section~\ref{sec-main} for definitions).
This semiflow has an indifferent periodic orbit corresponding to the indifferent fixed point $0\in X$.

In the finite measure case $\gamma\in(0,1)$, 
it follows from~\cite{M09} that typically 
$\int_\Lambda v \,w\circ f_t-\int_\Lambda v\int_\Lambda w$ decays at the rate $t^{-(\beta-1)}$  as $t\to\infty$ for observables $v,w:\Lambda\to\R$ where $v$ is H\"older and $w$ is sufficiently smooth in the flow direction.

Here we prove that such results are optimal: for any $\epsilon>0$
\[
\int_\Lambda v\,w\circ f_t-\int_\Lambda v\int_\Lambda w
=c\int_\Lambda v\int_\Lambda w\; t^{-(\beta-1)}
+O(\xi_{\beta,\epsilon}(t)),
\]
for all observables $v,w$ supported away from the indifferent periodic orbit 
 with $v$ H\"older and $w$ sufficiently smooth in the flow direction.
If moreover $\int_\Lambda v=0$, then 
$\int v\,w\circ f_t=O(t^{-(\beta-\epsilon)})$.
This is the direct analogue of the results in~\cite{Gouezel-sharp,Sarig02}.

\vspace{-2ex}
\paragraph{Continuous time, infinite measure}  Finally, consider the semiflow
$f_t:\Lambda\to\Lambda$ for $\gamma\ge1$.   For $\gamma\in(1,2)$ we prove in this paper that typically 
\[
t^{1-\beta}
 \int_\Lambda v \,w\circ f_t\to c
\int_\Lambda v\int_\Lambda w,
\] 
for all observables $v,w$ supported away from the indifferent periodic orbit 
 with $v$ H\"older and $w$ sufficiently smooth in the flow direction.
Again the same result holds for $\gamma=1$ with $t^{1-\beta}$ replaced by $\log t$, and we obtain higher order asymptotics.    
This is the direct analogue of the results in~\cite{MT12}.

\subsection{Ingredients of the proofs}

The methods in this paper combine 
\begin{itemize}
\item[(i)] Operator renewal theory developed by~\cite{Gouezel-stable,Sarig02} for the discrete time finite measure situation.
\item[(ii)]  The methods we introduced in the infinite ergodic theory setting in~\cite{MT12} which built upon~\cite{GarsiaLamperti62,Gouezel-stable,Sarig02}.
\item[(iii)] The ideas of~\cite{Dolgopyat98b} for uniformly hyperbolic flows and their extension~\cite{M07,M09} to the nonuniformly hyperbolic setting.   
\end{itemize}
However, there is a fourth and equally important component, namely 
\begin{itemize}
\item[(iv)] An operator renewal equation for flows (Theorem~\ref{thm-renew}).
\end{itemize}

As far as we can tell, the operator renewal equation for flows introduced in
Section~\ref{sec-renew} below has no counterpart in the existing probability theory literature.
Continuous time versions of renewal theory have been developed previously in the probability
theory literature.   
We refer to~\cite[Ch.\ XI]{Feller66} for the general framework surrounding Blackwell's renewal theorem~\cite{Blackwell48}. For such a theorem in the infinite mean setting 
(the continuous time analogue of~\cite{GarsiaLamperti62})
we refer to~\cite[Theorem 1]{Erickson70}. We also mention the work of Kingman (see for instance~\cite{Kingman64}) for the continuous analogue, developed for both finite and infinite mean setting, of Feller's theory on discrete regenerative phenomena.
However, it is unclear how to apply these methods here, and our approach seems to have certain advantages as discussed in Remark~\ref{rmk-oprenew}.

\vspace{2ex}

In Section~\ref{sec-main}, we state our main results for suspensions of nonuniformly expanding maps, and recover the statements in the introduction
(suspensions of intermittent maps) as a special case.
The remainder of this paper is then divided into three parts.
In Part~\ref{part-oprenew}, we derive the operator renewal equation for flows.
In Part~\ref{part-infinite}, we prove our results on infinite measure systems.
In Part~\ref{part-finite}, we prove our results on finite measure systems.
The paper is written in such a way that Parts~\ref{part-infinite}
and~\ref{part-finite} can be read independently.

\begin{rmk}
For discrete time systems, operator renewal theory was developed first in the finite measure case before being extended to the infinite measure situation.
For continuous time systems, we present the material in the reverse order.
The reason for this is that
having formulated the continuous time operator renewal equation described above in~(iv),  it is fairly straightforward to deduce our main results for infinite measure semiflows
from the existing work described in components~(ii) and~(iii).   
(We note however that certain technical estimates
in the proof and usage of Lemma~\ref{lem-R} are considerably more complicated than in the case of discrete time, infinite measure.)
In contrast, although our results for finite measure semiflows follow from components~(i),~(iii) and~(iv), it requires significantly more work to glue these methods together.
\end{rmk}

\paragraph{Notation}
We use the ``big $O$'' and $\ll$ notation interchangeably, writing $a_n=O(b_n)$ or $a_n\ll b_n$ if there is a constant $C>0$ such that
$a_n\le Cb_n$ for all $n\ge1$.
Three positive constants arise frequently throughout the paper: $C_1$ and $C_2$ introduced in Section~\ref{sec-main}, and $c_2=(C_2+1)^{-1}$ which appears for the first time in Section~\ref{sec-decomp}.

\section{Statement of the main results}

\label{sec-main}

\paragraph{Suspension semiflows}

Let $(Y,\mu)$ be a probability space and 
$F:Y\to Y$ an ergodic measure-preserving transformation.   Let $\varphi:Y\to\R^+$
be a measurable roof function.
Form the suspension
$Y^\varphi=\{(y,u)\in Y\times\R: 0\le u\le \varphi(y)\}/\sim$
where $(y,\varphi(y))\sim(Fy,0)$.
The suspension flow $f_t:Y^\varphi\to Y^\varphi$ is given by
$f_t(y,u)=(y,u+t)$ computed modulo identifications and the measure 
$\mu^\varphi=\mu\times{\rm Lebesgue}$ is ergodic and $f_t$-invariant.
In the finite measure case, we normalise by $\bar\varphi=\int_Y\varphi\,d\mu$ so that
$\mu^\varphi=(\mu\times{\rm Lebesgue})/\bv$ is a probability measure.

\paragraph{Gibbs-Markov maps}

We assume throughout that $F:Y\to Y$ is a full branch Gibbs-Markov map.
Roughly speaking $F$ is
uniformly expanding with good distortion properties.  

We recall the key definitions~\cite{Aaronson}.
Let $(Y,\mu)$ be a Lebesgue probability space with
countable measurable partition $\alpha$.
Let $F:Y\to Y$ be an ergodic, conservative, measure-preserving, Markov map transforming
each partition element bijectively onto $Y$.
For any $\theta\in(0,1)$, define $d_\theta(y,y')=\theta^{s(y,y')}$
where the {\em separation time} $s(y,y')$ is the least integer $n\ge0$
such that $F^ny$ and $F^ny'$ lie in distinct partition elements in $\alpha$.
It is assumed that the partition $\alpha$ separates orbits of $F$, so
$s(y,y')$ is finite for all $y\neq y'$.   Then $d_\theta$ is a metric.
Let $F_\theta(Y)$ be the Banach space of $d_\theta$-Lipschitz
functions $v:Y\to\R$ with norm $\|v\|=|v|_\infty+|v|_\theta$ where
$|v|_\theta$ is the Lipschitz constant of $v$.

Define the potential function $g=\log\frac{d\mu}{d\mu\circ F}:Y\to\R$.
We require that $g$ is uniformly piecewise Lipschitz: that is,
$g|_a$ is $d_\theta$-Lipschitz on each $a\in\alpha$ and
the Lipschitz constants can be chosen independent of $a$.

For $n\ge1$ we let $\alpha_n$ denote the partition into $n$-cylinders.
Let $g_n=\sum_{j=0}^{n-1}g\circ F^j$.
It follows from the Lipschitz property of $g$ together with full branches that
there exists a constant $C_1>0$ such that 
\begin{align} \label{eq-GM}
e^{g_n(y)}\le C_1\mu(a), \quad\text{and}\quad
|e^{g_n(y)}-e^{g_n(y')}|\le C_1\mu(a)d_\theta(F^ny,F^ny'),
\end{align}
for all $y,y'\in a$, $a\in\alpha_n$, $n\ge1$.

From now on, we adopt a convenient abuse of notation and  
define $|1_av|_\theta=\sup_{y,y'\in a:y\neq y'}|v(y)-v(y')|/d_\theta(y,y')$. 
We write $1_av\in F_\theta(Y)$ if
$1_av$ is bounded and $|1_av|_\theta<\infty$. 

\paragraph{Roof function}
The roof function $\varphi:Y\to\Z^+$
is assumed to be piecewise Lipschitz with respect to $d_{\theta_0}$ for some
$\theta_0\in(0,1)$,
(ie $1_a\varphi\in F_{\theta_0}(Y)$ for all $a\in\alpha$), and satisfying $\inf \varphi>0$.
For convenience of notation, we suppose that $\inf \varphi>2$.
In particular, the set $\tilde Y=Y\times[0,1]$ lies inside $Y^\varphi$.

We make various further assumptions:
\begin{itemize}
\item[(A1)] There is a constant $C_2>0$ such that 
$|1_a\varphi|_{\theta_0} \le C_2\inf_a\varphi$ for all $a\in\alpha$.
\item[(A2)] There exist two periodic orbits for $f_t$ with periods $\tau_1,\tau_2$ such
that $\tau_1/\tau_2$ is Diophantine.  We require that the periodic orbits intersect $Y$ only in the interior of partition elements.
\end{itemize}

\begin{rmk}  \label{rmk-A2}
Condition (A1) is automatic for a large class of examples discussed in Subsection~\ref{sec-GM}.   
Condition (A2) is sufficient for a rather technical ``approximate eigenfunction'' criterion of Dolgopyat~\cite{Dolgopyat98b} to be satisfied.  This criterion is stated
precisely in Definition~\ref{defn-approx} and holds also for an open dense set of roof functions~\cite{FMT07}.
\end{rmk}

\paragraph{Observables}
We consider observables $v,w:\tilde Y\to\R$ of the following form.
Writing $v^u(y)=v(y,u)$, define $|v|_\theta =\sup_{u\in[0,1]}|v^u|_\theta$ 
and $\|v\|_\theta=|v|_\infty+|v|_\theta$.
Then $F_\theta(\tilde Y)$ is the space consisting of those 
$v\in L^\infty(\tilde Y)$ with $\|v\|_\theta<\infty$.

For $m\ge0$,
set $|w|_{\infty,m}=\max_{j=0,\dots,m}|\partial_t^jw|_\infty$.
We write $w\in L^{\infty,m}(\tilde Y)$ if 
$w$ is supported
in $Y\times (0,1)$ with
$|w|_{\infty,m}<\infty$.

\vspace{2ex}
Define 
\[
\rho_{v,w}(t)=\int_{Y^\varphi} v\,w\circ f_t\,d\mu^\varphi,
\]
and write $\bar v=\int_{Y^\varphi}v\,d\mu^\varphi$,
$\bar w=\int_{Y^\varphi}w\,d\mu^\varphi$.  
In the finite measure case, the correlation function of $v$ and $w$ is
given by $\rho_{v,w}(t)-\bar v\bar w$.
We can now state our main theorems.

\begin{thm}[Infinite measure] \label{thm-infinite}
Assume that $F:Y\to Y$ is a full branch Gibbs-Markov map with roof
function $\varphi:Y\to\R^+$ satisfying conditions (A1) and (A2).
\begin{itemize}
\item[(a)]  Suppose that $\varphi$ is nonintegrable and
$\mu(\varphi>t)=\ell(t) t^{-\beta}$ where $\beta\in(\frac12,1]$
and $\ell$ is a measurable slowly varying function (so $\lim_{x\to\infty}\ell(\lambda x)/\ell(x)=1$ for all $\lambda>0$).

Let $d_\beta=\frac{1}{\pi}\sin\beta\pi$ for $\beta<1$ and $d_\beta=1$ for $\beta=1$.
Define $\tilde\ell(t)=\ell(t)$ for $\beta<1$ and 
$\tilde\ell(t)=\int_1^t \ell(s)s^{-1}\,ds$ for $\beta=1$.

Then there exist $\theta\in(0,1)$, $m\ge1$, and a function $a:(0,\infty)\to(0,\infty)$ with
$\lim_{t\to\infty}a(t)=0$ such that
\[
|\tilde\ell(t)t^{1-\beta}\rho_{v,w}(t) - d_\beta\bar v\bar w|
\le \|v\|_{\theta}|w|_{\infty,m}\,a(t), 
\]
for all $v\in F_\theta(\tilde Y)$, $w\in L^{\infty,m}(\tilde Y)$.
\item[(b)]  Suppose moreover
that $\mu(\varphi>t)=ct^{-\beta}+O(t^{-q})$ where $\beta\in(\frac12,1)$,
$q\in(1,2\beta)$ and $c>0$.
There exist constants
$d_1=c^{-1}d_\beta$, $d_2,d_3,\ldots \in\R$, and 
for any $\epsilon>0$, there exist $\theta\in(0,1)$, $m\ge1$, such that
\[
\rho_{v,w}(t)=
\sum_j d_jt^{-j(1-\beta)}
\bar v \bar w
+O(\|v\|_\theta|w|_{\infty,m}\,t^{-\beta(1-q^{-1}(2\beta-1)-\epsilon)}),
\] 
for all $v\in F_\theta(\tilde Y)$, $w\in L^{\infty,m}(\tilde Y)$, $t>0$.
Here, the sum is over those $j\ge1$ with $j(1-\beta)\le \beta(1-q^{-1}(2\beta-1)-\epsilon)$.

In particular, 
if $\mu(\varphi>t)=ct^{-\beta}+O(t^{-2\beta})$, then the error term is of the form
$O(\|v\|_\theta|w|_{\infty,m}\,t^{-(\frac12-\epsilon)})$.  If in addition
$\beta>\frac34$ and $d_2\neq0$, then we obtain
second order asymptotics.  
\end{itemize}
\end{thm}

\begin{rmk}  
Explicit formulas for the constants $d_j$, $j\ge2$, can  be found in~\cite[Section~9]{MT12}.  
Write $\mu(\varphi>t)=ct^{-\beta}(1+H(t))$.
Then $d_j=e_j\int_0^\infty H(t)\,dt$ where $e_j$ is a nonzero constant depending only on $j$ and $\beta$.
In particular, either all the constants $d_j$ are nonzero (the typical case) or $d_j=0$ for all $j\ge2$.
(The functions $H$ and $H_1$ in~\cite[Lemma~3.2]{MT12} coincide in the continuous time context.)
\end{rmk}

In the finite case, define
\begin{align} \label{eq-xi}
\zeta(t) & =\int_t^\infty \mu(\varphi>\tau)\,d\tau, 
\quad
\xi_{\beta,\epsilon}(t) = 
\begin{cases}  t^{-(\beta-\epsilon)}, &  \beta\ge 2 \\
t^{-(2\beta-2)}, & 1<\beta<2 \end{cases}.
\end{align}

\begin{thm}[Finite measure] \label{thm-finite}
Assume that $F:Y\to Y$ is a full branch Gibbs-Markov map with roof
function $\varphi:Y\to\R^+$ satisfying conditions (A1) and (A2).
\begin{itemize}
\item[(a)] Suppose that $\mu(\varphi>t)=O(t^{-\beta})$ where $\beta>1$.
Then for any $\epsilon>0$,  there exists $\theta\in(0,1)$, $m\ge1$, such that
\[
\rho_{v,w}(t)-\bar v\bar w =(1/\bv)\bar v\bar w\,\zeta(t)
+O(\|v\|_\theta|w|_{\infty,m}\,\xi_{\beta,\epsilon}(t)),
\]
for all $v\in F_\theta(\tilde Y)$, $w\in L^{\infty,m}(\tilde Y)$, $t>0$.  
\item[(b)]
Suppose further that 
$\bar v=0$.
Then for any $\epsilon>0$, there exists $\theta\in(0,1)$, $m\ge1$,
such that
\begin{align*}
\rho_{v,w}(t)=O\bigl(\|v\|_\theta|w|_{\infty,m}\,t^{-(\beta-\epsilon)}\bigr),
\end{align*}
for all $v\in F_\theta(\tilde Y)$, $w\in L^{\infty,m}(\tilde Y)$, $t>0$.
\end{itemize}
\end{thm}

\begin{rmk} It is well-known that the regular variation assumption 
is necessary for Theorem~\ref{thm-infinite}.  The assumption of polynomial tails in Theorem~\ref{thm-finite} can be relaxed as in~\cite{Gouezel-phd} or~\cite{MT14} but we do not pursue that here.
\end{rmk}

\subsection{Examples with full branch Gibbs-Markov first return maps}
\label{sec-GM}

In formulating Theorems~\ref{thm-infinite} and~\ref{thm-finite}, we considered suspensions where the map $F:Y\to Y$ is uniformly expanding and the roof function $\varphi:Y\to\R^+$
is unbounded.
Often it is convenient to reverse the roles and to start with a map $f:X\to X$ that is less well-behaved (nonuniformly expanding instead of uniformly 
expanding) together with a bounded roof function $\varphi_X:X\to\R$.

In particular a large class of examples covered by our methods are those where the map $f:X\to X$ has a first return
map $F:Y\to Y$ that is full branch and Gibbs-Markov and where $\varphi_X$ is globally Lipschitz.    This includes suspensions of parabolic rational
maps of the complex plane (Aaronson {\em et al.}~\cite{ADU93}) and Thaler's class of interval
maps with indifferent fixed points~\cite{Thaler95} (in particular the
family~\eqref{eq-LSV} defined above).

The separation time $s$, and hence the metric $d_\theta$, extends from $Y$ to $X$: define $s(f^\ell y,f^\ell y')=s(y,y')$ for all $y,y'\in a$, $a\in\alpha$, $0\le\ell<\tau(y)$.
Suppose that the roof function $\varphi_X:X\to\R^+$ is locally Lipschitz with respect to this metric
and define the induced roof function
$\varphi:Y\to\Z^+$, $\varphi(y)=\sum_{j=0}^{\tau(y)-1}\varphi_X(y)$.   
Thus we obtain equivalent semiflows using either $(f,\varphi_X)$ or $(F,\varphi)$.
Furthermore, condition (A1) is automatic if $\varphi_X$ is globally Lipschitz, and 
the statement of condition (A2) is unchanged in the new setting (the set of periods 
for $(f,\varphi_X)$ or $(F,\varphi)$ are identical).

\begin{example}   \label{ex-PM}
Consider the family of intermittent
maps~\eqref{eq-LSV} discussed in the introduction.
Such maps $f:X\to X$ have a full branch Gibbs-Markov first return map to the set $Y=[\frac12,1]$. 
Recall that $v:X\to\R$ is $C^\eta$ for $\eta\in(0,1]$ if $v$ is continuous and
$\sup_{x\neq x'}|v(x)-v(x')|/|x-x'|^\eta<\infty$.

\begin{prop} \label{prop-PMphi}
Suppose that $\varphi_X:X\to\R^+$ is $C^\eta$, $\eta\in(0,1]$.
Then
\[
\mu(y\in Y:\varphi(y)>t)=c_0 t^{-\beta}(1+O(t^{-\beta\eta})),
\]
where $\beta=1/\gamma$,
$c_0=\frac14\beta^\beta\varphi_X(0)^\beta h(\frac12)$,
and $h:Y\to\R^+$ is the density for $\mu$.   
\end{prop}

\begin{proof}
See~\cite[Theorem~1.3]{Gouezel-stable} for a similar calculation in  the case
$\gamma<1$.
Recall (see for example~\cite[Proposition~11.12]{MT12}) that the first return time $\tau:Y\to\Z^+$ satisfies
\begin{align} \label{eq-PMtau}
\mu(\tau>n)= c_1 n^{-\beta}(1+O(n^{-\beta})),
\end{align}
where
$c_1=\frac14\beta^\beta h(\frac12)$.

If $\tau(y)=k$, then write 
\[
\varphi(y)=\sum_{j=0}^{k-1} \varphi_X(f^jy)=\varphi_X(y)+(k-1)\varphi_X(0)
+\sum_{j=1}^{k-1}(\varphi_X(f^jy)-\varphi_X(0)),
\]
where $f^jy=O((k-j)^{-\beta})$ for $j=1,\dots,k-1$.  It follows that
\[
\Bigl|\sum_{j=1}^{k-1}(\varphi_X(f^jy)-\varphi_X(0))\Bigr|\le
|\varphi_X|_\eta\sum_{j=1}^{k-1}|f^jy)|^\eta \ll k^{1-\eta\beta}.
\]
Hence $\varphi(y)=k(\varphi_X(0)+O(k^{-\beta\eta}))=
\varphi_X(0)\tau(y)(1+O(\tau(y)^{-\beta\eta})$.
This combined with~\eqref{eq-PMtau} yields the result.
\end{proof}

\begin{cor} \label{cor-PM}  
Suppose that $X=[0,1]$ and that $f:X\to X$ is an intermittent map of the form~\eqref{eq-LSV}, with $\gamma\in(0,2)$.
Let $\varphi_X:X\to\R^+$ be a $C^\eta$-roof function, $\eta\in(0,1]$.
Suppose further that the suspension semiflow
 possesses a pair of periodic orbits with Diophantine ratio of periods.  
Let $\beta=1/\gamma$ and 
$c_0=\frac14\beta^\beta\varphi_X(0)^\beta h(\frac12)$.
Then
\begin{itemize}
\item[(a)]  For $\gamma\in[1,2)$,
the conclusion of Theorem~\ref{thm-infinite}(a) holds with $\tilde\ell(t)\sim c_0$ for $\gamma\in(1,2)$ and
$\tilde\ell(t)\sim c_0\log t$ for $\gamma=1$.

Moreover if $\eta\in(0,1]$ is sufficiently large ($\eta\in(\frac{1-\beta}{\beta},1]$ suffices), then the conclusion of Theorem~\ref{thm-infinite}(b) holds.
\item[(b)]  For $\gamma\in(0,1)$, the conclusion of Theorem~\ref{thm-finite}(a)
holds in the form
\[
\rho_{v,w}(t)-\bar v\bar w =(1/\bv)c_0(\beta-1)^{-1}\bar v\bar w\, t^{-(\beta-1)}
+O(t^{-p}),
\]
where $p=\min\{\beta-1+\beta\eta,2\beta-2\}$ for $\beta<2$ and
$p=\min\{\beta-1+\beta\eta,\beta-\epsilon\}$ for $\beta\ge2$.

Moreover, if $\bar v=0$, then we obtain $\rho_{v,w}(t)=O(t^{-(\beta-\epsilon)})$ as in Theorem~\ref{thm-finite}(b).
\end{itemize}
\end{cor}

\begin{proof}
Condition~(A1) is automatic, and we have explicitly assumed condition~(A2).  
Proposition~\ref{prop-PMphi} gives the required estimates on $\mu(\varphi>t)$.
Hence the results follow from
Theorems~\ref{thm-infinite} and~\ref{thm-finite}.
\end{proof}

\end{example}   

\part{Continuous time operator renewal theory}
\label{part-oprenew}

In this part of the paper, we formulate an operator renewal equation for flows.

\section{The operator renewal equation}
\label{sec-renew}

\paragraph{Transfer operators}
Let $R:L^1(Y)\to L^1(Y)$ denote the transfer operator for $F:Y\to Y$
and let $L_t:L^1(Y^\varphi)\to L^1(Y^\varphi)$ denote the family of transfer operators for $f_t$.
(So $\int_Y Rv\,w\,d\mu=\int_Y v\,w\circ F\,d\mu$,
$\int_{Y^\varphi} L_tv\,w\,d\mu^\varphi=\int_{Y^\varphi} v\,w\circ f_t\,d\mu^\varphi$ for suitable test functions $v,w$.)

Recall that $\tilde Y=Y\times[0,1]$.
We define the probability measure $\tilde\mu=\mu\times{\rm Lebesgue}$ on $\tilde Y$.
Note that in the infinite measure case $\mu^\varphi|_{\tilde Y}=\tilde\mu$,
whereas in the finite measure case $\mu^\varphi|_{\tilde Y}=(1/\bv)\tilde\mu$.

Define $\tilde F:\tilde Y\to\tilde Y$ by setting
$\tilde F(y,u)=(Fy,u)$.   Note that
$\tilde F(y,u)=f_{\varphi(y)}(y,u)$.
Define $\tilde\varphi:\tilde Y\to\R^+$, $\tilde\varphi(y,u)=\varphi(y)$.
Then $\tilde F=f_{\tilde\varphi}$.
Let $\tilde R$ denote the transfer operator 
corresponding to the map $\tilde F:\tilde Y\to\tilde Y$
($\int_{\tilde Y}\tilde  Rv\,w\,d\tilde \mu=\int_{\tilde Y} v\,w\circ \tilde F\,d\tilde \mu$).
Given $v\in L^1(\tilde Y)$ and $u\in [0,1]$ we define $v^u\in L^1(Y)$, $v^u(y)=v(y,u)$.
It is easily verified that
\begin{align} \label{eq-Rtilde}
(\tilde R v)(y,u)=(Rv^u)(y).
\end{align}

\paragraph{Renewal operators}

For $t>0$, define $T_t,U_t:L^1(\tilde Y)\to L^1(\tilde Y)$ by setting
\[
T_tv = 1_{\tilde Y}L_t(1_{\tilde Y}v), \quad
U_tv = 1_{\tilde Y}L_t(1_{\{\tilde\varphi>t\}}v).
\]
For $s\in\C$, define the families of operators on $L^1(\tilde Y)$,
\[
\hat R(s)v=\tilde R(e^{-s\tilde\varphi}v), \quad
\hat T(s)v=\int_0^\infty e^{-st}T_tv\,dt, \quad
\hat U(s)v=\int_0^\infty e^{-st}U_tv\,dt.
\]
Note that $\hat R$, $\hat T$, $\hat U$ are analytic on $\H=\{\Re s>0\}$
and that $\hat R$ is well-defined on $\overline\H=\{\Re s\ge0\}$.
We also define $\hat R_0(s):L^1(Y)\to L^1(Y)$ for $s\in\overline\H$:
$\hat R_0(s)v=R(e^{-s\varphi}v)$.

\begin{rmk} \label{rmk-LT} Throughout the paper, we write $\hat a(s)$ to denote a function that is analytic on $\H$, with inverse Laplace transform $a(t)$.

We note that $\hat R_0(s)$ has the formal inverse Laplace transform 
$R_0(t)v=R(\delta_{\varphi}(t)v)$, where $\delta_x$ is the $\delta$-measure at $x$, but this is not used explicitly in the paper.
\end{rmk}

\begin{prop} \label{prop-renew}
$\hat\rho(s)=\int_{\tilde Y}\hat T(s)v\, w\,d\mu^\varphi$ for
$v\in L^1(\tilde Y)$, 
$w\in L^\infty(\tilde Y)$,
$s\in\H$.
\end{prop}

\begin{proof}
We have $\rho(t)=\int_{\tilde Y}v\,w\circ f_t\,d\mu^\varphi
=\int_{\tilde Y}L_tv\,w\,d\mu^\varphi
=\int_{\tilde Y}T_tv\,w\,d\mu^\varphi$, so that
\[
\hat\rho(s)=\int_0^\infty e^{-st}\rho(t)\,dt
=\int_0^\infty e^{-st}\int_{\tilde Y} T_tv\, w\,d\mu^\varphi\,dt
=\int_{\tilde Y}\hat T(s)v\, w\,d\mu^\varphi,
\]
as required.
\end{proof}

\begin{thm} \label{thm-renew}
$\hat T(s)\hat R(s)=\hat T(s)-\hat U(s)$ for $s\in\H$.
\end{thm}

\begin{proof}
Let $w\in L^\infty(\tilde Y)$.  We compute that
\begin{align*}
& \int_{\tilde Y}\hat T(s)\hat R(s)v\,w\,d\tilde\mu  = 
 \int_{\tilde Y} \int_0^\infty e^{-st}L_t\tilde R(e^{-s\tilde\varphi}v)\,w\,dt\,d\tilde\mu
 =\int_0^\infty e^{-st}\int_{\tilde Y} e^{-s\tilde\varphi}v\,w\circ f_t\circ \tilde F\,d\tilde\mu\,dt \\
& \qquad = \int_{\tilde Y} \int_0^\infty e^{-s(t+\tilde\varphi)} v\, w\circ f_{t+\tilde\varphi}\,dt\,d\tilde\mu
 = \int_{\tilde Y} \int_{\tilde\varphi}^\infty e^{-st} v\,w\circ f_t\,dt\,d\tilde\mu \\ & \qquad
 = \int_{\tilde Y} \int_0^\infty e^{-st} v\,w\circ f_t\,dt\,d\tilde\mu
 - \int_{\tilde Y} \int_0^{\tilde\varphi} e^{-st} v\,w\circ f_t\,dt\,d\tilde\mu \\ & \qquad
 = \int_{\tilde Y} \hat T(s)v\,w\,d\tilde\mu
 - \int_{\tilde Y} \int_0^\infty 1_{\{\tilde\varphi>t\}} e^{-st} v\,w\circ f_t\,dt\,d\tilde\mu 
 = \int_{\tilde Y} \hat T(s)v\,w\,d\tilde\mu
 - \int_{\tilde Y} \hat U(s)v\,w\,d\tilde\mu
\end{align*}
so $\hat T\hat R=\hat T-\hat U$ as required.
\end{proof}

For future reference, we record the following formula for $U_t$.

\begin{prop}   \label{prop-U}
Suppose that $v\in L^1(\tilde Y)$.  Then
\[
(U_tv)(y,u)=\begin{cases}
v(y,u-t)1_{[t,1]}(u), & 0\le t\le 1 \\
(\tilde Rv_t)(y,u), & t>1
\end{cases},
\]
where $v_t(y,u)=1_{\{t<\varphi(y)<t+1-u\}}v(y,u-t+\varphi(y))$.
\end{prop}

\begin{proof}
For $t\le1$, we have $U_tv=T_tv$ and so
\begin{align*}
 \int_{\tilde Y} (U_tv)\,w\,d\tilde\mu  & = \int_{\tilde Y} v\,w\circ f_t\,d\tilde\mu = \int_{\tilde Y} v(y,u)\,w(y,u+t)\,d\tilde\mu \\ & 
=\int_Y \int_0^{1-t} v(y,u) w(y,u+t)\,du\,d\mu 
=\int_Y \int_t^1 v(y,u-t) w(y,u)\,du\,d\mu \\ &
=\int_{\tilde Y} v(y,u-t)1_{[t,1]}(u)w(y,u)\,d\tilde\mu.
\end{align*}
For $t>1$,
\begin{align*}
&  \int_{\tilde Y} (U_tv)\,w\,d\tilde\mu   = 
\int_{\tilde Y} T_t(1_{\{\tilde\varphi>t\}}v)\,w\,d\tilde\mu  = 
\int_{\tilde Y} 1_{\{\tilde\varphi>t\}}v\,w\circ f_t\,d\tilde\mu    
\\ & \qquad  = \int_{\tilde Y} (1_{\{\tilde\varphi>t\}}v)(y,u)w(Fy,u+t-\varphi(y))\,d\tilde\mu  
 \\ &  \qquad
= \int_Y\int_{\varphi(y)-t}^1 1_{\{\varphi(y)>t\}}v(y,u)w(Fy,u+t-\varphi(y)))\,du\,d\mu   \\ &  \qquad
= \int_Y\int_0^{1+t-\varphi(y)} 1_{\{\varphi(y)>t\}}v(y,u-t+\varphi(y))w(Fy,u))\,du\,d\mu   \\ &  \qquad
= \int_{\tilde Y} 1_{\{t<\varphi(y)<t+1-u\}}v(y,u-t+\varphi(y))w\circ \tilde F(y,u)\,d\tilde\mu   
= \int_{\tilde Y} (\tilde R v_t)(y,u)   w(y,u)\,d\tilde\mu,
\end{align*}
as required.
\end{proof}

In Section~\ref{sec-GM}, we defined the symbolic metric $d_\theta$ on $Y$
and the Banach space $F_\theta(Y)$ of $d_\theta$-Lipschitz functions
$v:Y\to\R$.
The next result makes use of the Gibbs-Markov structure and a weakened
version of condition (A2).

\begin{prop} \label{prop-H2}
Let $\theta\in(0,1)$.
Viewing the family of twisted transfer operators $\hat R_0(s)$ as operators
on $F_\theta(Y)$,
\begin{itemize}
\item[(a)] The spectral radius of $\hat R_0(s)$ is less than $1$ for
$s\in\overline\H-\{0\}$ and is equal to $1$ for $s=0$.
\item[(b)] $1$ is a simple eigenvalue for $\hat R_0(0)$ and is isolated in
the spectrum of $\hat R_0(0)$.
\end{itemize}
\end{prop}

\begin{proof}  This is standard.   By for example the proof of~\cite[Proposition~11.4]{MT12},
$\hat R_0(s)$ has spectral radius at most $1$ and essential spectral radius at most $\theta$ for all $s\in\overline\H$.  Also the spectral radius is less than $1$ for all $s\in\H$.
Hence it suffices to consider eigenvalues at $1$ for $s=ib$.   
It follows from ergodicity of $F$ that $1$ is a simple eigenvalue for its transfer 
operator $\hat R_0(0)$, so it remains to rule out $1$ as an eigenvalue for $\hat R_0(ib)$, $b\neq0$.

Consider the family of operators $M_b:F_\theta(X)\to F_\theta(X)$ given by
$M_bv=e^{ib\varphi}v\circ F$.   The operators $\hat R_0(ib)$ and $M_b$ are $L^2$ adjoints
so it is equivalent to show that $1$ is not an eigenvalue for $M_b$.
We claim that if $1$ is an eigenvalue, then every period $\tau$ corresponding to
a periodic orbit for the semiflow lies in $(2\pi/b)\Z$ which violates condition (A2).
(For this proposition it suffices to have two irrationally related periods.)

To prove the claim, suppose that $M_bv=v$ for some $v\in F_\theta(X)$, $v\not\equiv0$.
In other words, $e^{ib\varphi}v\circ F=v$.  In particular, $|v|\circ F=|v|$
and it follows by ergodicity that $|v|$ is constant.  Hence $v$ is nonvanishing.
Iterating, we have $e^{ib\sum_{j=0}^{k-1}\varphi\circ F^j}v\circ F^k=v$.
Now suppose that $y$ is a periodic point for $F$ of period $k$.
The period $\tau$ of the corresponding periodic orbit for $f_t$ is given by
$\tau=\sum_{j=0}^{k-1}\varphi(F^jy)$ and so $v(y)=e^{ib\tau}v(y)$.  Dividing by
$v(y)$, we obtain $e^{ib\tau}=1$ verifying the claim.
\end{proof}

Let $\hat T_0(s)=(I-\hat R_0(s))^{-1}$; this is well-defined for $s\in\overline\H-\{0\}$.
Write $v^u(y)=v(y,u)$.   Then
\[
((I-\hat R(s))^{-1}v)(y,u)=(\hat T_0(s)v^u)(y),\quad s\in\overline\H-\{0\}.
\]
Also we obtain the renewal equation
\[
\hat T(s)=\hat U(s)(I-\hat R(s))^{-1}, \quad s\in\overline\H-\{0\}.
\]
By Proposition~\ref{prop-renew}, we obtain an analytic extension
\begin{align} \label{eq-rhohat}
\hat\rho(s)=\int_{\tilde Y}\hat U(s)(I-\hat R(s))^{-1}v\,w\,d\mu^\varphi,
\end{align}
defined on a neighbourhood of $\overline\H-\{0\}$.

\begin{rmk} \label{rmk-oprenew}
The operator renewal equation $\hat T(s)=\hat U(s)(I-\hat R(s))^{-1}$ has the desired effect of relating the Laplace transform of the 
transfer operators $T_t$ for the flow with the perturbed transfer operator
$\hat R_0(s)v=R(e^{-s\varphi}v)$ where $R$ is the transfer operator for the Poincar\'e map~$F$.   

In dynamical systems theory, there are two standard types of discrete time system that can be obtained from a continuous time system: Poincar\'e maps
such as $F$ and the time-$h$ map $f_h$ for fixed $h>0$.
In the probability theory literature, a standard technique after Kingman~\cite{Kingman63} is to 
consider discrete time ``skeletons'' $f_h$ and to pass to
the continuous time limit as $h\to0$.
However, for the properties studied in the current paper, the partially hyperbolic time-$h$ map is as difficult to study as the underlying
continuous time system.  In contrast, the uniformly expanding Poincar\'e map $F$ is much more tractable.

Hence, at least for certain situations in dynamical systems theory, and perhaps in probability theory too,
the renewal equation presented here seems a more useful approach than 
passing to discrete time skeletons.
\end{rmk}

The following elementary result is required in both Part~\ref{part-infinite}
and Part~\ref{part-finite}.

\begin{prop} \label{prop-m}
Let $m\ge1$.
Suppose that $v\in L^1(\tilde Y)$ and $w\in L^{\infty,m}(\tilde Y)$.
Then 
\[
\hat\rho_{v,w}(s)=
\sum_{j=1}^m \rho_{v,\partial_t^{j-1}w}(0)s^{-j}+s^{-m}\hat \rho_{v,\partial_t^mw}(s).
\]
\end{prop}

\begin{proof}
First note that 
$\rho_{v,w}$ is $m$-times differentiable and
${\rho_{v,w}}^{(j)}=\rho_{v,\partial_t^jw}$ for $j=0,\dots,m$.
By Taylor's Theorem, $\rho_{v,w}(t)=P_m(t)+H_m(t)$, where
\[
P_m(t)=\sum_{j=0}^{m-1} \frac{1}{j!}{\rho_{v,w}}\!^{(j)}(0)t^j, \quad
H_m(t)=\int_0^t g(t-\tau){\rho_{v,w}}\!^{(m)}(\tau)\,d\tau,
\quad g(t)=\frac{t^{m-1}}{(m-1)!}.
\]
Hence
$\hat\rho_{v,w}(s)=\sum_{j=0}^{m-1}\rho_{v,\partial_t^jw}(0)s^{-(j+1)}+
\hat H_m(s)$,
where $\hat H_m(s)=
\hat g(s)\hat \rho_{v,\partial_t^mw}(s)
=s^{-m}\hat \rho_{v,\partial_t^mw}(s)$.
\end{proof}

\section{Dolgopyat-type estimates}

In this section, we recall estimates of~\cite{Dolgopyat98b} extended to the nonuniformly hyperbolic setting~\cite{M07,M09}.  These are required to control
$(I-\hat R(s))^{-1}$ for $s=ib$, $b$ large.
The arguments need some modification here to allow for the possibility
that $\varphi\not\in L^1$.

We recall that the twisted transfer operators $\hat R_0(s):L^1(Y)\to L^1(Y)$ satisfy for $n\ge1$,
\begin{align} \label{eq-Rn}
(\hat R_0(s)^nv)(y)=\sum_{a\in\alpha_n}e^{g_n(y_a)}e^{-s\varphi_n(y_a)}v(y_a),
\end{align}
where $y_a$ denotes the unique preimage $y_a\in a\cap F^{-n}y$ and
$\varphi_n=\sum_{j=0}^{n-1}\varphi\circ F^j$.

\begin{lemma} \label{lem-Rb}
For every $\epsilon>0$, there 
exist constants $C\ge1$ and $\theta,\tau\in(0,1)$ such that for 
every $v\in F_\theta(Y)$, $b\in\R$, 
\begin{itemize}
\item[(a)] $|\hat R_0(ib)v|_\infty\le|v|_\infty$,
\item[(b)]
$|\hat R_0(ib)^nv|_\theta \le C\{(1+|b|^\epsilon\int_Y\varphi^\epsilon\,d\mu)|v|_\infty+\theta^n|v|_\theta\}$.
\item[(c)] $\|R^nv-\int_Yv\,d\mu\|_\theta\le C\tau^n\|v\|_\theta$.
\end{itemize}
\end{lemma}

\begin{proof}
Note that
$\hat R_0(s)^nv=R^n(e^{-s\varphi_n}v)$.
Since $|R|_\infty=1$, it follows that part (a) is valid.
Full branch Gibbs-Markov maps are mixing, so $R$ has no eigenvalues on the unit circle except for the simple eigenvalue at $1$.  
Part~(c) follows from this together with quasicompactness~\cite[Section~4.7]{Aaronson}.

It remains to prove (b).  Our argument improves~\cite{BHM05} where
it is assumed that $\varphi\in L^1(Y)$.
Let $y,y'\in Y$.   Then
\[
(\hat R_0(ib)^n(1_av ))(y)-(\hat R_0(ib)^n(1_av ))(y')=D_1+D_2+D_3,
\]
where 
\begin{align*}
D_1 & = (e^{g_n(y_a)}-e^{g_n(y'_a)})e^{ib\varphi_n(y_a)}v(y_a), \quad
D_2  = e^{g_n(y'_a)}(e^{ib\varphi_n(y_a)}-e^{ib\varphi_n(y'_a)})v(y_a), \\
D_3 & = e^{g_n(y'_a)}e^{ib\varphi_n(y'_a)}(v(y_a)-v(y'_a)).
\end{align*}
By the estimates~\eqref{eq-GM},
\begin{align*}
|D_1| & \le C_1 \mu(a)|v|_\infty d_\theta(y,y'), \quad
|D_3|  \le C_1 \mu(a)|v|_\theta d_\theta(y_a,y'_a) = C_1\theta^n\mu(a)|v|_\theta d_\theta(y,y').
\end{align*}
Summing over $a\in\alpha_n$, we obtain that the terms of type $D_1$ and $D_3$
contribute $C_1|v|_\infty$ and $C_1\theta^n|v|_\theta$ respectively
to $|\hat R_0(ib)^nv|_\theta$.

Next,
\[
|D_2|\le C_1 \mu(a)\sum_{j=0}^{n-1}|e^{ib\varphi(F^jy_a)}-e^{ib\varphi(F^jy'_a)}||v|_\infty.
\]
Recall that by assumption $1_a \varphi \in F_{\theta_0}(Y)$ for some $\theta_0\in (0,1)$.
 Using the inequality $|e^{ix}-1|\le\min\{2,|x|\}\le 2|x|^\epsilon$
for $x\in\R$, $\epsilon\in[0,1]$,
\[
|D_2|\le 2C_1 \mu(a)|b|^\epsilon\sum_{j=0}^{n-1}|1_{F^ja}\varphi|_{\theta_0}^\epsilon
d_{\theta_0}(F^jy_a,F^jy'_a)^\epsilon|v|_\infty.
\]
Let $\theta=\theta_0^\epsilon$.  Then 
$d_{\theta_0}(F^jy_a,F^jy'_a)^\epsilon
=d_\theta(F^jy_a,F^jy'_a)
=\theta^{n-j}d_\theta(y,y')$.  By (A1),
\[
|D_2|\le 2C_1C_2^\epsilon \mu(a)|b|^\epsilon\sum_{j=0}^{n-1}\inf(1_{F^ja}\varphi)^\epsilon
\theta^{n-j}d_\theta(y,y')|v|_\infty.
\]

Hence, summing over $a\in\alpha_n$, the $D_2$ terms contribute $2C_1C_2^\epsilon|b|^\epsilon|v|_\infty S$, where
\begin{align*}
	S & =\sum_{a\in\alpha_n}\mu(a)\sum_{j=0}^{n-1}\theta^{n-j}\inf(1_{F^ja}\varphi)^\epsilon
=\sum_{j=0}^{n-1}\sum_{d\in\alpha_{n-j}}\sum_{a\in\alpha_n:F^ja=d}\mu(a)\theta^{n-j}\inf(1_{F^ja}\varphi)^\epsilon
	\\ & 
=\sum_{j=0}^{n-1}\theta^{n-j}\sum_{d\in\alpha_{n-j}}
\inf(1_d\varphi)^\epsilon \sum_{a\in\alpha_n:F^ja=d}\mu(a)
=\sum_{j=0}^{n-1}\theta^{n-j}\sum_{d\in\alpha_{n-j}}
\inf(1_d\varphi)^\epsilon \mu(d)
\\ & 
\le \sum_{j=0}^{n-1}\theta^{n-j} \int_Y\varphi^\epsilon\,d\mu
\le \theta(1-\theta)^{-1} \int_Y\varphi^\epsilon\,d\mu.
\end{align*}
This completes the proof of part (b).
\end{proof}

For $b\in\R$, define $M_b:L^\infty(Y)\to L^\infty(Y)$,
$M_bv=e^{ib\varphi}v\circ F$.

\begin{defn}  \label{defn-approx} There are {\em approximate eigenfunctions} on a subset $Z\subset Y$
if there exist constants $A>0$ arbitrarily large, $\beta>0$ and $C\ge1$, and
sequences $|b_k|\to\infty$, $\psi_k\in[0,2\pi)$, 
$u_k\in F_\theta(Y)$ with
$|u_k|\equiv1$, such that setting $n_k=[\beta\ln|b_k|]$,
\[
|M_{b_k}^{n_k}u_k(y)-e^{i\psi_k}u_k(y)|\le C|b_k|^{-A},
\]
for all $y\in Z$ and all $k\ge1$.
\end{defn}

A subset $Z_0\subset Y$ is called a {\em finite subsystem} if $Z_0=\bigcap_{n\ge0} F^{-n}Z$ where $Z$ is a finite union of partition elements $a\in\alpha$.

\begin{prop} \label{prop-A2}
There exists a finite subsystem $Z_0$ such that there are no approximate eigenfunctions on $Z_0$.
\end{prop}

\begin{proof}
By (A2), we can fix two periodic orbits with periods $\tau_1$ and $\tau_2$
such that $\tau_1/\tau_2$ is Diophantine.
 Let $Z$ be the union of the partition elements $a\in\alpha$ intersected by the periodic orbits
and define $Z_0 =\bigcap_{n\ge0}F^{-n}Z$.
It follows from~\cite[Section~13]{Dolgopyat98b} 
that there are no approximate eigenfunctions on $Z_0$.
\end{proof}

\begin{lemma} \label{lem-approx}
There exists $A>0$ and $C\ge1$ such that
$\|(I-\hat R(ib))^{-1}\|_\theta\le C|b|^A$ for all $b\in\R$ with $|b|\ge1$.
\end{lemma}

\begin{proof}  
By~\eqref{eq-Rtilde}, it suffices to
prove this with $\hat R(ib)$ replaced by $R_0(ib):F_\theta(Y)\to F_\theta(Y)$. 
By Proposition~\ref{prop-A2}, there is a finite subsystem on which there are no approximate eigenfunctions.
The estimate for $(I-\hat R_0)^{-1}$ follows from this
by exactly the argument used in~\cite[Lemmas~3.12 and~3.13]{M07}.
Lemma~\ref{lem-Rb} plays the role of~\cite[Proposition~3.7]{M07}.
\end{proof}

\begin{rmk} \label{rmk-norm}
We shall consider various families of linear operators acting on various function spaces.   
Throughout, $\|\,\|_\theta$ denotes the operator norm on
either $F_\theta(Y)$ or $F_\theta(\tilde Y)$.
Similarly, $\|\,\|_\infty$ denotes the operator norm on
either $L^\infty(Y)$ or $L^\infty(\tilde Y)$.
Whether the space is $Y$ or $\tilde Y$ should be clear from the context.
\end{rmk}

\part{Infinite measure systems}
\label{part-infinite}

In this part of the paper, we prove our main results in the infinite measure context.
Throughout, we assume the setup from Section~\ref{sec-main}, so $F:Y\to Y$ is a full branch Gibbs-Markov map and $\varphi:Y\to\R^+$ is a roof function satisfying assumptions (A1) and~(A2).  In addition, we make the
standing assumption throughout this part of the paper that $\varphi$ is nonintegrable and $\mu(\varphi>t)\sim \ell(t)t^{-\beta}$ where $\beta\in(\frac12,1]$
and $\ell(t)$ is slowly varying.

Section~\ref{sec-funct} contains various operator-theoretic estimates.
Our result on first order asymptotics (mixing) stated in Theorem~\ref{thm-infinite}(a) is proved in 
Section~\ref{sec-mixing}.
Our result on second order asymptotics and rates of mixing, Theorem~\ref{thm-infinite}(b), is proved in Section~\ref{sec-rates}.

\section{Functional analytic estimates}
\label{sec-funct}

In this section, we carry out various operator-theoretic estimates.
Most of these are fairly straightforward generalisations of the estimates
in~\cite{MT12} which built upon~\cite{AaronsonDenker01,GarsiaLamperti62}.
However, the estimates
in Lemma~\ref{lem-R} are considerably more complicated than in the discrete time case.

\subsection{Estimates for $\tilde R$}
\label{sec-R}

In this subsection, we prove a key technical estimate that we have not seen 
elsewhere in the literature (though Lemma~\ref{lem-R} has a similar flavour 
to estimates in~\cite{Terhesiu-app}).

We have the estimate $\mu(E>t)=O(\ell(t)t^{-\beta})$
for various functions $E:\alpha\to\R$ related to $\varphi$
including the locally constant functions 
$E(a)=\inf_a\varphi$ and hence 
$E(a)=|1_a\varphi|_\infty$ and
$E(a)=|1_a\varphi|_{\theta_0}$ by condition (A1).

We use the following resummation argument extensively.

\begin{prop} \label{prop-resum}
Let $\omega,E:\alpha\to\R$ be such that $\mu(E>t)=O(\ell(t)t^{-\beta})$,
$\omega$ is a bounded function, and $\omega(a)\le GE(a)$.
Then $\sum_{a\in\alpha} \mu(a) \omega(a) \ll \ell(1/G) G^\beta$.
\end{prop}

\begin{proof}
For $L\ge1$, write
\begin{align*}
\sum_{a\in\alpha} \mu(a) \omega(a) 
& \le G\sum_{a:E(a)\le L} \mu(a)E(a)+
\sum_{a:E(a)>L} \mu(a)|\omega|_\infty
 =GK+O(\ell(L) L^{-\beta}),
\end{align*}
where 
\begin{align*}
K & = \sum_{a:E(a)\le L} \mu(a)E(a)  \le
\sum_{j=1}^L \sum_{a:E(a)\in(j-1,j]} \mu(a)j 
=\sum_{j=1}^L \sum_{a:E(a)> j-1} \mu(a)j- 
\sum_{j=0}^L \sum_{a:E(a)> j} \mu(a)j 
\\ &
 = 
\sum_{j=0}^{L-1} \sum_{a:E(a)> j} \mu(a)(j+1)- 
\sum_{j=0}^L \sum_{a:E(a)> j} \mu(a)j 
 \le \sum_{j=0}^{L-1} \sum_{a:E(a)> j} \mu(a)
\\ & = \sum_{j=0}^{L-1} \mu(E>j) \ll \ell(L)  L^{1-\beta}.
\end{align*}
Taking $L\approx 1/G$ yields the result.
\end{proof}

\begin{lemma} \label{lem-R}   
Let $\epsilon\in(0,\beta)$.  There exists $\theta\in(0,1)$, $C>0$ such that 
\[
\|\tilde R(ib_1)-\tilde R(ib_2)\|_\theta\le  C\bigl\{\ell(|b_1-b_2|^{-1})|b_1-b_2|^\beta+
\ell(|b_1-b_2|^{\epsilon/\beta-1}|b_2|^{-\epsilon/\beta})|b_2|^\epsilon|b_1-b_2|^{\beta-\epsilon}\bigr\}.
\]
\end{lemma}

\begin{proof}   
By~\eqref{eq-Rtilde}, it suffices to prove the result for $\hat R_0(ib)$.
We show that
\begin{align*}
|\hat R_0(ib_1)-\hat R_0(ib_2)v|_{\theta} & \le C\bigl\{\ell(|b_1-b_2|^{-1})|b_1-b_2|^\beta \\ & \qquad \qquad +\ell(|b_1-b_2|^{\epsilon/\beta-1}|b_2|^{-\epsilon/\beta})|b_2|^\epsilon|b_1-b_2|^{\beta-\epsilon}\bigr\}\|v\|_\theta.
\end{align*}
A simpler argument which we omit shows that 
$|(\hat R_0(ib_1)-\hat R_0(ib_2))v|_\infty \le C\ell(|b_1-b_2|^{-1})|b_1-b_2|^\beta|v|_\infty$ and the
result follows.

The structure of the calculation begins as 
in~\cite[Theorem~2.4]{AaronsonDenker01}.
Let $DR=\hat R_0(ib_1)-\hat R_0(ib_2)$ and $\Delta(y)=
e^{ib_1\varphi(y)}-e^{ib_2\varphi(y)}$.
Recalling formula~\eqref{eq-Rn}, we have
$(DRv)(y)=\sum_{a\in\alpha}\Delta(y_a)e^{g(y_a)}v(y_a)$ and so
\begin{align*}
& (DRv)(y)- (DRv)(y')
 =\sum_{a\in\alpha} \Delta(y_a)e^{g(y_a)}v(y_a)
-\Delta(y'_a)e^{g(y'_a)}v(y'_a) \\
& \qquad =\sum_{a\in\alpha} \Delta(y_a)[e^{g(y_a)}v(y_a)-e^{g(y'_a)}v(y'_a)]
+[\Delta(y_a)-\Delta(y'_a)]e^{g(y'_a)}v(y'_a).
\end{align*}
By the estimates~\eqref{eq-GM},
\begin{align*}
 |(DRv)(y)- (DRv)(y')| &  \le C_1|v|_\theta J_1+ C_1|v|_\infty J_2,
\end{align*}
where
\[
J_1= \sum_{a\in\alpha} \mu(a)|\Delta(y_a)|d_\theta(y,y'), \qquad
J_2= \sum_{a\in\alpha}\mu(a)|\Delta(y_a)-\Delta(y'_a)|.
\]

Now $\Delta$ is bounded and also
$|\Delta(y_a)|\le |b_1-b_2||1_a\varphi|_\infty =GE(a)$
where $G=|b_1-b_2|$ and $\mu(E>t)=O(\ell(t)t^{-\beta})$.
By Proposition~\ref{prop-resum}, $J_1\ll \ell(|b_1-b_2|^{-1})|b_1-b_2|^\beta d_\theta(y,y')$.

Next,
 \[
  |\Delta(y_a)-\Delta(y'_a)|   \le  
|e^{i(b_1-b_2)\varphi(y_a)}-e^{i(b_1-b_2)\varphi(y'_a)}|
+|e^{i(b_1-b_2)\varphi(y'_a)}-1||e^{ib_2\varphi(y_a)}-e^{ib_2\varphi(y'_a)}|,
\]
so $J_2\le J_2'+J_2''$ where
\begin{align*}
J_2'& = \sum_{a\in\alpha}\mu(a)
|e^{i(b_1-b_2)\varphi(y_a)}-e^{i(b_1-b_2)\varphi(y'_a)}|,
\\
J_2'' & = \sum_{a\in\alpha}\mu(a)
|e^{i(b_1-b_2)\varphi(y'_a)}-1||e^{ib_2\varphi(y_a)}-e^{ib_2\varphi(y'_a)}|,
\end{align*}
Now
$ J_2'= \sum_{a\in\alpha}\mu(a)\omega(a)$ where $\omega$ is bounded and
$|\omega(a)|
\le GE(a)$ where $G=|b_1-b_2|d_{\theta_0}(y,y')$ and $E(a)=|1_a\varphi|_{\theta_0}$.
It follows from (A1) that $\mu(E>t)=O(\ell(t)t^{-\beta})$.
By Proposition~\ref{prop-resum}, 
$J_2'\ll \ell((|b_1-b_2|d_{\theta_0}(y,y'))^{-1})
(|b_1-b_2|d_{\theta_0}(y,y'))^\beta$.
By Potter's bounds (see for instance~\cite{BGT}),
$J_2'\ll \ell(|b_1-b_2|^{-1})
|b_1-b_2|^\beta d_{\theta_0}(y,y')^{\beta-\epsilon}$.
Choosing $\theta\ge\theta_0^{\beta-\epsilon}$, we obtain
$J_2'\ll \ell(|b_1-b_2|)^{-1}|b_1-b_2|^\beta d_{\theta}(y,y')$.

Finally, write $J_2''=\sum_{a\in\alpha}\mu(a)\omega(a)$ where
$\omega$ is bounded.
The inequality $|e^{ix}-1|\le 2|x|^\delta$ holds for all $x\ge0$, $\delta\in[0,1]$, so
\[
|\omega(a)|\le 4\bigl\{|b_1-b_2||1_a\varphi|_\infty\bigr\}^{1-\epsilon/\beta}
\bigl\{|b_2||1_a\varphi|_{\theta_0} d_{\theta_0}(y,y')\bigr\}^{\epsilon/\beta}
=4GE(a),
\]
where $G=|b_1-b_2|^{1-\epsilon/\beta}|b_2|^{\epsilon/\beta}
d_{{\theta_0}^{\epsilon/\beta}}(y,y')$
and $E(a)=|1_a\varphi|_\infty^{1-\epsilon/\beta}
|1_a\varphi|_{\theta_0}^{\epsilon/\beta}$.
Again $\mu(E>t)=O(\ell(t)t^{-\beta})$ so it follows from Proposition~\ref{prop-resum} and Potter's bounds that
\begin{align*}
J_2''\ll \ell(G^{-1})G^\beta & =
\ell(|b_1-b_2|^{\epsilon/\beta-1}|b_2|^{-\epsilon/\beta}
d_{{\theta_0}^{\epsilon/\beta}}(y,y')^{-1})|b_1-b_2|^{\beta-\epsilon}|b_2|^{\epsilon}
d_{{\theta_0}^{\epsilon}}(y,y') \\
& \ll
\ell(|b_1-b_2|^{\epsilon/\beta-1}|b_2|^{-\epsilon/\beta})
|b_1-b_2|^{\beta-\epsilon}|b_2|^{\epsilon}
d_{{\theta_0}^{\epsilon/2}}(y,y').
\end{align*}
Choosing $\theta\ge \theta_0^{\epsilon/2}$, we obtain
\[
J_2''\ll
\ell(|b_1-b_2|^{\epsilon/\beta-1}|b_2|^{-\epsilon/\beta})
|b_1-b_2|^{\beta-\epsilon}|b_2|^{\epsilon}
d_{\theta}(y,y'),
\]
completing the proof.
\end{proof}

\begin{rmk} \label{rmk-Dphi}
If $\sup_{a\in\alpha}|1_a\varphi|_\theta<\infty$, then the proof simplifies
considerably~\cite{AaronsonDenker01} and we obtain that 
$\|\hat R(ib_1)-\hat R(ib_2)\|_\theta\le C\ell(|b_1-b_2|^{-1})|b_1-b_2|^\beta$.
However, such a condition is too restrictive for the inducing step in Subsection~\ref{sec-GM} and in particular is not satisfied for Example~\ref{ex-PM}.
\end{rmk}

\begin{prop} \label{prop-R_s}
$\|\hat R(s)-\hat R(0)\|_\infty\ll \ell(1/|s|)|s|^\beta$ for all $s\in\overline\H$.
\end{prop}

\begin{proof} Again it suffices to prove the result for the operators
$\hat R_0(s)$.
It follows from the proof of Lemma~\ref{lem-R}
that $\|\hat R_0(ib)-\hat R_0(0)\|_\infty\ll \ell(\/|b|)|b|^\beta$.
An identical argument shows that
$\|\hat R_0(ib+h)-\hat R_0(ib)\|_\infty\ll \ell(\/h)h^\beta$ for all $b\in\R$ and $h>0$
(the restriction to $h>0$ guarantees that the function $1-e^{-h\varphi}$
is bounded).
\end{proof}

\subsection{Estimates for $(I-\hat R)^{-1}$}

Let $c_\beta=i\int_0^\infty e^{-i\sigma}\sigma^{-\beta}\,d\sigma$.

\begin{lemma} \label{lem-zero}   
Viewing $(I-\hat R(s))^{-1}$ as a family of linear operators on
$F_\theta(\tilde Y)$, 
\[
(I-\hat R(ib))^{-1}v\sim c_\beta^{-1}\ell(1/b)^{-1}b^{-\beta}\int_Y
v(y,\cdot)\,d\mu(y),
\quad\text{as $b\to0^+$}.
\]
\end{lemma}

\begin{proof}   
Since $\hat R(ib)v(y,u)=(\hat R_0(ib)v^u)(y)$, where $v^u(y)=v(y,u)$,
it suffices to prove that
$((I-\hat R_0(ib))^{-1}v)\sim c_\beta^{-1}\ell(1/b)^{-1}b^{-\beta}\int_{Y}
v\,d\mu$ as $b\to0^+$, for all $v\in F_\theta(Y)$.

By Lemma~\ref{lem-R}, the map $b\mapsto \hat R_0(ib)$ is continuous.
By Proposition~\ref{prop-H2}(a), $\hat R_0(0)$ has $1$ as a simple eigenvalue, so
there exists $\delta>0$ and a continuous family $\lambda(ib)$ of simple
eigenvalues of $\hat R_0(ib)$ for $b\in (-\delta,\delta)$ with
$\lambda(0)=1$.  Let $P(ib)$ denote the
corresponding family of spectral projections  with complementary
projections $Q(ib)=I-P(ib)$. Also, let
$v(ib)$ denote the corresponding family of
eigenfunctions normalized so that $\int_Y v(ib)\,d\mu=1$.
In particular, $v(0)\equiv1$ and $P(0)w=\int_Yw\,d\mu$ for all $w\in L^1(Y)$.

Following Gou\"ezel~\cite{Gouezel10b} (a simplification of~\cite{AaronsonDenker01}), we write
\begin{align*}
\lambda(ib) =\int_Y\lambda(ib)v(ib)\,d\mu=\int_Y R_0(e^{-ib\varphi}v(ib))\,d\mu=\int_Y
e^{-ib\varphi}\, d\mu+V(ib),
\end{align*}
 where $V(ib)=\int_Y (\hat R_0(ib)-\hat R_0(0))(v(ib)-v(0))\,d\mu$. 

By the argument in~\cite{GarsiaLamperti62} (see also~\cite{AaronsonDenker01,MT12}),
 $1-\int_Y e^{-ib\varphi}\,d\mu\sim c_\beta\ell(1/b)b^{\beta}$ as $b\to0^+$.
By Lemma~\ref{lem-R}, $\hat R_0$ and hence $v$ are $C^{\beta-2\epsilon}$ (say), so 
$|V(ib)|=O(b^{2(\beta-2\epsilon)})$.  Hence,
\[
1-\lambda(ib)\sim c_\beta\ell(1/b)b^{\beta}\enspace\text{as $b\to0^+$}.
\]
Next, for $b\in (-\delta,\delta)$,
\begin{align*}
(I-\hat R_0(ib))^{-1} & =(1-\lambda(ib))^{-1}P(0)- (1-\lambda(ib))^{-1}(P(ib)-P(0))\\ 
& \qquad\qquad\qquad 
 \qquad \qquad\qquad\qquad 
+(I-\hat R_0(ib))^{-1}Q(ib).
\end{align*}
By Proposition~\ref{prop-H2}(a), $\|(I-\hat R_0(ib))^{-1}Q(ib)\|_\theta=O(1)$.
By Lemma~\ref{lem-R}, $\|P(ib)-P(0)\|_\theta\ll b^{\beta-2\epsilon}$.  Hence
\begin{align*}
(I-\hat R_0(ib))^{-1}=(1-\lambda(ib))^{-1}(P(0)+o(1))+O(1)\sim
c_\beta^{-1}\ell(1/b)^{-1}b^{-\beta}P(0)
\enspace\text{as $b\to0^+$}, 
\end{align*}
as required.
\end{proof}

\begin{prop} \label{prop-zero_s}
$|(I-\hat R(s))^{-1}v|_\infty\ll\ell(1/|s|)^{-1}|s|^{-\beta}\|v\|_\theta$
for all $s\in\overline\H$ with $|s|$ sufficiently small,
 and all $v\in F_\theta(\tilde Y)$.
\end{prop}

\begin{proof}
As in the proof of Lemma~\ref{lem-zero}, for $s\in\overline\H$ close to zero, we have the decomposition
\[
(I-\hat R_0(s))^{-1}v = (1-\lambda(s))^{-1}P(0)v+(1-\lambda(s))^{-1}(P(s)-P(0))v
+(I-\hat R_0(s))^{-1}Q(s)v,
\]
where the last term is bounded.   
(As before, $\lambda(s)$ is the leading eigenvalue for $\hat R_0(s)$ with spectral 
projection $P(s)$ and $Q(s)=I-P(s)$.)   
By Proposition~\ref{prop-R_s}, $P(s)-P(0)=o(1)$ as $s\to0$,
so it remains to estimate $(1-\lambda(s))^{-1}$.  Again, write
\[
\lambda(s)=\int_Y e^{-s\varphi}\,d\mu+V(s), \quad
V(s)=\int_Y(\hat R_0(s)-\hat R_0(0))(v(s)-v(0))\,d\mu,
\]
where $v(s)$ are the normalised eigenfunctions for $\lambda(s)$.
By Proposition~\ref{prop-R_s}, 
$\|\hat R_0(s)-\hat R_0(0)\|_{L^\infty(Y)}\ll \ell(1/|s|)|s|^\beta$
and this estimate is inherited by
$v(s)$ so that $|V(s)|_\infty \ll |s|^{2\beta-\epsilon}$.

Let $G(x)=\mu(\varphi<x)$.
Following and using the proof of~\cite[Lemma~2.4]{MT13},
\begin{align*}
1-\int_Y e^{-s\varphi}\,d\mu &
=\int_0^\infty (1-e^{-sx})\,dG(x)
 =s\int_0^\infty e^{-sx}(1-G(x))\,dx
=\ell(1/|s|)s^\beta J(s)
\end{align*}
where $J(s)\to \Gamma(1-\beta)$ as $s\to0$.
Hence $1-\lambda(s)\sim c \ell(1/|s|)s^{\beta}$ with $c=\Gamma(1-\beta)$.
\end{proof}

\begin{lemma} \label{lem-middle}
For $b\in(0,1]$, $\|(I-\hat R(ib))^{-1}\|_\theta\ll \ell(1/b)^{-1}b^{-\beta}$.
\end{lemma}

\begin{proof}
The proof of Lemma~\ref{lem-zero} shows that there exists $\delta>0$ so that
the result holds for $b\in(-\delta,\delta)$.
Proposition~\ref{prop-H2}(b) guarantees that 
$\|(I-\hat R(ib))^{-1}\|_\theta=O(1)$ for $b\in(\delta,1]$.
\end{proof}

\subsection{Estimates for $\hat U$}
\label{sec-U}

In this subsection, we obtain estimates for the family of operators
$\hat U(s)$ that appeared in the renewal equation.

\begin{lemma} \label{lem-U}
\begin{itemize}
\item[(a)] $\hat U(0)v(y,u)=\int_0^u v(y,\tau)\,d\tau+\int_u^1 (\tilde Rv)(y,\tau)\,d\tau$.
\item[(b)]
The family of linear maps $\hat U(ib):L^\infty(\tilde Y)\to L^1(\tilde Y)$,
$b\in\R$, is uniformly bounded (indeed $\|\hat U(ib)\|\le 2$ for all $b\in\R$)
and $\|\hat{U}(i(b+h))-\hat{U}(ib)\|\ll
\ell(1/h) h^{\beta}$ for all $h>0$.
\end{itemize}
\end{lemma}

\begin{proof}   
(a) Write
$\int_{\tilde{Y}}\hat{U}(0)v\, w\,d\tilde\mu  =\int_{\tilde{Y}}\int_0^1 U_tv\, w\,dt\,d\tilde\mu +\int_{\tilde{Y}}\int_1^\infty U_tv\, w\,dt\,d\tilde\mu$.
Using Proposition~\ref{prop-U},
\begin{align*}
\int_{\tilde{Y}}\int_0^1 U_tv\, w\,dt\,d\tilde\mu 
& =\int_{\tilde Y}\int_0^1 1_{\{t,1)}(u) v(y,u-t)w(y,u)\,dt\,d\tilde\mu  \\
& =\int_{\tilde Y}  \Bigl\{\int_0^u v(y,u-t)\,dt\Bigr\}w(y,u)\,d\tilde\mu   
 = \int_{\tilde Y} \Bigl\{\int_0^u v(y,\tau)\,d\tau\Bigr\} w(y,u)\,d\tilde\mu,
\end{align*}
and
\begin{align*}
& \int_{\tilde{Y}}\int_1^\infty U_tv\, w\,dt\,d\tilde\mu
 =\int_{\tilde{Y}}\int_1^\infty 1_{\{t<\tilde{\varphi}(y)<t+1-u\}}v(y,u-t+\tilde\varphi(y))w(Fy,u)\,dt\,d\tilde\mu\\ 
& \quad =\int_{\tilde Y}\Bigl\{\int_{\tilde{\varphi}-1+u}^{\tilde{\varphi}}v(y,u-t+\tilde\varphi(y))\,dt\Bigr\}w(Fy,u)\,d\tilde\mu\\
& \quad =\int_{\tilde Y}\Bigl\{\int_u^1 v(y,\tau)\,d\tau\Bigr\}w\circ\tilde F(y,u)\,d\tilde\mu
 = \int_{\tilde Y} \Bigl\{\int_u^1 (\tilde Rv)(y,\tau)\,d\tau\Bigr\} w(y,u)\,d\tilde\mu.
\end{align*}
This completes the proof of part (a).

\noindent(b)  By  Proposition~\ref{prop-U},
$|U_tv|_1\le |v|_\infty$ for $0<t<1$,
and
$|U_tv|_1\le \tilde\mu\{(y,u):t<\tilde\varphi(y,u)<t+1-u\}|v|_\infty
\le \mu\{t<\varphi<t+1\}|v|_\infty$.
Hence,
\begin{align*}
& |\hat{U}(ib)v|_1  =\Bigl|\int_0^\infty e^{-ibt}U_tv\,dt\Bigr|_1
\le \int_0^\infty |U_tv|_1\,dt 
\le \Bigl(1+\int_1^\infty \mu(t<\varphi<t+1)\,dt\Bigr)|v|_\infty
\\ & \quad
= \Bigl(1+\int_1^\infty (\mu(\varphi>t)-\mu(\varphi>t+1))\,dt\Bigr)|v|_\infty
= \Bigl(1+\int_1^2 \mu(\varphi>t))\,dt\Bigr)|v|_\infty\le 2|v|_\infty.
\end{align*}
Also,
\begin{align*}
|(\hat{U}(i(b+h))-\hat{U}(ib))v|_1 & =\Bigl|\int_0^\infty e^{-ibt}(e^{-iht}-1)U_tv\,dt\Bigr|_1 \\
& \le 
 |v|_\infty\Bigl(h+\int_1^\infty|e^{-iht}-1|\mu(t<\varphi<t+1)\,dt\Bigr).
\end{align*}
But
\begin{align*}
\int_1^\infty|e^{-iht}-1|\mu(t<\varphi<t+1) \,dt
& \leq h\int_1^L t\mu(t<\varphi<t+1) \,dt 
\\ & \qquad+2\int_L^\infty \mu(t<\varphi<t+1) \,dt.
\end{align*}
Also, note that
\begin{align*}
& \int_1^L t\mu(t<\varphi<t+1) \,dt=\int_1^Lt\mu(\varphi>t) \,dt-\int_1^Lt\mu(\varphi>t+1) \,dt
\\ & \quad =\int_1^L t\mu(\varphi>t) \,dt -\int_2^{L+1}(t-1)\mu(\varphi>t) \,dt 
\le 1+\int_2^L\mu(\varphi>t) \,dt 
 \ll \ell(L)L^{1-\beta},
\end{align*}
by Karamata's Theorem (see for instance~\cite{BGT}).  Similarly,
\begin{align*}
\int_L^\infty \mu(t<\varphi<t+1) \,dt & =\int_L^\infty \mu(\varphi>t) \,dt-\int_{L+1}^\infty \mu(\varphi>t) \,dt
\\ & =\int_{L}^{L+1}\mu(\varphi>t)\,dt\le \mu(\varphi>L)=\ell(L) L^{-\beta}.
\end{align*}
Putting these together, $|\hat{U}(i(b+h))-\hat{U}(ib)|_1\ll h+h\ell(L)L^{1-\beta}+\ell(L)L^{-\beta}$. The conclusion follows by taking $L=1/h$.
\end{proof}

\begin{rmk} \label{rmk-U}
It is immediate from the proof that
$\|\hat U(s)\|_{L^\infty(\tilde Y)\mapsto L^1(\tilde Y)}\le2$
for all $s\in\overline\H$.
\end{rmk}

\subsection{Estimates for $\hat T$}
\label{sec-T}

In this subsection, we combine our estimates from the previous subsections to
estimate $\hat T=\hat U(I-\hat R)^{-1}$.    
Recall that $c_\beta=i\int_0^\infty e^{-i\sigma}\sigma^{-\beta}\,d\sigma$.

\begin{cor} \label{cor-T}
There exists $A>0$ such that for all $\epsilon\in(0,\beta/2)$,
the family of linear maps $\hat T(ib):F_\theta(\tilde Y)\to L^1(\tilde Y)$
satisfies the following properties.
\begin{itemize}
\item[(a)]  $\hat T(ib)v\sim c_\beta^{-1}\ell(1/b)^{-1}b^{-\beta}\int_{\tilde Y}v\,d\tilde\mu$ as $b\to0^+$.
\item[(b)]  $\|\hat T(ib)\|\ll \begin{cases} \ell(1/b)^{-1}b^{-\beta}, & 0<b<1 \\
b^A, & b\ge1 \end{cases}$.
\item[(c)] For $0<b_1<b_2<1$,
\begin{align*}
\|\hat T(ib_1)-\hat T(ib_2)\| \ll & 
\ell(1/b_1)^{-2}b_1^{-2\beta}\bigl\{\ell(|b_1-b_2|^{-1})|b_1-b_2|^\beta\\
&\qquad\qquad +\ell(|b_1-b_2|^{\epsilon/\beta-1}|b_2|^{-\epsilon/\beta})|b_2|^\epsilon|b_1-b_2|^{\beta-\epsilon}\bigr\}.
\end{align*}
For $1<b_1<b_2<b_1+1$,
$\|\hat T(ib_1)-\hat T(ib_2)\| \ll b_2^A |b_1-b_2|^{\beta-2\epsilon}$.
\end{itemize}
\end{cor}

\begin{proof}
(a) By continuity of $\hat U$ (Lemma~\ref{lem-U}(b)), we have
\begin{align*}
\lim_{b\to0^+}\ell(1/b)b^{\beta}\hat T(ib)v & =
\lim_{b\to0^+}\hat U(ib)\ell(1/b)b^{\beta}(I-\hat R(ib))^{-1}v \\ & =
\hat U(0)\lim_{b\to0^+}\ell(1/b)b^{\beta}(I-\hat R(ib))^{-1}v.
\end{align*}
By Lemma~\ref{lem-U}(a),
\begin{align*}
 \Bigl(\lim_{b\to0^+}\ell(1/b)b^{\beta}\hat T(ib)v\Bigr)(y,u)  & =
\int_0^u\Bigl(\lim_{b\to0^+}\ell(1/b)b^{\beta}(I-\hat R(ib))^{-1}v\Bigr)(y,\tau)\,d\tau
\\ & \qquad  +\int_u^1\tilde R\Bigl(\lim_{b\to0^+}\ell(1/b)b^{\beta}(I-\hat R(ib))^{-1}v\Bigr)(y,\tau)\,d\tau.
\end{align*}
Hence, by Lemma~\ref{lem-zero},
\begin{align*}
 \Bigl(\lim_{b\to0^+}\ell(1/b)b^{\beta}\hat T(ib)v\Bigr)(y,u)  & =
c_\beta^{-1}\int_0^u\int_Yv(y,\tau))\,d\mu\,d\tau
+c_\beta^{-1} \int_u^1\Bigl(\tilde R\int_Yv(y,\tau)\,d\mu\Bigr)\,d\tau \\ &
= c_\beta^{-1}\int_0^1\int_Yv(y,\tau)\,d\tilde\mu,
= c_\beta^{-1}\int_{\tilde Y}v\,d\tilde\mu,
\end{align*}
where we have used also the fact that $\tilde R$ fixes functions that are independent of $y$.  This proves part (a).

\noindent(b)  This follows from Lemma~\ref{lem-middle} and
Lemma~\ref{lem-U}(b) for $0< b<1$ and from
Lemma~\ref{lem-approx} and Lemma~\ref{lem-U}(b) for $b\ge1$.

\noindent(c)
We give the details for $0<b_1\le b_2\le1$. 
Recall that $\hat T(ib)=\hat U(ib)S(ib)$ where 
$S(ib)=(I-\hat R(ib))^{-1}$.
By the resolvent inequality,
\begin{align*}
 & \|S(ib_1)-S(ib_2)\|_\theta   \le 
 \|S(ib_1)\|_\theta\|\hat R(ib_1)-\hat R(ib_2)\|_\theta
 \|S(ib_2)\|_\theta \\ & 
\qquad \ll \ell(1/b_1)^{-2}b_1^{-2\beta}\bigl\{\ell(|b_1-b_2|^{-1})|b_1-b_2|^\beta
+\ell(|b_1-b_2|^{\epsilon/\beta-1}b_2^{-\epsilon/\beta})b_2^\epsilon|b_1-b_2|^{\beta-\epsilon}\bigr\} ,
\end{align*}
using Lemma~\ref{lem-R} and Lemma~\ref{lem-middle}.
Combining this with Lemma~\ref{lem-U}(b),
\begin{align*}
& \|\hat T(ib_1)-\hat T(ib_2)\|_{F_\theta\mapsto L^1}
  \le \|\hat U(ib_1)-\hat U(ib_2)\|_{L^\infty\mapsto L^1}\|S(ib_1)\|_\theta
\\ & \qquad\qquad\qquad
 \qquad\qquad\qquad
 + \|\hat U(ib_1)\|_{L^\infty\mapsto L^1}\|S(ib_1)-S(ib_2)\|_\theta \\
& \ll \ell(|b_1-b_2|^{-1})|b_1-b_2|^\beta \ell(1/b_1)^{-1} b_1^{-\beta}\\
&\qquad + \ell(1/b_1)^{-2}b_1^{-2\beta}\{\ell(|b_1-b_2|^{-1})|b_1-b_2|^\beta
+\ell(|b_1-b_2|^{\epsilon/\beta-1}b_2^{-\epsilon/\beta})b_2^\epsilon|b_1-b_2|^{\beta-\epsilon}\}.
\end{align*}

The argument for $1<b_1<b_2<b_1+1$ is similar but simpler because we establish 
a cruder estimate.   The slowly varying functions are taken care of by 
$\epsilon's$ in the exponents, and by increasing the value of $A$.
\end{proof}

\section{First order asymptotics (mixing)}
\label{sec-mixing}

In this section, we prove Theorem~\ref{thm-infinite}(a).

\subsection{The case $\beta\in(\frac12,1)$}
\label{sec-half}

\begin{prop} \label{prop-0cx}
There exists $\delta>0$ such that
$|\hat\rho(s)|\ll \ell(1/|s|)^{-1}|s|^{-\beta}$ for $s\in\overline\H$ satisfying $|s|\le\delta$.
\end{prop}

\begin{proof}  
This follows from Theorem~\ref{thm-renew}, 
Proposition~\ref{prop-zero_s} and Remark~\ref{rmk-U}.~
\end{proof}

\begin{prop} \label{prop-inverse}
$\rho(t)=\frac{1}{2\pi}\int_{-\infty}^\infty e^{ibt}\hat\rho(ib)\,db
= \frac{1}{\pi}\Re\int_0^\infty e^{ibt}\hat\rho(ib)\,db$.
\end{prop}

\begin{proof}
Since $\hat\rho$ is analytic on $\H$, we can invert the Laplace transform by computing
$\rho(t)=\frac{1}{2\pi i}\int_{\Gamma_1}e^{st}\hat\rho(s)\,ds$ where 
$\Gamma_1$ is the contour $\Re s=1$ traversed upwards.
As noted in~\eqref{eq-rhohat},
$\hat\rho(ib)$ is well-defined and continuous on the imaginary axis except
for the singularity at zero, so by Cauchy's Theorem
we can move the contour to
a contour $\Gamma_0$ which consists of the segments of the imaginary
axis $\{s=ib:-\infty<b<-\delta\}\cup\{s=ib:\delta<b<\infty\}$
together with a semicircle $\Gamma_\delta=\{s=\delta e^{i\psi}:-\pi/2<\psi<\pi/2\}$
where $\delta>0$ is arbitrarily small.

Let $\epsilon\in(0,1-\beta)$.  It follows from Proposition~\ref{prop-0cx} that
$\int_{\Gamma_\delta} e^{st}\hat\rho(s)\,ds=O(e^{\delta t}\delta^{1-\beta-\epsilon})$ and
$\int_{-\delta}^\delta e^{ibt}\hat\rho(ib)\,db=O(\delta^{1-\beta-\epsilon})$.
Letting $\delta\to0$, we obtain $\rho(t)=\frac{1}{2\pi i}\int_{\Gamma_1}e^{st}\hat\rho(s)\,ds=
\frac{1}{2\pi i}\int_{\Gamma_0}e^{st}\hat\rho(s)\,ds=
\frac{1}{2\pi}\int_{-\infty}^\infty e^{ibt}\hat\rho(ib)\,db$ as required.
\end{proof}

Recall that $c_\beta=i\int_0^\infty e^{-i\sigma}\sigma^{-\beta}\,d\sigma$.

\begin{prop} \label{prop-near0}
For any $a>0$,
\[
\lim_{t\to\infty}\ell(t) t^{1-\beta}\int_0^{a/t}  e^{ibt}\hat\rho_{v,w}(ib)\,db
=c_\beta^{-1}\int_0^a e^{i\sigma}\sigma^{-\beta}\,d\sigma
\int_{\tilde Y}v\,d\mu \int_{\tilde Y}w\,d\mu.
\]
\end{prop}

\begin{proof}
It follows from Proposition~\ref{prop-renew} and Corollary~\ref{cor-T}(a)
that $\hat\rho(ib)= c_\beta^{-1} \ell(1/b)^{-1}b^{-\beta}h(b)
\int_{\tilde Y}v\,d\tilde\mu
\int_{\tilde Y}w\,d\tilde\mu$ where $\lim_{b\to0^+}h(b)=1$.
The result follows from the dominated convergence theorem as 
in~\cite[Lemma~5.2]{MT12}.
\end{proof}

\begin{prop} \label{prop-middle}
Let $\beta'\in(\frac12,\beta)$.
For all  $2\pi<a<t$, 
\[
\int_{a/t}^1  e^{ibt}\hat\rho_{v,w}(ib)\,db
=O(\ell(t)^{-1}t^{-(1-\beta)}a^{-(2\beta'-1)}).
\]
\end{prop}

\begin{proof}
We modify~\cite[Lemma~5.1]{MT12} to deal with the $\epsilon$
in Corollary~\ref{cor-T}(c).   
Let $b,b_1,b_2\in(0,1]$, $b_1<b_2$.
By Proposition~\ref{prop-renew} and Corollary~\ref{cor-T}(b,c),
\begin{align*}
&  |\hat\rho(ib)|    \ll \ell(1/b)^{-1}  b^{-\beta}\|v\|_\theta|w|_\infty, \\[.75ex]
&  |\hat\rho(ib_1)-\hat\rho(ib_2)|  \ll  \ell(1/b_1)^{-2}b_1^{-2\beta}\bigl\{\ell(|b_1-b_2|^{-1})|b_1-b_2|^\beta\\
&\qquad\qquad 
\qquad\qquad 
+\ell(|b_1-b_2|^{\epsilon/\beta-1}b_2^{-\epsilon/\beta})b_2^\epsilon|b_1-b_2|^{\beta-\epsilon}\bigr\}\|v\|_\theta|w|_\infty.
\end{align*}

Write
\[
I=\int_{a/t}^1 e^{ibt}\hat\rho(ib)\,db
=-\int_{(a+\pi)/t}^{1+\pi/t} e^{ibt}\hat\rho(i(b-\pi/t))\,db.
\]
Then $2I=J_1 + J_2 +  J_3$, where
\begin{align*} 
J_1 & = -\int_1^{1+\pi/t} e^{ibt}\hat\rho(i(b-\pi/t))\,db, \qquad
J_2  = \int_{a/t}^{(a+\pi)/t} e^{ibt}\hat\rho(ib)\,db, \\
J_3 & = \int_{(a+\pi)/t}^1 e^{ibt}(\hat\rho(ib)-\hat\rho(i(b-\pi/t)))\,db.
\end{align*}
We suppress the factor $\|v\|_\theta |w|_\infty$ from now on.
Clearly $J_1=O(t^{-1})$, and by Potter's bounds,
\begin{align*}
|J_2| & \ll \int_{a/t}^{(a+\pi)/t} \ell(1/b)^{-1}b^{-\beta}\,db 
= \ell(t)^{-1}t^{-(1-\beta)}\int_{a}^{a+\pi} [\ell(t)/\ell(t/\sigma)]\sigma^{-\beta}\,d\sigma  \\ &
\ll  \ell(t)^{-1}t^{-(1-\beta)}\int_{a}^{a+\pi} \sigma^{-\beta'}\,d\sigma \ll 
 \ell(t)^{-1}t^{-(1-\beta)}a^{-\beta'}.
\end{align*}

Finally,
\begin{align*}
|J_3| & \ll \ell(t)t^{-\beta}\int_{a/t}^1\ell(1/b)^{-2}b^{-2\beta}\,db+
t^{-\beta+\epsilon}\int_{a/t}^1\ell(1/b)^{-2}b^{-2\beta+\epsilon}\ell(t^{1-\epsilon/\beta}b^{-\epsilon/\beta})\,db \\
&=J_3'+J_3''.
\end{align*}
By Potter's bounds,
\begin{align*}
J_3'  
=\ell(t)^{-1}t^{\beta-1}\int_{a}^t [\ell(t)/\ell (\sigma/t)]^2  \sigma^{-2\beta}\,d\sigma & 
\ll  \ell(t)^{-1}t^{\beta-1}\int_a^\infty \sigma^{-2\beta'}\,d\sigma  \\ & 
\ll
\ell(t)^{-1}t^{\beta-1}a^{-(2\beta'-1)},
\end{align*}
and shrinking $\epsilon$ if necessary so that $\epsilon<2(\beta-\beta')$, 
\begin{align*}
J_3'' & 
= \ell(t)^{-1}t^{\beta-1}\int_{a}^t [\ell(t)/\ell (\sigma/t)]^2[\ell(t/\sigma^{\epsilon/\beta})/\ell(t)] \sigma^{-(2\beta-\epsilon)}\,d\sigma
\\ & \ll  \ell(t)^{-1}t^{\beta-1}\int_a^\infty \sigma^{-2\beta'}\,d\sigma  
\ll
\ell(t)^{-1}t^{\beta-1}a^{-(2\beta'-1)},
\end{align*} 
as required.
\end{proof}

\begin{prop}   \label{prop-infty}
For any $\epsilon\in(0,\beta)$, there exist $\theta\in(0,1)$, $m\ge1$, such that
$\int_1^\infty  e^{ibt}\hat\rho_{v,w}(ib)\,db
=O(t^{-(\beta-\epsilon)})$,
for all $v\in F_\theta(\tilde Y)$, $w\in L^{\infty,m}(\tilde Y)$.
\end{prop}

\begin{proof}
Choose $m>2A+\epsilon+1$.
By Proposition~\ref{prop-m},
$\hat\rho_{v,w}(s)=\hat P_m(s)+\hat H_m(s)$, where
$\hat P_m(s)$ is a linear combination of $s^{-j}$, $j=1,\dots,m$, and 
$\hat H_m(s)=s^{-m}\hat \rho_{v,\partial_t^mw}(s)$.

Since $\hat P_m$ is analytic on $\Im s\ge1$, we can write
\begin{align*}
i\int_1^\infty e^{ibt}\hat P_m(ib)\,db & =
-\int_0^1 e^{(i-a)t}\hat P_m(i-a)\,da +
i\int_1^\infty e^{(-1+ib)t}\hat P_m(-1+ib)\,db \\
& = O(t^{-1})+O(e^{-t})=O(t^{-1}).
\end{align*}

It remains to estimate the contribution from
$\hat H_m(ib)=b^{-m}\hat \rho_{v,\partial_t^mw}(ib)$.   
Modifying the proof of Proposition~\ref{prop-middle} (using the fact
that $b\mapsto \hat \rho_{v,\partial_t^mw}(ib)$ satisfies the
other estimates in Corollary~\ref{cor-T}(b,c)), we have that for any $\epsilon>0$ and any $\epsilon'>\epsilon$,
\[
|\hat H_m(ib_1)-\hat H_m(ib_2)|\le 
b_2^{-(m-2A-\epsilon')}|b_1-b_2|^{\beta-\epsilon}\|v\|_\theta |\partial_t^mw|_\infty.
\]
Hence,
\begin{align*}
\Bigl|2\int_1^\infty e^{ibt}\hat H_m(ib)\,db\Bigr| & \le \int_1^\infty|\hat H_m(ib)-\hat H_m(i(b-\pi/t))|\,db
\\ & \qquad\qquad\qquad + \int_1^{1+\pi/t} |\hat H_m(i(b-\pi/t))|\,db \\ &
\ll t^{-(\beta-\epsilon)}\int_1^\infty b^{-(m-2A-\epsilon')}\,db
+O(t^{-1})= O(t^{-(\beta-\epsilon)}),
\end{align*}
as required.
\end{proof}

\begin{pfof}{Theorem~\ref{thm-infinite}(a)}
Combining Propositions~\ref{prop-near0},~\ref{prop-middle} and~\ref{prop-infty}
(with $\epsilon<2\beta-1$),
\[
\lim_{t\to\infty}\ell(t)t^{1-\beta}\int_0^\infty e^{ibt}\hat \rho_{v,w}(ib)\,db
=c_\beta^{-1}\int_0^a e^{i\sigma}\sigma^{-\beta}\,d\sigma
\int_{\tilde Y}v\,d\tilde\mu
\int_{\tilde Y}w\,d\tilde\mu
+O(a^{-(2\beta'-1)}).
\]
Since $a$ is arbitrary and $\beta'>1/2$,
\[
\lim_{t\to\infty}\ell(t)t^{1-\beta}\int_0^\infty e^{ibt}\hat \rho_{v,w}(ib)\,db
=c_\beta^{-1}\int_0^\infty e^{i\sigma}\sigma^{-\beta}\,d\sigma
\int_{\tilde Y}v\,d\tilde\mu
\int_{\tilde Y}w\,d\tilde\mu.
\]
A standard calculation shows that $\Re(c_\beta^{-1}\int_0^\infty e^{i\sigma}\sigma^{-\beta}\,d\sigma)=\sin\beta\pi$.  Hence
the result follows from Proposition~\ref{prop-inverse}.
\end{pfof}

\subsection{The case $\beta=1$}

We sketch the differences for $\beta=1$.
Here $\mu(\varphi>t)=\ell(t)t^{-1}$ where 
$\int_1^\infty \ell(t)t ^{-1}\,dt=\infty$.
By Karamata's Theorem, 
$\tilde\ell(t)=\int_1^t \ell(s)s^{-1}\,ds$
is slowly varying and $\ell(t)/\tilde\ell(t)\to0$ 
as $t\to\infty$.  In particular, $\tilde\ell$ is monotone increasing
and $\lim_{t\to\infty}\tilde\ell(t)=\infty$.

Many of the basic estimates change in a mild way.
The estimates on the imaginary axis ($s=ib$) in Section~\ref{sec-funct} are
unchanged except that all occurrences of $\ell(1/b)$ on the right-hand-side are
replaced by $\tilde\ell(1/b)$.

The major alteration is that 
$\hat\rho(ib)$ is not integrable near zero.  As in~\cite[Section~6]{MT12}, we 
replace $\int_{-\infty}^\infty e^{ibt}\hat\rho(ib)\,db$ by
the expression
\[
\int_{-\infty}^\infty e^{ibt}\Re \hat\rho(ib)\,db=
2\int_0^\infty \cos bt\,\Re \hat\rho(ib)\,db.
\]
In addition to the modified estimates for $\hat\rho(ib)$
(which are inherited by $\Re\hat\rho(ib)$), we have the improved asymptotics
\[
\Re\hat\rho(ib)\sim\frac{\pi}{2} \ell(1/b)\tilde\ell(1/b)^{-2}b^{-1}
\int_{\tilde Y}v\,d\tilde\mu
\int_{\tilde Y}w\,d\tilde\mu\enspace\text{as}\enspace b\to0^+,
\]
from which it follows that 
\[
\lim_{t\to\infty}\tilde\ell(t)\int_0^{a/t}\cos tb\,\Re\hat\rho(ib)\,db=\frac{\pi}{2}
\int_{\tilde Y}v\,d\tilde\mu
\int_{\tilde Y}w\,d\tilde\mu.
\]
We omit the details of these last two assertions which follow
from straightforward
modifications of the calculations for $\beta\in(\frac12,1)$ (cf.~\cite[Section 6]{MT12}).
It now follows exactly as in Subsection~\ref{sec-half} that
\[
\lim_{t\to\infty}\tilde\ell(t)\int_0^\infty\cos tb\,\Re\hat\rho(ib)\,db=\frac{\pi}{2}
\int_{\tilde Y}v\,d\tilde\mu
\int_{\tilde Y}w\,d\tilde\mu.
\]

Hence to prove Theorem~\ref{thm-infinite}(a) for $\beta=1$, it remains to prove the following result.
\begin{prop} \label{prop-inverse_one}
$\rho(t)=\frac{1}{\pi}\int_{-\infty}^\infty e^{ibt}\Re\hat\rho(ib)\,db$.
\end{prop}

\begin{proof}
Write $s=a+ib$.
Define $g:\R\to\R$ to be even with $g(t)=e^{-at}\rho(t)$ for $t>0$.
Then $\Re\hat\rho(s)=\frac12\int_{-\infty}^\infty e^{-ib\tau}g(\tau)\,d\tau$.
By the Fourier inversion formula,
\[
\int_{-\infty}^\infty e^{ibt}\Re\hat\rho(s)\,db
=\frac12 \int_{-\infty}^\infty e^{ibt}\Bigl\{\int_{-\infty}^\infty e^{-ib\tau}g(\tau)\,d\tau\Bigr\}\,db=\pi g(t).
\]
Hence, restricting to $t>0$,
\[
\rho(t)=e^{at}g(t)=\frac{1}{\pi}\int_{-\infty}^\infty e^{st}\Re\hat\rho(s)\,db
=\frac{1}{\pi}\int_{\Gamma_1} e^{st}\Re\hat\rho(s)\,ds,
\]
where $\Gamma_1$ is the contour $\Re s=1$ traversed upwards.

As in Proposition~\ref{prop-inverse}, we can move the contour to the
contour $\Gamma_0$ consisting again of two segments of the imaginary axis and a semicircle $\Gamma_\delta$ of radius $\delta$ around the origin.
By~\cite[Proposition~6.1]{MT12}, 
$|\int_{-\delta}^\delta e^{ibt}\hat\rho(ib)\,db|
\ll \int_0^\delta\ell(1/b)\tilde\ell(1/b)^{-2}b^{-1}\,db\ll \tilde\ell(1/\delta)^{-1}\to0$
as $\delta\to0$.   
Using the estimates in~\cite[Lemma~6.4]{MT12}, it can be shown that
$|\int_{\Gamma_\delta}e^{ibt}\hat\rho(ib)\,db|\to0$ as $\delta\to0$.
Hence the contour can be moved to the imaginary axis completing the proof.
\end{proof}

\section{Second order asymptotics and rates of mixing}
\label{sec-rates}

In this section, we prove Theorem~\ref{thm-infinite}(b).
% We suppose without loss that $q\in(1,2\beta)$.
Choose $\delta>0$ such that $\lambda(ib)$
is well defined for $b\in (0, \delta)$. 
\begin{prop} \label{prop-hot}
There are constants $e_1,e_2\in\C$ such that
$1-\lambda(ib)=e_1b^\beta(1-e_2b^{1-\beta}+O(b^{q-\beta}))$ 
for $b\in(0,\delta)$.
\end{prop}

\begin{proof}
 From the proof of Lemma~\ref{lem-zero}, we recall that 
 $\lambda(ib) =\int_Y
e^{-ib\varphi}\, d\mu+V(ib)$, where $|V(ib)|=O(b^{2\beta-\epsilon})$. 
Here, $\epsilon>0$ is arbitrarily small so 
$|V(ib)|=O(b^q)$.
By~\cite[Lemma~3.2]{MT12}, 
$\int_Y e^{-ib\varphi}\,d\mu=1-e_1 b^\beta+e_1e_2b+O(b^q)$
and the result follows.  (Note that much of the proof of~\cite[Lemma~3.2]{MT12}
is not required.  
The functions $H$ and $H_1$ introduced there coincide in the continuous time case.
Moreover, the four terms involving $v_\theta^s$ in~\cite[Lemma 3.2]{MT12} 
have been subsumed into the $V(ib)$ term.)
\end{proof}

\begin{cor} \label{cor-hot}
There are constants $c_j\in\C$ such that
\[
\hat\rho(ib)=\sum_jc_j b^{-((j+1)\beta-j)}
\int_{\tilde Y} v\,d\tilde\mu 
\int_{\tilde Y} w\,d\tilde\mu 
+ O(b^{-(2\beta-q)})\|v\|_\theta\,|w|_\infty,
\]
for $b\in(0,\delta)$, $v\in F_\theta(\tilde Y)$, $w\in L^\infty(\tilde Y)$,
where the sum is over those $j\ge0$ with $(j+1)\beta-j\ge 2\beta-q$.
\end{cor}

\begin{proof}
Recall that
\[
(I-\hat R_0(ib))^{-1}=(1-\lambda(ib))^{-1}P(0)- (1-\lambda(ib))^{-1}(P(ib)-P(0))+
(I-\hat R_0(ib))^{-1}Q(ib),
\]
for $b\in(0,\delta)$.
It follows from Proposition~\ref{prop-hot} that
\[
 (1-\lambda(ib))^{-1}=\sum_j c_j b^{-((j+1)\beta-j)}+ O(b^{-(2\beta-q)}),
\]
for constants $c_0,c_1,\ldots\in\C$.
By Proposition~\ref{prop-H2}(a), $\|(I-\hat R_0(ib))^{-1}Q(ib)\|_\theta=O(1)$. 
By Lemma~\ref{lem-R}, 
$\|\hat R_0(ib)-R\|_\theta\ll b^{\beta}$ and it follows that
$\|P(ib)-P(0)\|_\theta\ll b^{\beta}$. 
Hence
\[
(I-\hat R_0(ib))^{-1}v_0
=\sum_jc_j b^{-((j+1)\beta-j)}
\int_Y v_0\,d\mu + O(b^{-(2\beta-q)})\|v_0\|_\theta,
\]
for all $v_0\in F_\theta(Y)$.
By Lemma~\ref{lem-U}(a),
\[
\hat U(0)(I-\hat R(ib))^{-1}v
=\sum_jc_j b^{-((j+1)\beta-j)}
\int_{\tilde Y} v\,d\tilde\mu + O(b^{-(2\beta-q)})\|v\|_\theta,
\]
By Lemma~\ref{lem-U}(b), $\|\hat U(ib)-\hat U(0)\|_\theta\ll b^\beta$ and so
\[
\hat T(ib)v=\hat U(ib)(I-\hat R(ib))^{-1}v
=\sum_jc_j b^{-((j+1)\beta-j)}
\int_{\tilde Y} v\,d\tilde\mu + O(b^{-(2\beta-q)})\|v\|_\theta.
\]
The result follows from Proposition~\ref{prop-renew}.
\end{proof}

\begin{pfof}{Theorem~\ref{thm-infinite}(b)}
By Proposition~\ref{prop-middle},
for all $\beta'\in(\frac12,\beta)$, 
$\int_{a/t}^1  e^{ibt}\hat\rho_{v,w}(ib)\,db
=O(t^{-(1-\beta)}a^{-(2\beta'-1)})$ where $\beta-\beta'$ is arbitrarily small.

A calculation (see for example~\cite[Proposition~9.5]{MT12}) shows that
$\int_0^{a/t}b^{-((j+1)\beta-j)}e^{-itb}db={\rm const}\,t^{-(j+1)(1-\beta)}(1+
O(a^{-((j+1)\beta-j)}))$.
Also, $2\beta-q\in(0,1)$ so that
$\int_0^{a/t}b^{-(2\beta-q)}\,db=O((a/t)^{1-2\beta+q})$.

Choosing $a=t^{1-q^{-1}\beta-\epsilon'}$, we obtain from Corollary~\ref{cor-hot} that
\begin{align*}
\int_{0}^1e^{ibt}\hat\rho_{v,w}(ib)\,db
=\sum_j d_j t^{-j(1-\beta)}
\int_{\tilde Y} v\,d\tilde\mu 
\int_{\tilde Y} w\,d\tilde\mu 
+O(t^{-\beta(1-q^{-1}(2\beta-1)-\epsilon)}).
\end{align*}
Also, by Proposition~\ref{prop-infty},
$\int_1^\infty  e^{ibt}\hat\rho_{v,w}(ib)\,db =O(t^{-(\beta-\epsilon)})$.
Hence the result follows from Proposition~\ref{prop-inverse}.
\end{pfof}

\part{Finite measure systems}
\label{part-finite}

In this part of the paper, we prove our main results in the finite measure context.
Throughout, we continue to assume the setup from Section~\ref{sec-main}, so $F:Y\to Y$ is a full branch Gibbs-Markov map and $\varphi:Y\to\Z^+$ is a roof function satisfying assumptions (A1) and (A2).  In addition, we make the
standing assumption throughout this part of the paper that $\mu(\varphi>t)=O(t^{-\beta})$ where $\beta>1$.

In Section~\ref{sec-decomp}, we decompose the family
of operators $\hat T_0(s)$ into various pieces and formulate 
Lemmas~\ref{lem-T12},~\ref{lem-T3},~\ref{lem-T4} 
that provide estimates for each of the pieces.
Theorem~\ref{thm-finite}(a) thereby reduces to proving these lemmas.

In Section~\ref{sec-T12}, we prove Lemma~\ref{lem-T12}.
Section~\ref{sec-D} contains an operator-theoretic estimate, and
in Section~\ref{sec-T3}, we prove Lemma~\ref{lem-T3}.
Section~\ref{sec-smooth} contains estimates on derivatives of
various families of operators.
In Section~\ref{sec-fml}, we derive a continuous time version of
the ``first main lemma'' that was crucial in~\cite{Gouezel-sharp,Sarig02}.
In Section~\ref{sec-T4}, we prove Lemma~\ref{lem-T4} completing the proof of
Theorem~\ref{thm-finite}(a).
In Section~\ref{sec-zero}, we prove Theorem~\ref{thm-finite}(b).

In this part of the paper, 
$k\ge1$ is fixed but chosen sufficiently large.  All implied constants are allowed to depend on $k$ unless stated otherwise.  

\section{Decomposition for $\hat T_0$}
\label{sec-decomp}

Let $k\ge1$ and define $\varphi^*=\varphi\wedge k$,
$\bv^*=\int_Y\varphi^*\,d\mu$.
Recall that $P(0):L^1(Y)\to L^1(Y)$ is the projection
$P(0)v=\int_Y v\,d\mu$ and
define
\begin{align*}
& P_\varphi=(1/\bar\varphi)P(0),
\quad P_\varphi^*=(1/\bar\varphi^*)P(0), \quad
\hat R_0^*(s)v=R(e^{-s\varphi^*}v), \\
 & \hat T_0^*(s)=(I-\hat R_0^*(s))^{-1}=s^{-1}P_\varphi^*+\hat H^*(s), \quad
 \hat B(s)= s\hat T_0(s).
\end{align*}
Also define
\begin{align*}
\hat C^*(s)=s^{-1}(\hat R_0^*(s)-\hat R_0(s)),
\quad \hat D^*(s)=\hat C^*(0)-\hat C^*(s).
\end{align*}
We note that $\hat C^*(0)v=R((\varphi-\varphi^*)v)$ and hence
$\int_Y\hat C^*(0)1_Y\,d\mu=\bv-\bv^*$.  

\begin{prop} \label{prop-k}
Let $c_2=(C_2+1)^{-1}$ where $C_2$ is as in (A1).
If $a\in\alpha$ such that $|1_a\varphi|_\infty>t$, then
$\varphi(y)>c_2t$ for all $y\in a$.
\end{prop}

\begin{proof}
Choose $y_0\in a$ with $\varphi(y_0)>t$.
By assumption (A1), for all $y\in a$,
\[
t-\varphi(y)< \varphi(y_0)-\varphi(y)\le |1_a\varphi|_{\theta_0}d_{\theta_0}(y,y_0)<|1_a\varphi|_{\theta_0}\le  C_2\varphi(y),
\]
and the result follows.
\end{proof}

\begin{prop} \label{prop-C}
$\|\hat C^*(s)\|_\infty\le C_1\int_{\{\varphi>c_2 k\}}\varphi\,d\mu$
for $s\in\overline\H$. 
In particular, if $\varphi\in L^1$, then
$\hat C^*$ extends continuously to $\overline\H$ and
$\|\hat C^*(s)\|_\infty\to0$ as $k\to\infty$ uniformly on $\overline\H$.   
\end{prop}

\begin{proof}
Recalling~\eqref{eq-Rn}, we have
\begin{align} \label{eq-hatC}
(\hat C^*(s)v)(y)= &
s^{-1}\sum_{a\in\alpha:|1_a\varphi|_\infty\ge k}e^{g(y_a)}v(y_a)
(e^{-s\varphi^*(y_a)}-e^{-s\varphi(y_a)}),
\end{align}
so by~\eqref{eq-GM} and Proposition~\ref{prop-k},
\begin{align*}
|(\hat C^*(s)v)(y)|  & 
\le C_1|v|_\infty \sum_{a:|1_a\varphi|_\infty\ge k}\mu(a)
(\varphi(y_a)-\varphi^*(y_a))
\\ &   \le C_1|v|_\infty \sum_{a:|1_a\varphi|_\infty\ge k}\mu(a) \varphi(y_a)
\le C_1 |v|_\infty\int_{\{\varphi>c_2 k\}}\varphi\,d\mu,
\end{align*}
as required.
\end{proof}

\begin{prop} \label{prop-decomp}
For $k$ sufficiently large,
$\hat T_0(s) = \hat T_{0,1}(s)+\hat T_{0,2}(s)+ \hat T_{0,3}(s)+ \hat T_{0,4}(s)$ for all $s\in\H$, where
\begin{align*}
\hat T_{0,1}(s) & = s^{-1}P_\varphi, \qquad
\hat T_{0,2}(s)  = s^{-1}P_\varphi\hat D^*P_\varphi, \qquad
\hat T_{0,3}(s) & =  s^{-1}(I-P_\varphi \hat D^*)^{-1}(P_\varphi\hat D^*)^2P_\varphi,
\\ 
\hat T_{0,4}(s) & = (I-P_\varphi \hat D^*)^{-1}(I-P_\varphi \hat C^*(0))\hat H^*(I-\hat C^*\hat B).
\end{align*}
\end{prop}

\begin{proof}
Since $\int_Y\hat C^*(0)1_Y\,d\mu=\bv-\bv^*$,  it follows that
\begin{align} \label{eq-phi}
(I-P_\varphi \hat C^*(0))P_\varphi^*=P_\varphi, \quad
(I-P_\varphi \hat C^*(0))(I+P_\varphi^*\hat C^*)=I-P_\varphi \hat D^*. 
\end{align}

Using the identity 
$\hat T_0=\hat T_0^*-\hat T_0^*(\hat R_0^*-\hat R_0)\hat T_0$, 
it follows that
\begin{align*}
\hat T_0  =\hat T_0^*-\hat T_0^*\hat C^*\hat B 
  & =s^{-1}P_\varphi^*+\hat H^*-s^{-1}P_\varphi^*\hat C^*\hat B-\hat H^*\hat C^*\hat B
\\ & =s^{-1}P_\varphi^*-P_\varphi^*\hat C^*\hat T_0+\hat H^*(I-\hat C^*\hat B).
\end{align*}
Hence
$(I+P_\varphi^*\hat C^*)\hat T_0 = 
s^{-1}P_\varphi^*+\hat H^*(I-\hat C^*\hat B)$.
Multiplying throughout by $(I-P_\varphi \hat C^*(0))$, and using~\eqref{eq-phi},
we obtain
\[
(I-P_\varphi \hat D^*)\hat T_0 = 
s^{-1}P_\varphi +(I-P_\varphi \hat C^*(0))\hat H^*(I-\hat C^*\hat B).
\]
 For $k$ sufficiently large,
we can invert $I-P_\varphi \hat D^*$ by Proposition~\ref{prop-C} and the result follows.

\end{proof}

Substituting into~\eqref{eq-rhohat}, we obtain that
$\hat \rho(s)=\sum_{i=1}^4 \hat\rho_i(s)$, where
\begin{align} \label{eq-rhoi}
& \hat \rho_i(s)  = 
\int_{\tilde Y} \hat U(s)\hat T_i(s)v\,w\,d\mu^\varphi=
(1/\bv)\int_{\tilde Y} \hat U(s)\hat T_i(s)v\,w\,d\tilde\mu, \\
& (\hat T_i(s)v)(y,u) =(\hat T_{0,i}(s)v^u)(y),
\quad i=1,2,3,4. \nonumber
\end{align}
Theorem~\ref{thm-finite} is an immediate consequence of the next three lemmas.
We recall that $\zeta(t)$ and $\xi_{\beta,\epsilon}(t)$
were defined in~\eqref{eq-xi}.

\begin{lemma} \label{lem-T12}
Suppose that $\mu(\varphi>t)=O(1/t^\beta)$ for some $\beta>1$.
Then
\begin{itemize}
\item[(a)] $\rho_1(t)- \bar v\bar w =O\bigl(|v|_\infty|w|_\infty t^{-\beta}\bigr)$, and
\item[(b)]
$\rho_2(t)  = 
(1/\bar\varphi)\zeta(t)\bar v\bar w 
+O\bigl(v|_\infty|w|_\infty  t^{-\beta}\bigr)$,
\end{itemize}
for all $v,w\in L^\infty(\tilde Y)$, $t>0$, $k\ge1$.
\end{lemma}

\begin{lemma} \label{lem-T3}
Suppose that $\mu(\varphi>t)=O(1/t^\beta)$ for some $\beta>1$.
Then for any $\epsilon>0$ and for all $k$ sufficiently large, there is a constant $C>0$ such that 
\begin{align*}
|\rho_3(t)|\le C|v|_\infty|w|_\infty\xi_{\beta,\epsilon}(t),
\end{align*}
for all $v,w\in L^\infty(\tilde Y)$, $t>0$.
\end{lemma}

\begin{lemma} \label{lem-T4}
Assume conditions (A1) and (A2) and suppose that
$\mu(\varphi>t)=O(t^{-\beta})$ where $\beta>1$.
Then for any $\epsilon>0$ and for all $k$ sufficiently large, there exists $\theta\in(0,1)$, $m\ge1$, $C>0$,
such that
\begin{align*}
|\rho_4(t)|\le C\|v\|_\theta|w|_{\infty,m}t^{-(\beta-\epsilon)},
\end{align*}
for all $v\in F_\theta(\tilde Y)$, $w\in L^{\infty,m}(\tilde Y)$, $t>0$.
\end{lemma}

\section{Proof of Lemma~\ref{lem-T12}}
\label{sec-T12}

We require the following preliminary result.

\begin{prop} \label{prop-r}
  Let $\hat r(s)= s^{-1}\int_{\tilde Y}\hat U(s)v\,w\,d\tilde\mu$.
Then
\begin{align*}
r(t) & =\int_{\tilde Y}\int_0^u v(y,\tau)\,d\tau\,w(y,u)\,d\tilde\mu
+\int_{\tilde Y}\int_u^1 v(y,\tau)\,d\tau\,w(Fy,u)\,d\tilde\mu
+O(|v|_\infty|w|_\infty\mu(\varphi>t)).
\end{align*}
\end{prop}

\begin{proof}
By Proposition~\ref{prop-U},
\begin{align*}
s^{-1}(\hat U(s)v)(y,u) &  =
s^{-1}\int_0^1 e^{-s\tau}v(y,u-\tau)1_{[\tau,1]}(u)\,d\tau+
s^{-1}\int_1^\infty e^{-s\tau}(\tilde R v_\tau)(y,u)\,d\tau,
\end{align*}
with inverse Laplace transform
\begin{align*}
\int_0^u 1_{[\tau,\infty)}(t)v(y,u-\tau)\,d\tau+
\int_1^\infty 1_{[\tau,\infty)}(t)(\tilde R v_\tau)(y,u)\,d\tau.
\end{align*}
Hence $r(t)=r_1(t)+r_2(t)$ where
\begin{align*}
r_1(t) & = \int_{\tilde Y}\int_0^u 1_{[\tau,\infty)}(t)v(y,u-\tau)\,d\tau\,w\,d\tilde\mu,  \qquad
r_2(t)  = \int_{\tilde Y}\int_1^\infty 1_{[\tau,\infty)}(t)
(\tilde R v_\tau)(y,u)\,d\tau\,w\,d\tilde\mu.
\end{align*}

For $t>1$,
\[
r_1(t)=\int_{\tilde Y}\int_0^u v(y,u-\tau)\,d\tau\,w\,d\tilde\mu
=\int_{\tilde Y}\int_0^u v(y,\tau)\,d\tau\,w\,d\tilde\mu.
\]
Also,
\begin{align*}
r_2(t) & = \int_{\tilde Y}\int_1^\infty 1_{[\tau,\infty)}(t)
1_{\{\tau<\varphi<\tau+1-u\}}v(y,u-\tau+\varphi)\,d\tau\,w\circ \tilde F\,d\tilde\mu \\
& = \int_{\tilde Y}\int_{\varphi-1+u}^\varphi 1_{[\tau,\infty)}(t)
v(y,u-\tau+\varphi)\,d\tau\,w\circ \tilde F\,d\tilde\mu \\
& = \int_{\tilde Y}\int_u^1 1_{[u-\tau+\varphi,\infty)}(t)
v(y,\tau)\,d\tau\,w\circ \tilde F\,d\tilde\mu  \\
& = \int_{\tilde Y}\int_u^1 
v(y,\tau)\,d\tau\,w\circ \tilde F\,d\tilde\mu -E(t),
\end{align*}
where
\[
E(t)=
\int_{\tilde Y}\int_u^1 1_{[0,u-\tau+\varphi]}(t)
v(y,\tau)\,d\tau\,w\circ \tilde F\,d\tilde\mu
=\int_{\tilde Y}\int_u^1 1_{\{\varphi>t+\tau-u\}}
v(y,\tau)\,d\tau\,w\circ \tilde F\,d\tilde\mu.
\]
Finally, note that $|E(t)|\le |v|_\infty|w|_\infty\mu(\varphi>t)$.
\end{proof}

\begin{pfof}{Lemma~\ref{lem-T12}(a)}
We have
$\hat \rho_1(s)  =(1/\bv) \int_{\tilde Y} \hat U(s)\hat T_1(s)v\,w\,d\tilde\mu$,
where
\begin{align*}
(\hat T_1(s)v)(y,u)=(\hat T_{0,1}(s)v^u)(y)=(1/\bv)s^{-1}\int_Y v^u\,d\mu.
\end{align*}
By Proposition~\ref{prop-r},
\begin{align*}
\rho_1(t) &= (1/\bv)^2\int_{\tilde Y}\int_0^u\Bigl(\int_Y v^\tau\,d\mu\Bigr)\,d\tau
\,w(y,u)\,d\tilde\mu \\ & \qquad  +
 (1/\bv)^2\int_{\tilde Y}\int_u^1\Bigl(\int_Y v^\tau\,d\mu\Bigr)\,d\tau
\,w(Fy,u)\,d\tilde\mu+O(\mu(\varphi>t)).
\end{align*}
Note that $\int_Y v^\tau\,d\mu$ is independent of $y$ and $\tilde F$ acts trivially on the second coordinate, so the second term 
reduces to
\[
 (1/\bv)^2\int_{\tilde Y}\Bigr\{\int_u^1\Bigl(\int_Y v^\tau\,d\mu\Bigr)\,d\tau
\,w\Bigr\}\circ \tilde F\,d\tilde\mu=
 (1/\bv)^2\int_{\tilde Y}\int_u^1\Bigl(\int_Y v^\tau\,d\mu\Bigr)\,d\tau
\,w\,d\tilde\mu.
\]
Hence
\begin{align*}
\rho_1(t) &= (1/\bv)\int_{\tilde Y}\int_0^1\Bigl(\int_Y v^\tau\,d\mu\Bigr)\,d\tau
\,w(y,u)\,d\mu^\varphi + O(\mu(\varphi>t)).
\end{align*}
But
\[
\int_0^1\Bigl(\int_Y v^\tau\,d\mu\Bigr)\,d\tau
=\int_Y\int_0^1 v(y,\tau)\,d\tau\,d\mu=
\int_{\tilde Y}v\,d\tilde\mu=
\bv \int_{\tilde Y}v\,d\mu^\varphi,
\]
and the result follows.
\end{pfof}

Recall that $C^*(t)$ is the inverse Laplace transform of $\hat C^*(s)$.

\begin{prop} \label{prop-D}
$\int_Y \hat C^*(0)1_Y\,d\mu=\bar\varphi-\bar\varphi^*=\zeta(k)$ and
$\int_Y C^*(t)1_Y\,d\mu=1_{\{t>k\}}\mu(\varphi>t)$.
\end{prop}

\begin{proof}
Let $G(x)=\mu(\varphi<x)$ denote the distribution function of $\varphi$.
For the first statement,
\begin{align*}
& \int_Y\hat C^*(0)1_Y\,d\mu  =\int_Y (\varphi-\varphi^*)\,d\mu
=\int_0^\infty x\,dG-
\int_0^k x\,dG^*=\int_k^\infty x\,dG-k\mu(\varphi>k) \\
& \qquad  =-\int_k^\infty  x\,d(1-G(x))-k\mu(\varphi>k)
\\ & \qquad = -x(1-G(x))\Bigr|_{x=k}^{x=\infty}+\int_k^\infty (1-G(x))\,dx-k\mu(\varphi>k)
= \int_k^\infty \mu(\varphi>x)\,dx.
\end{align*}

For the second statement,
$\int_Y\hat C^*(s)1_Y\,d\mu  =\int_Y s^{-1}(e^{-s\varphi^*}-e^{-s\varphi})\,d\mu$, so
\begin{align*}
\int_Y C^*(t)1_Y\,d\mu  
& =\int_Y (1_{[\varphi^*,\infty)}(t)- 1_{[\varphi,\infty)}(t))\,d\mu
=\mu(\varphi^*<t)- \mu(\varphi<t)
\\ & =\mu(\varphi>t)- \mu(\varphi^*>t)
 =\mu(\varphi>t)- \mu(\varphi\wedge k>t).
\end{align*}
But 
 $\mu(\varphi\wedge k>t)=0$ if $k<t$ and
 $\mu(\varphi\wedge k>t)=\mu(\varphi>t)$ if $k>t$.
\end{proof}

In the next proof, $a\star b$ denotes the convolution
$(a\star b)(t)=\int_0^t a(\tau)b(t-\tau)\,d\tau$ of real-valued functions of $t\in[0,\infty)$.   (In subsequent sections, we speak also of the convolution of operator-valued functions of $t$.)

\begin{pfof}{Lemma~\ref{lem-T12}(b)}
We have
$\hat \rho_2(s)  =(1/\bv) \int_{\tilde Y} \hat U(s)\hat T_2(s)v\,w\,d\tilde\mu$,
where
\begin{align*}
(\hat T_2(s)v)(y,u)  =(\hat T_{0,2}(s)v^u)(y) & =
(1/\bv)^2s^{-1}P\hat D^*(s)P(0)v^u 
\\ & = (1/\bv)^2s^{-1}\int_Y\hat D^*(s)1_Y\,d\mu\int_Y v^u\,d\mu.
\end{align*}
Comparing with the proof of part (a), we observe that
\[
\bv\hat\rho_2(s)=\hat\rho_1(s)\int_Y\hat D^*(s)1_Y\,d\mu
=\hat\rho_1(s)\int_Y \hat C^*(0)1_Y\,d\mu
-\hat\rho_1(s)\int_Y \hat C^*(s)1_Y\,d\mu.
\]
By Proposition~\ref{prop-D}, for $t>k$,
\begin{align*}
\bv\rho_2(t) & =\zeta(k) \,\rho_1(t)-
(1_{\{t>k\}}\mu(\varphi>t))\star(\rho_1(t)) \\
& =\zeta(k) \bigl(\bar v\bar w+O(\mu(\varphi>t))\bigr)
-  \int_k^t\mu(\varphi>\tau)\rho_1(t-\tau)\,d\tau \\
& = \zeta(k) \bar v\bar w
+O\bigl(\zeta(k) \mu(\varphi>t)\bigr)
-\bar v\bar w(\zeta(k)-\zeta(t))
\\ & \qquad  \qquad \qquad  \qquad \qquad +O\Bigl(\int_k^t\mu(\varphi>\tau)\mu(\varphi>t-\tau)\,d\tau\Bigr)
\\ & = \zeta(t) \bar v\bar w
+O\bigl(\zeta(k) \mu(\varphi>t)\bigr)
+O\bigl(\mu(\varphi>t)\star\mu(\varphi>t)\bigr).
\end{align*}
The result follows.
\end{pfof}

\section{An estimate for $(I-P_\varphi\hat D^*)^{-1}$}
\label{sec-D}

In Proposition~\ref{prop-C}, we showed that for $k$ sufficiently large the
family of operators
$(I-P_\varphi\hat D^*(s))^{-1}$ on $L^\infty(Y)$ is analytic on $\H$ with a continuous extension to $\overline\H$.   
Moreover, $\|\hat D^*\|_\infty$ is uniformly small on $\overline\H$ for $k$ large.

In this section, we obtain an estimate on the decay of its inverse Laplace transform.  This is required in a diluted form in Section~\ref{sec-T3}
and in its full strength in Section~\ref{sec-T4}.

Let $\mathcal{B}$ be a Banach space and suppose that
$S:[0,\infty)\to\mathcal{B}$ lies in $L^1$ with Laplace transform
$\hat S:\overline\H\to\mathcal{B}$.
We write $\hat S\in\mathcal{R}(a(t))$ if $\|S(t)\|\le Ca(t)$ for all $t\ge0$.

\begin{prop} \label{prop-CC}
$\|C^*(t)\|_\infty \le C_1 1_{\{t\ge k\}}\mu(\varphi>c_2t)$,
for all $k\ge1$.
In particular, $\hat C^*\in\mathcal{R}(\mu(\varphi>c_2 t))$.
\end{prop}

\begin{proof}
Starting from formula~\eqref{eq-hatC} for $\hat C^*$,
the inverse Laplace transform is given by
\begin{align*}
(C^*(t)v)(y) & =\sum_{a\in\alpha} e^{g(y_a)}v(y_a)1_{[\varphi^*(y_a),\varphi(y_a)]}(t)
=\sum_{a\in\alpha} e^{g(y_a)}v(y_a)1_{\{\varphi(y_a)>k\}}1_{[k,\varphi(y_a)]}(t) \\
& =1_{\{t\ge k\}}\sum_{a\in\alpha} e^{g(y_a)}v(y_a)1_{[0,\varphi(y_a)]}(t).
\end{align*}
Hence by Proposition~\ref{prop-k},
\begin{align*}
|(C^*(t)v)(y)| & 
\le C_1 1_{\{t\ge k\}}|v|_\infty\sum_{a\in\alpha:|1_a\varphi|_\infty>t} \mu(a)
\le C_1 1_{\{t\ge k\}}|v|_\infty\mu(\varphi>c_2t),
\end{align*}
as required.
\end{proof}

\begin{rmk} \label{rmk-D}
Since $\hat D^*(s)=\hat C^*(0)-\hat C^*(s)$, it follows from Proposition~\ref{prop-CC}
that formally we have $D^*(t)=C^*(0)\delta_0(t)-C^*(t)$ where
$\|C^*(t)\|_\infty\le C_1 1_{\{t\ge k\}}\mu(\varphi>c_2 t)$.
To avoid such formal expressions, we restrict to estimating
expressions like $\hat D^*(s)\hat E(s)$ where $\hat E(s)$ has no constant terms.
\end{rmk}

\begin{cor} \label{cor-D}
Let $\mathcal{B}$ be a Banach space.
Let $\beta>1$.
Suppose that $\mu(\varphi>t)=O(1/t^\beta)$ 
and that $\hat E:\mathcal{B}\to L^\infty(Y)$ lies in $\mathcal{R}(1/t^{\beta})$.
Then 
$\hat D^*(s)\hat E(s):\mathcal{B}\to L^\infty(Y)$
lies in
$\mathcal{R}(1/t^{\beta})$.
\end{cor}

\begin{proof}
We have $\hat D^*(s)\hat E(s)=\hat C^*(0)\hat E(s)-\hat C^*(s)\hat E(s)$ with inverse Laplace transform $\hat C^*(0)E(t)-(C^*\star E)(t)\in \mathcal{R}(1/t^{\beta})$.
\end{proof}

\begin{prop} \label{prop-Dinv}
Let $\mathcal{B}$ be a Banach space.
Let $\beta>1$, $\epsilon>0$ such that $\beta-\epsilon>1$.
Suppose that $\mu(\varphi>t)=O(1/t^\beta)$ 
and that $\hat E:\mathcal{B}\to L^\infty(Y)$ lies in $\mathcal{R}(1/t^{\beta-\epsilon})$.
Then for $k$ sufficiently large,
$(I-P_\varphi\hat D^*(s))^{-1}\hat E(s):\mathcal{B}\to L^\infty(Y)$
lies in
$\mathcal{R}(1/t^{\beta-\epsilon})$.
\end{prop}

\begin{proof}
By Proposition~\ref{prop-C}, we can choose $k$ so large that
$\|P_\varphi\hat C^*(s)\|_\infty\le\frac13$ for all $s\in\overline\H$ and hence we can
write
\[
\hat Q(s)=(1-P_\varphi \hat D^*(s))^{-1}\hat E(s)=
\sum_{n=0}^\infty \hat Q_n(s),\quad  \hat Q_n(s)=(P_\varphi\hat C^*(0)-P_\varphi\hat C^*(s))^n\hat E(s).
\]

Given $\hat S\in\mathcal{R}(1/t^\beta)$, $\beta>1$, we
define $\|\hat S\|_{{\mathcal R}_\beta}=\int_0^\infty \|S(t)\|_\infty\,dt+
\sup_{t\ge0}\|S(t)\|_\infty t^\beta$.
This makes $\mathcal{R}(1/t^\beta)$ into a Banach algebra under composition
and we can rescale the norm so
that $\|\hat S_1\hat S_2\|_{\mathcal{R}_\beta}\le \|\hat S_1\|_{\mathcal{R}_\beta}\|\hat S_2\|_{\mathcal{R}_\beta}$.
In particular, 
\begin{align*}
\|\hat Q_n\|_{\mathcal{R}_{\beta-\epsilon}} & \le 
\sum_{j=0}^n\binom{n}{j}
\|P_\varphi \hat C^*(0)\|^j\bigl(\|P_\varphi\hat C^*\|_{\mathcal{R}_{\beta-\epsilon}}\bigr)^{n-j}\|\hat E\|_{\mathcal{R}_{\beta-\epsilon}} \\ &  \le
\sum_{j=0}^n\binom{n}{j}
\Bigl(\frac13\Bigr)^j\bigl(\|P_\varphi\hat C^*\|_{\mathcal{R}_{\beta-\epsilon}}\bigr)^{n-j}\|\hat E\|_{\mathcal{R}_{\beta-\epsilon}}.
\end{align*}

By Proposition~\ref{prop-CC},
\[
\|P_\varphi \hat C^*\|_{\mathcal{R}_{\beta-\epsilon}}
\le C_1 \int_k^\infty\mu(\varphi>c_2 t)\,dt+C_1\sup_{t>k}\mu(\varphi>c_2 t)t^{\beta-\epsilon}
\ll k^{-(\beta-\epsilon-1)}+k^{-\epsilon},
\]
so we can ensure that
$\|P_\varphi\hat C^*\|_{\mathcal{R}_{\beta-\epsilon}}<\frac13$ by choosing $k$ sufficiently large.  
Then
\[
\|\hat Q_n\|_{\mathcal{R}_{\beta-\epsilon}}
\le \sum_{j=0}^n\binom{n}{j}
\Bigl(\frac13\Bigr)^j\Bigl(\frac13\Bigr)^{n-j}\|
\hat E\|_{\mathcal{R}_{\beta-\epsilon}}
=\Bigl(\frac23\Bigr)^n \|\hat E\|_{\mathcal{R}_{\beta-\epsilon}}.
\]
Hence $\hat Q=(I-P_\varphi\hat D^*)^{-1}\hat E\in\mathcal{R}(1/t^\beta)$ as required.
\end{proof}

\begin{rmk} \label{rmk-Dinv}
Equally we can consider products of the form
$\hat E(s)(I-P_\varphi\hat D^*(s))^{-1}:L^\infty(Y)\to\mathcal{B}$
where $\hat E:L^\infty(Y)\to\mathcal{B}$ lies in $\mathcal{R}(1/t^{\beta-\epsilon})$ and the conclusion of Proposition~\ref{prop-Dinv} is unchanged.
\end{rmk}

\section{Proof of Lemma~\ref{lem-T3}}
\label{sec-T3}

We have
$\hat \rho_3(s)  = (1/\bv) \int_{\tilde Y} \hat U(s)\hat T_3(s)v\,w\,d\tilde\mu$,
where
\begin{align*}
(\hat T_3(s)v)(y,u)& =(\hat T_{0,3}(s)v^u)(y)
\\ & =
(1/\bv)^3s^{-1}\Bigl(1-(1/\bv)\int_Y\hat D^*(s)1_Yd\mu\Bigr)^{-1}\Bigl(\int_Y\hat D^*(s)1_Y\,d\mu\Bigr)^2\int_Yv^u\,d\mu.
\end{align*}

\begin{prop} \label{prop-D00}
$\int_Y \hat D^*(s)1_Y\,d\mu$ extends continuously to $\overline\H$ and
$\int_Y \hat D^*(s)1_Y\,d\mu=
\int_k^\infty (1-e^{-sx})\mu(\varphi>x)\,dx \le \zeta(k)$ for $s\in\overline\H$.
\end{prop}

\begin{proof}
Let $G(x)=\mu(\varphi<x)$ and
$G^*(x)=\mu(\varphi*<x)$ 
denote the distribution functions of $\varphi$ and $\varphi^*$.
Then
\begin{align*}
& s\int_Y\hat C^*(s)1_Y\,d\mu  
= \int_Y(e^{-s\varphi^*}-e^{-s\varphi})\,d\mu 
=\int_0^k e^{-sx}\,dG^*-
\int_0^\infty e^{-sx}\,dG
\\ & 
=e^{-sk}\mu(\varphi>k) -\int_k^\infty e^{-sx}\,dG
=e^{-sk}\mu(\varphi>k) +\int_k^\infty e^{-sx}\,d(1-G(x))
\\ & = 
e^{-sk}\mu(\varphi>k) 
+e^{-sx}(1-G(x))\Bigr|_{x=k}^{x=\infty}+s\int_k^\infty e^{-sx}(1-G(x))\,dx
\\ & = s\int_k^\infty e^{-sx}\mu(\varphi>x)\,dx.
\end{align*}
Hence $\int_Y\hat C^*(s)1_Y\,d\mu= \int_k^\infty e^{-sx}\mu(\varphi>x)\,dx$ and the formula for
$\int_Y\hat D^*(s)1_Y\,d\mu$ follows since $\hat D^*(s)=\hat C^*(0)-\hat C^*(s)$.
\end{proof}

\begin{prop}  \label{prop-DD}
Suppose that $\mu(\varphi>t)=O(1/t^\beta)$ for some $\beta>1$.  
Define $\xi_\beta(t)$ as in~\eqref{eq-xixi}.
Then
\begin{itemize}
\item[(a)] $s^{-1}\int_Y\hat D^*(s)1_Y\,d\mu\in \mathcal{R}(1/t^{\beta-1})$.
\item[(b)]
$s^{-1}(\int_Y\hat D^*(s)1_Y\,d\mu)^2\in 
\mathcal{R}(\xi_\beta(t))$.
\end{itemize}
\end{prop}

\begin{proof}
(a) 
Let $\hat q(s)=s^{-1}\int_Y\hat D^*(s)1_Y\,d\mu=
s^{-1}\int_Y\hat C^*(0)1_Y\,d\mu-
s^{-1}\int_Y\hat C^*(s)1_Y\,d\mu$
with inverse Laplace transform
$q(t)=
\int_Y\hat C^*(0)1_Y\,d\mu+\int_0^t \int_YC^*(\tau)1_Y\,d\mu\,d\tau$.
By Proposition~\ref{prop-D}, for $t>k$,
\begin{align*}
q(t) & =
\zeta(k)+\int_0^t 1_{\{\tau>k\}}\mu(\varphi>\tau)\,d\tau
=\zeta(k)+
\int_k^t \mu(\varphi>\tau)\,d\tau
=\zeta(t).
\end{align*}

\noindent(b)
Let $'$ denote $\frac{d}{ds}$.  By Proposition~\ref{prop-D},
$\int_Y C^*(t)1_Y\,d\mu=1_{\{t>k\}}\mu(\varphi>t)=O(t^{-\beta})$ and hence
$\int_Y \hat {D^*}{'}(s)1_Y\,d\mu=-\int_Y\hat {C^*}{'}(s)1_Y\,d\mu\in\mathcal{R}(1/t^{\beta-1})$.

Let $\hat q(s)= s^{-1}(\int_Y\hat D^*(s)1_Y\,d\mu)^2$.
Then
\begin{align*} 
\hat q'(s)
& =-\Bigl(s^{-1}\int_Y\hat D^*(s)1_Y\,d\mu\Bigr)^2
+2\Bigl(\int_Y{\hat D^*}{'}(s)1_Y\,d\mu\Bigr)\Bigl(s^{-1}\int_Y\hat D^*(s)1_Y\,d\mu\Bigr).
\end{align*}
Each term is a product of two elements of $\mathcal{R}(1/t^{\beta-1})$.
Hence $\hat q'\in\mathcal{R} (\{1/t^{\beta-1}\}\star \{1/t^{\beta-1}\})$.
But $\hat q'(s)$ is the Laplace transform of $tq(t)$ so we obtain
that 
\begin{align*}
tq(t) & =O(\{1/t^{\beta-1}\}\star \{1/t^{\beta-1}\})= t\xi_\beta(t),
\end{align*}
as required.
\end{proof}

\begin{prop} \label{prop-U2}
$\hat U:L^\infty(\tilde Y)\to L^1(\tilde Y)$ lies in
$\mathcal{R}(\mu(\varphi>t))$.
\end{prop}

\begin{proof}
Directly from the definition of $U(t)$, we have
\[
|U(t)v|_1=|T_t(1_{\{\tilde\varphi>t\}}v)|_1
=|1_{\{\tilde\varphi>t\}}v|_1
\le\tilde\mu(\tilde\varphi>t)|v|_\infty
=\mu(\varphi>t)|v|_\infty.
\vspace*{-5ex}
\]
\end{proof}

\begin{pfof}{Lemma~\ref{lem-T3}}
By Propositions~\ref{prop-Dinv} and~\ref{prop-DD}(b),
$\hat T_3:L^\infty(\tilde Y)\to L^\infty(\tilde Y)$ lies in $\mathcal{R}(\xi_{\beta,\epsilon}(t))$.
By Proposition~\ref{prop-U2},
$\hat U\in \mathcal{R}(1/t^\beta)$.
Hence $\hat \rho_3\in\mathcal{R}(\xi_{\beta,\epsilon}(t))$.
\end{pfof}

\section{Smoothness of some families of operators}
\label{sec-smooth}

Let $s\mapsto \hat S(s)$ be an analytic family of operators,
$s\in\H$, such that the family extends continuously to $\overline\H$.
If $p\ge0$ is an integer, define
\[
d_p\hat S(ib)=\max_{j=0,\dots,p}\|\hat S^{(j)}(ib)\|.
\]
If $p>0$ is not an integer, define
\[
d_p\hat S(ib)= d_{[p]}\hat S(ib)+\sup_{h\neq0}\|\hat S^{([p])}(i(b+h)))-\hat S^{([p])}(ib)\|/|h|^{p-[p]}.
\]

\begin{prop}  \label{prop-smoothR}
Suppose that $\varphi\in L^p$ for some $p>0$, and let $\epsilon\in(0,p)$.
\begin{itemize}
\item[(a)]  Viewed as a family of operators on $L^\infty(Y)$, 
$b\mapsto \hat R_0(ib)$ is $C^p$ and there exists a constant $C>0$ such that
$d_p\hat R_0(ib)\le C$ for all $b\in\R$.
\item[(b)]  There exists $\theta\in(0,1)$ such that
viewed as a family of operators on $F_\theta(Y)$, 
$b\mapsto \hat R_0(ib)$ is $C^{p-\epsilon}$ and
$d_{p-\epsilon}\hat R_0(ib)\le C(1+|b|^\epsilon)$ for all $b\in\R$.
\end{itemize}
\end{prop}

\begin{proof}
(a)  By~\eqref{eq-Rn},
\begin{align*}
(\hat R_0^{(j)}(ib)v)(y)=\sum_{a\in\alpha}e^{g(y_a)}v(y_a)(i\varphi(y_a))^je^{ib\varphi(y_a)}
\end{align*}
and hence by~\eqref{eq-GM} and assumption (A1),
\begin{align*}
|(\hat R_0^{(j)}(ib)v)(y)|  & \le 
C_1\sum_{a\in\alpha}\mu(a)|v|_\infty \varphi(y_a)^j
\le C_1(C_2+1)^j|v|_\infty\sum_{a\in\alpha}\mu(a)\inf_a\varphi^j
\\ & \le C_1(C_2+1)^j|v|_\infty |\varphi|_j^j.
\end{align*}
Also, for $p$ not an integer,
\begin{align*}
(\{\hat R_0^{([p])}(i(b+h))-
\hat R_0^{([p])}(ib)\}v)(y)
=\sum_{a\in\alpha}e^{g(y_a)}v(y_a)(i\varphi(y_a))^{[p]}e^{ib\varphi(y_a)}(e^{ih\varphi(y_a)}-1)
\end{align*}
and hence using the inequality $|e^{ix}-1|\le |x|^\delta$ for all $x\in\R$, $\delta\in[0,1]$,
\begin{align*}
& |(\{\hat R_0^{([p])}(i(b+h))- \hat R_0^{([p])}(ib)\}v)(y)|
 \le C_1\sum_{a\in\alpha}\mu(a)|v|_\infty \varphi(y_a)^{[p]}|e^{ih\varphi(y_a)}-1|
\\ & \qquad \le C_1|v|_\infty \sum_{a\in\alpha}\mu(a)\varphi(y_a)^{[p]}|h|^{p-[p]}\varphi(y_a)^{p-[p]}
= C_1|v|_\infty |h|^{p-[p]}\sum_{a\in\alpha}\mu(a)\varphi(y_a)^p
\\ &  \qquad  \le C_1(C_2+1)^p|v|_\infty |h|^{p-[p]}|\varphi|_p^p.
\end{align*}
Hence $d_p\hat R_0(ib)\ll |\varphi|_p^p$.

\vspace{1ex}
\noindent(b)   
We give the details for $p$ not an integer, and $\epsilon<p-[p]$.
Set $\theta=\theta_0^\epsilon$.
Let $j\in\{0,1,\dots,[p]\}$ and write $(\hat R_0^{(j)}(ib)v)(y)
-(\hat R_0^{(j)}(ib)v)(y')= I+II+III+IV$, where
\begin{align*}
I & 
=\sum_{a\in\alpha}(e^{g(y_a)}-e^{g(y_a')})v(y_a)(i\varphi(y_a))^je^{ib\varphi(y_a)}, \\
II 
& =\sum_{a\in\alpha}e^{g(y_a')}(v(y_a)-v(y_a'))(i\varphi(y_a))^je^{ib\varphi(y_a)}, \\
III & 
=\sum_{a\in\alpha}e^{g(y_a')}v(y_a')i^j(\varphi(y_a)^j-\varphi(y_a')^j)e^{ib\varphi(y_a)}, \\
IV & 
=\sum_{a\in\alpha}e^{g(y_a')}v(y_a')(i\varphi(y_a'))^j(e^{ib\varphi(y_a)}-e^{ib\varphi(y_a')}).
\end{align*}
We have 
\begin{align*}
|I| & \le C_1\sum_{a\in\alpha}\mu(a)d_\theta(y,y')|v|_\infty\varphi(y_a)^j\le C_1(C_2+1)^j|v|_\infty|\varphi|_j^jd_\theta(y,y'), \\
|II| & \le C_1\sum_{a\in\alpha}\mu(a)|v|_\theta d_\theta(y,y')\varphi(y_a)^j\le C_1(C_2+1)^j|v|_\theta|\varphi|_j^jd_\theta(y,y'), \\
|III| & \le C_1j\sum_{a\in\alpha}\mu(a)|v|_\infty\varphi(y_a)^{j-1}|1_a\varphi|_\theta d_\theta(y,y')\le C_1(C_2+1)^jj|v|_\infty|\varphi|_j^jd_\theta(y,y'), \\
|IV| & \le C_1\sum_{a\in\alpha}\mu(a)|v|_\infty \varphi(y'_a)^j|b|^\epsilon|1_a\varphi|_{\theta_0}^\epsilon d_{\theta_0}(y,y')^\epsilon
\le C_1(C_2+1)^{j+\epsilon}|b|^\epsilon|v|_\infty|\varphi|_{j+\epsilon}^{j+\epsilon}d_\theta(y,y'),
\end{align*}
so that
\begin{align} \label{eq-dj}
|\hat R_0^{(j)}(ib)v|_\theta\ll (1+|b|^\epsilon)|\varphi|_p^p\|v\|_\theta.
\end{align}

Finally, 
\begin{align*}
& (\{\hat R_0^{([p])}(i(b+h)) -\hat R_0^{([p])}(ib)\}v)(y)
-(\{\hat R_0^{([p])}(i(b+h)) -\hat R_0^{([p])}(ib)\}v)(y')
\\ & =
\sum_{a\in\alpha}e^{g(y_a)}v(y_a)(i\varphi(y_a))^{[p]}e^{ib\varphi(y_a)}(e^{ih\varphi(y_a)}-1)
\\ & \qquad -\sum_{a\in\alpha}e^{g(y_a')}v(y_a')(i\varphi(y_a'))^{[p]}e^{ib\varphi(y_a')}(e^{ih\varphi(y_a')}-1)
\\ & = I+II+III+IV+V, 
\end{align*}
where
\begin{align*}
I & = \sum_{a\in\alpha}(e^{g(y_a)}-e^{g(y_a')})v(y_a)(i\varphi(y_a))^{[p]}e^{ib\varphi(y_a)}(e^{ih\varphi(y_a)}-1), \\
II & = \sum_{a\in\alpha}e^{g(y_a')}(v(y_a)-v(y_a'))(i\varphi(y_a))^{[p]}e^{ib\varphi(y_a)}(e^{ih\varphi(y_a)}-1), \\
III & = \sum_{a\in\alpha}e^{g(y_a')}v(y_a')i^{[p]}(\varphi(y_a)^{[p]}-\varphi(y_a')^{[p]})e^{ib\varphi(y_a)}(e^{ih\varphi(y_a)}-1), \\
IV & = \sum_{a\in\alpha}e^{g(y_a')}v(y_a')(i\varphi(y_a'))^{[p]}(e^{ib\varphi(y_a)}-e^{ib\varphi(y_a')})(e^{ih\varphi(y_a)}-1), \\
V & = \sum_{a\in\alpha}e^{g(y_a')}v(y_a')(i\varphi(y_a'))^{[p]}e^{ib\varphi(y_a')}(e^{ih\varphi(y_a)}-e^{ih\varphi(y_a')}).
\end{align*}
These terms are estimated using the same techniques as the previous ones that arose in this proof.   For example, we use the inequalities
$|e^{ib\varphi(y_a)}-e^{ib\varphi(y_a')}|\le |b|^\epsilon|1_a\varphi|_{\theta_0}^\epsilon d_\theta(y,y')$
and $|e^{ih\varphi(y_a)}-1|\le |h|^{p-[p]-\epsilon}\varphi(y_a)^{p-[p]-\epsilon}$ to obtain
$|IV|\le C_1(C_2+1)^p|b|^\epsilon|v|_\infty |\varphi|_p^p|h|^{p-[p]-\epsilon}d_\theta(y,y')$, 
and we use the inequality $|e^{ih\varphi(y_a)}-e^{ih\varphi(y_a')}|\le |h|^{p-[p]}|1_a\varphi|_{\theta_0}^{p-[p]}d_\theta(y,y')$ to obtain
$|V|\le C_1(C_2+1)^p|v|_\infty |\varphi|_p^p|h|^{p-[p]}d_\theta(y,y')$.
Altogether, we obtain
\begin{align} \label{eq-dp}
|\{\hat R_0^{([p])}(i(b+h))-\hat R_0^{([p])}(ib)\}v|_\theta\ll
 (1+|b|^\epsilon)\|v\|_\theta |\varphi|_p^p|h^{p-[p]-\epsilon}.
\end{align}

The estimates~\eqref{eq-dj} and~\eqref{eq-dp} combined with the estimates in~(a) yield the required result.
\end{proof}

\begin{prop} \label{prop-DolgR}
Suppose that $\varphi\in L^p$ for some $p>0$, and let $\epsilon>0$.
Viewed as a family of operators on $F_\theta(Y)$,
$b\mapsto (I-\hat R_0(ib))^{-1}$ is $C^{p-\epsilon}$ 
and
there exist $C,A>0$ such that
$d_{p-\epsilon}(I-\hat R_0(ib))^{-1}\le C|b|^{A}$ for all $|b|>1$.
\end{prop}

\begin{proof}  
Again, we give the details for $p$ not an integer, and $\epsilon<p-[p]$.

A straightforward induction argument shows that 
$\frac{d^j}{db^j}(I-\hat R_0(ib))^{-1}$ is a finite linear combination of finite products of factors $\hat F$ where
\[
\hat F\in\{(I-\hat R_0)^{-1}, \quad \hat R_0^{(k)},\enspace k=1,\ldots,j\},
\]
for each $j\le p$.   For each choice of $\hat F$, there exists $A_1>0$ such that
$\|\hat F(ib)\|\ll |b|^{A_1}$,
and moreover
$d_{p-[p]-\epsilon}\hat F(ib)\ll |b|^{A_1}$, by Proposition~\ref{prop-smoothR} and Lemma~\ref{lem-approx}.
The required estimate is an immediate consequence.
\end{proof}

\section{First main lemma}
\label{sec-fml}

In this section we prove the following counterpart of the ``first main lemma'' of~\cite{Sarig02,Gouezel-sharp}.
We view  $\hat B(s)=s(I-\hat R_0(s))^{-1}$ as a family of operators on $F_\theta(Y)$
.
Inverse Laplace transforms will be computed by moving the contour of integration
to the imaginary axis (the functions in question are nonsingular on $\overline \H$) and hence can be viewed as inverse Fourier transforms.  
Recall that we defined $\mathcal{R}(a(t))$ to be the space of Laplace transforms of maps $S:[0,\infty)\to\mathcal{B}$ with $\|S(t)\|\le Ca(t)$.
We now enlarge the
definition of $\mathcal{R}(a(t))$ to include (operator-valued) functions defined on the 
imaginary axis with inverse Fourier transform dominated by $a(t)$.

Also, we write $\mathcal{R}(1/t^{p-})$ to denote domination by $1/t^q$ for all $q<p$.   Similarly, an (operator-valued) function is $C^{p-}$ if it is $C^q$ for all $q<p$.

\begin{lemma} \label{lem-fml}
Suppose that $\varphi\in L^p$ for some $p>1$.
Let $\psi:\R\to\R$ be $C^\infty$ with
$\supp\psi\subset [-r,r]$ where $r\in(0,1)$ is sufficiently small
and such that $\psi\equiv1$ on a neighborhood of $0$.
Then $\psi\hat B\in\mathcal{R}(1/t^{p-})$.
\end{lemma}

First we derive an elementary calculus estimate.

\begin{prop} \label{prop-s}
Let $s(y)=(e^{iy}-1)/y$.
For any $n\ge0$, there exists a constant $C>0$ such that
$|s^{(n)}(y)|\le C$ and
$|s^{(n)}(y)|\le C/|y|$ for all $y\in\R$.
\end{prop}

\begin{proof}
Define the analytic functions $q_n,r_n:\C\to\C$ for $n\ge1$,
\[
q_n(z)=e^z-\sum_{j=0}^{n-1}\frac{z^j}{j!}, \quad r_n(z)=\frac{q_n(z)}{z^n}.
\]
By Taylor's theorem, there exists $\xi$ between $0$ and $z$ such that
\[
q_n(z)
=\sum_{j=0}^{n-1}q_n^{(j)}(0)z^j/j! + q_n^{(n)}(\xi)z^n/n!
= e^{\xi}z^n/n!,
\]
so that $|q_n(iy)|\le |y|^n/n!\,$
Similarly,
\[
q_n(z)
=\sum_{j=0}^{n-2}q_n^{(j)}(0)z^j/j! + q_n^{(n-1)}(\xi)z^{n-1}/(n-1)!
= (e^{\xi}-1)z^{n-1}/(n-1)!,
\]
so that $|q_n(iy)|\le |y|^{n-1}/(n-1)!$

Next, note by induction that 
\[
r_1^{(n)}\in \R\{e^z/z,\,e^z/z^2,\dots,e^z/z^n,\,(e^z-1)/z^{n+1}\}.
\]
But $e^z/z^k-r_k\in\R\{1/z,\dots,1/z^k\}$.
Hence there exist constants $a_1,\dots,a_{n+1}$ and a polynomial $p$ of degree at most $n$ such that
\[
r_1^{(n)}(z)=\sum_{k=1}^{n+1}a_kr_k(z)+p(z)/z^{n+1}.
\]
Since all terms in this identity are analytic with the possible exception of the last one, we deduce that $p\equiv0$.
Hence
\[
r_1^{(n)}(z)=\sum_{k=1}^{n+1}a_kr_k(z)
=\sum_{k=1}^{n+1}a_kq_k(z)/z^k.
\]

Since $s(y)=ir_1(iy)$,
the result follows for each fixed $n$ by substituting in the estimates for $q_k$.
\end{proof}

\begin{lemma} \label{lem-hatR}
Suppose that $\varphi\in L^p$ for some $p>0$, and let $\epsilon>0$.
Then there exists $\theta\in(0,1)$ such that viewed as an operator on
$F_\theta(Y)$,
\[
\chi(b)\frac{\hat R_0(ib)-\hat R_0(0)}{b}\in\mathcal{R}(1/t^{p-\epsilon}),
\]
for all $C^\infty$ functions $\chi:\R\to[0,1]$ with $\supp \chi\subset[-3,3]$.
\end{lemma}

\begin{proof}
Let $k\ge0$ such that $p\in (k,k+1]$ and
let $\epsilon\in(0,p-k)$.  Set
$\theta=\theta_0^\epsilon$.

Let $S(t)$ denote the inverse Fourier transform of $\chi(b)(\hat R_0(ib)-\hat R_0(0))/b$.
We show that
 $\|S(t)\|_\theta\ll |\varphi|_p^p|t|^{-(p-\epsilon)}$.
Let $v\in F_\theta(Y)$. By~\eqref{eq-Rn},
\begin{align*}
((\hat R_0(ib)-\hat R_0(0)) v)(y) & =
\sum_{a\in \alpha}e^{g(y_a)}v(y_a) (e^{ib\varphi(y_a)}-1).
\end{align*}
Hence 
\begin{align*}
(S(t)v)(y) = \sum_{a\in\alpha} e^{g(y_a)}v(y_a)
\int_{-3}^3 r(b,\varphi(y_a))e^{ibt}\,db
\end{align*}
where 
\[
r(b,x)=\chi(b) (e^{ibx}-1)/b=\chi(b)xs(xb), \quad
s(y)=(e^{iy}-1)/y.
\]
Let $n\ge0$.  By Proposition~\ref{prop-s},
$|s^{(n)}(y)|\ll 1$ and $|s^{(n)}(y)|\ll |y|^{-1}$.
Hence
\begin{align} \label{eq-s}
|s^{(n)}(y)|\ll |y|^{-1}\min\{1,|y|\}\le |y|^{-(1-\epsilon)},
\end{align} 
for all $y\in\R$.
It follows from~\eqref{eq-s} that
$|\partial^nr(b,x)/\partial b^n|\ll x^{n+\epsilon}|b|^{-(1-\epsilon)}$
for all $x\ge1$, $b\in\R$.

Integrating by parts $n$ times,
\begin{align*} 
\Bigl|\int_{-3}^3 r(b,x)e^{ibt}\,db\Bigr| & 
= |t|^{-n}
\Bigl|\int_{-3}^3 \partial^n r(b,x)/\partial b^n\,e^{ibt}\,db\Bigr| 
\\ & \ll x^{n+\epsilon}|t|^{-n}\int_{-3}^3 |b|^{-(1-\epsilon)} \,db
\ll x^{n+\epsilon}|t|^{-n}.
\end{align*}
Applying this with $n=k$ and $n=k+1$,
\begin{align*} 
\Bigl|\int_{-3}^3 r(b,x)e^{ibt}\,db\Bigr| & 
\ll
\min\{ x^{k+\epsilon}|t|^{-k},
x^{k+1+\epsilon}|t|^{-(k+1)}\}  
 = 
x^{k+\epsilon}|t|^{-k}
\min\{ 1, x|t|^{-1}\}
 \\ & \le  x^{k+\epsilon+\delta}|t|^{-(k+\delta)},
\end{align*}
for all $\delta\in[0,1]$.   Taking $\delta=p-k-\epsilon$, we obtain 
$|\int_{-3}^3 r(b,x)e^{ibt}\,db| \le x^p |t|^{-(p-\epsilon)}$.
In particular,
\begin{align} \label{eq-r}
\Bigl|\int_{-3}^3 r(b,\varphi(y_a))e^{ibt}\,db\Bigr| & \ll
\varphi(y_a)^p|t|^{-(p-\epsilon)}. 
\end{align}
Also, $r(b,x)-r(b,x')=\chi(b)e^{ibx}(x-x')s(b(x'-x))$ and it follows from~\eqref{eq-s}
that
\begin{align*}
& |(\partial^n r(b,x)/\partial b^n-
(\partial^n r(b,x')/\partial b^n|
 \ll \sum_{j=0}^n x^{n-j}|x-x'|^{j+\epsilon}
|b|^{-(1-\epsilon)},
\end{align*}
for all $x,x'\ge1$, $b\in\R$.
Hence using the estimate $|1_a\varphi|_\theta\le C_2\inf_a\varphi\le C_2|1_a\varphi|_\infty$ from assumption~(A1),
\begin{align*}
|(\partial^n r(b,\varphi(y_a))/\partial b^n-
(\partial^n r(b,\varphi(y'_a))/\partial b^n|
& \ll 
|1_a\varphi|_\infty^{n+\epsilon}d_{\theta_0}(y,y')^{\epsilon}|b|^{-(1-\epsilon)}
\\ & =|1_a\varphi|_\infty^{n+\epsilon}d_{\theta}(y,y')|b|^{-(1-\epsilon)}.
\end{align*}
Integrating by parts $k$ and $k+1$ times,
\begin{align} \label{eq-rr} \nonumber
\Bigl|\int_{-3}^3 (r(b,\varphi(y_a))-r(b,\varphi(y'_a)))e^{ibt}\,db\Bigr| & \ll
|1_a\varphi|_\infty^{k+\epsilon}d_\theta(y,y')|t|^{-k}
\min\{ 1, |1_a\varphi|_\infty|t|^{-1}\} \\ & \le
|1_a\varphi|_\infty^pd_\theta(y,y')|t|^{-(p-\epsilon)}.
\end{align}

We are now ready to estimate $\|S(t)\|$.
Using~\eqref{eq-r}, we obtain that
\begin{align*}
|S(t)v)(y)| & \ll \sum_{a\in\alpha} \mu(a)|v|_\infty\Bigl|\int_{-3}^3 r(b,\varphi(y_a))e^{ibt}\,db\Bigr| \ll
|v|_\infty \sum_{a\in \alpha}\mu(a)
\varphi(y_a)^p|t|^{-(p-\epsilon)},
\end{align*}
and so
\begin{align} \label{eq-infty}
|S(t)|_\infty \ll |\varphi|_p^p|t|^{-(p-\epsilon)}|v|_\infty.
\end{align}

Next, 
\begin{align*}
(S(t)v)(y)- (S(t)v)(y')= I+II+III,
\end{align*}
where
 \begin{align*}
I & = \sum_{a\in\alpha} (e^{g(y_a)}-e^{g(y'_a)})v(y_a) \int_{-3}^3 r(b,\varphi(y_a))e^{ibt}\,db, \\
II & = \sum_{a\in\alpha} e^{g(y'_a)}(v(y_a)-v(y'_a)) \int_{-3}^3 r(b,\varphi(y_a))e^{ibt}\,db, \\
III & = \sum_{a\in\alpha} e^{g(y'_a)}v(y'_a) \int_{-3}^3 (r(b,\varphi(y_a))-r(b,\varphi(y'_a))e^{ibt}\,db.
\end{align*}
The first two terms are estimated using~\eqref{eq-r}:
\begin{align*}
|I| & \ll \sum_{a\in\alpha} \mu(a)d_\theta(y,y')|v|_\infty \varphi(y_a)^p|t|^{-(p-\epsilon)} \ll |\varphi|_p^p|t|^{-(p-\epsilon)}|v|_\infty d_\theta(y,y'), \\
|II| & \ll \sum_{a\in\alpha} \mu(a)|v|_\theta d_\theta(y,y')\varphi(y_a)^p|t|^{-(p-\epsilon)} \ll |\varphi|_p^p|t|^{-(p-\epsilon)}|v|_\theta d_\theta(y,y').
\end{align*}
The third term is estimated using~\eqref{eq-rr}:
\begin{align*}
|III| & \ll \sum_{a\in\alpha} \mu(a)|v|_\infty |1_a\varphi|_\infty^p d_\theta(y,y')|t|^{-(p-\epsilon)}
\ll |\varphi|_p^p|t|^{-(p-\epsilon)}|v|_\infty d_\theta(y,y').
\end{align*}
Combining the estimates for $I,II,III$ we obtain
\begin{align} \label{eq-Lip}
|S(t)v|_\theta\ll |\varphi|_p^p|t|^{-(p-\epsilon)}\|v\|_\theta.
\end{align}
By~\eqref{eq-infty} and~\eqref{eq-Lip}, 
$\|S(t)\|_\theta \ll |\varphi|_p^p|t|^{-(p-\epsilon)}$ as required.
\end{proof}

\begin{prop}  \label{prop-tildeR}
Suppose that $\varphi\in L^p$ for some $p>0$, and let $\epsilon>0$.
Let $\delta>0$.  For all $r>0$ sufficiently small, there exists 
a $C^{p-\epsilon}$ family $b\mapsto \tilde R_0(b)$ with a $C^{p-\epsilon}$ family of simple eigenvalues
$\tilde\lambda(b)\in\{z\in\C:|z-1|<\delta\}$ such that
\begin{itemize}
\item[(a)] $\tilde R_0(b)\equiv \hat R_0(ib)$ for $|b|\le r$.
\item[(b)] $\tilde R_0(b)\equiv \hat R_0(0)$ and $\tilde\lambda(b)\equiv1$ for $|b|\ge 2$.
\item[(c)] $\|\tilde R_0(b)-\hat R_0(0)\|_\theta<\delta$ for all $b\in\R$.
\item[(d)] 
For all $b\in\R$,
the spectrum of $\tilde R_0(b)$ consists of $\tilde\lambda(b)$ together with a subset of $\{z:|z-1|\ge 3\delta\}$.
\item[(e)] $(1-\tilde\lambda(b))/b$ is bounded away from zero on $[-r,r]$.
\end{itemize}
\end{prop}

\begin{proof}
Recall that $\hat R_0(0)$ has a simple eigenvalue at $1$.   Also there
exists $\delta_0>0$ such that the remainder of the spectrum lies outside the disk $\{|z-1|<\delta_0\}$.  We suppose without loss that $\delta<\delta_0/3$.

Choose $\delta_1\in(0,\delta)$ with the property that if $A$ is an operator
on $F_\theta(Y)$ and $\|A-\hat R_0(0)\|_\theta<\delta_1$, then the spectrum of
$A$ consists of a simple eigenvalue within distance $\delta$ of $1$ and the remainder of the spectrum lies outside the disk $\{|z-1|<3\delta\}$.

By Proposition~\ref{prop-smoothR}
there is a $C^{p-}$ family $b\mapsto \lambda(b)$, defined for $b$ sufficiently small, consisting of simple eigenvalues for $\hat R_0(ib)$ with $\lambda(0)=1$.  Moreover, $\lambda(b)=1+ib\bv+o(b)$ as $b\to0$.

Choose $r_0\in(0,1)$ small so that $[-r_0,r_0]$ lies inside the domain of definition of this $C^{p-}$ family and such that $(1-\lambda(b))/b$ is bounded away from zero on $[-r_0,r_0]$.
Fix $r\in(0,r_0)$.

Let $\psi_1,\psi_2,\psi_3:\R\to[0,1]$ be even $C^\infty$ functions such that
$\psi_1+\psi_2+\psi_3\equiv1$ and such that restricted to $[0,\infty)$,
\begin{itemize}
\item[] $\psi_1\equiv1$ on $[0,r]$, $\supp\psi_1\subset[0,r_0]$.
\item[] $\psi_2\equiv1$ on $[r_0,1]$, $\supp\psi_2\subset[r,2]$.
\item[] $\psi_3\equiv1$ on $[2,\infty)$, $\supp\psi_2\subset[1,\infty)$.
\end{itemize}
Define the $C^{p-}$ family of operators
\[
\tilde R_0(b)=\psi_1(b)\hat R_0(ib) +\psi_2(b)\hat R_0(i\sgn(b)r_0) +\psi_3(b)\hat R_0(0).
\]
For $b\ge0$ we have 
\[
\tilde R_0(b)  = \begin{cases} 
\hat R_0(ib), & b\in[0,r] \\
\hat R_0(ir_0)+\psi_1(b)(\hat R_0(ib)-\hat R_0(ir_0)), & b\in[r,r_0] \\
\hat R_0(ir_0), & b\in[r_0,1] \\
\hat R_0(0)+\psi_2(b)(\hat R_0(ir_0)-\hat R_0(0)), & b\in[1,2] \\
\hat R_0(0), & b\in[2,\infty) \end{cases}\quad.
\]
Shrinking $r_0$ if necessary, we can ensure that 
\[
\|\hat R_0(ib)-\hat R_0(0)\|_\theta<\delta_1/2\quad\text{for all $b\in[0,r_0]$},
\quad
\|\tilde R_0(b)-\hat R_0(0)\|_\theta<\delta_1 \quad\text{for all $b\in[1,2]$.}
\]
Then choosing $r$ sufficiently close to $r_0$, we can ensure that 
$\|\tilde R_0(b)-\hat R_0(ir_0)\|_\theta<\delta_1/2$ for all $b\in[r,r_0]$.   
Altogether, we have that $\|\tilde R_0(b)-\hat R_0(0)\|_\theta<\delta_1$ for all $b\ge0$.  A similar picture holds for $b\le 0$ and so we obtain that \[
\|\tilde R_0(b)-\hat R_0(0)\|_\theta<\delta_1<\delta, \quad\text{for all $b\in\R$}.
\]
This verifies condition~(c).  Moreover, by definition of $\delta_1$ we obtain the required spectral properties for $\tilde R_0$, namely the family of simple eigenvalues $\tilde\lambda$ (which is $C^{p-}$ by standard perturbation theory) together with the estimate in condition~(d).
Finally, we observe that
properties (a,b,e) are immediate consequences of the construction.
\end{proof}

Let $\tilde P$ be the spectral projection corresponding to $\tilde\lambda$.
By Proposition~\ref{prop-tildeR}, $b\mapsto\tilde P(b)$ is $C^{p-}$.

\begin{prop}  \label{prop-lambda}
Suppose that $\varphi\in L^p$ for some $p>0$, and let $\epsilon>0$.
For $\delta>0$ small enough in Proposition~\ref{prop-tildeR},
$\BIG\frac{1-\tilde\lambda(b)}{b}\in\mathcal{R}(1/t^{p-\epsilon})$.
\end{prop}

\begin{proof}
Recall the formula
\begin{align} \label{eq-formula}
\frac{1-\tilde\lambda(b)}{b}\tilde P(b)=
\frac{\tilde R_0(0)-\tilde R_0(b)}{b}\tilde P(b)
+(I-\tilde R_0(0))\frac{\tilde P(b)-\tilde P(0)}{b}.
\end{align}
Let $\chi:\R\to[0,1]$ be a $C^\infty$ function supported in $[-3,3]$
with $\chi\equiv1$ on $[-2,2]$.  By Proposition~\ref{prop-tildeR}(a,b),
\[
\frac{\tilde R_0(b)-\tilde R_0(0)}{b}=
\chi(b)\frac{\tilde R_0(b)-\hat R_0(0)}{b}=
\chi(b)\frac{\tilde R_0(b)-\hat R_0(ib)}{b}+
\chi(b)\frac{\hat R_0(ib)-\hat R_0(0)}{b}.
\]
The first term on the RHS vanishes near zero by Proposition~\ref{prop-tildeR}(a) and hence is $C^{p-}$.
Also it is compactly supported and so lies in $\mathcal{R}(1/t^{p-})$.
The second term on the RHS lies in $\mathcal{R}(1/t^{p-})$ by Lemma~\ref{lem-hatR}.
We deduce that
$\BIG\frac{\tilde R_0(b)-\tilde R_0(0)}{b}\in\mathcal{R}(1/t^{p-})$.
Moreover,
\[
\frac{\tilde R_0(b)-\tilde R_0(0)}{b}\tilde P(b)=
\frac{\tilde R_0(b)-\tilde R_0(0)}{b}\chi(b)\tilde P(b),
\]
where $\chi\tilde P$ is $C^{p-}$ and compactly supported.   It follows that
$\BIG\frac{\tilde R_0(b)-\tilde R_0(0)}{b}\tilde P(b)\in \mathcal{R}(1/t^{p-})$.

Let $\Gamma$ be the circle of radius $2\delta$ around $1$.
By Proposition~\ref{prop-tildeR}(d), 
$\tilde P(b)=(1/2\pi i)\int_\Gamma(\xi-\tilde R_0(b))^{-1}\,d\xi$.  
By Proposition~\ref{prop-tildeR}(b),
$\tilde P(b)=\tilde P(0)$ for $|b|\ge2$.  Hence
\[
\frac{\tilde P(b)-\tilde P(0)}{b}
=\chi(b)\frac{\tilde P(b)-\tilde P(0)}{b}
=\frac{1}{2\pi i}\int_\Gamma G_1(b,\xi)G_2(b)G_3(\xi)\,db,
\]
where
\[
G_1(b,\xi)=\chi(b)(\xi-\tilde R_0(b))^{-1}, \quad
G_2(b)=\frac{\tilde R_0(b)-\tilde R_0(0)}{b}, \quad
G_3(\xi)=(\xi-\tilde R_0(0))^{-1}.
\]
We already showed that $G_2\in\mathcal{R}(1/t^{p-})$.
Also $b\mapsto G_1(b,\xi)$ is compactly supported and
$C^{p-}$ uniformly in $\xi$.
Hence 
$G_1(b,\xi)G_2(b)G_3(\xi)\in\mathcal{R}(1/t^{p-})$ with norms uniform in $\xi$ and so
$\BIG\frac{\tilde P(b)-\tilde P(0)}{b}\in\mathcal{R}(1/t^{p-})$.

The above arguments together with~\eqref{eq-formula} imply that
$\BIG\frac{1-\tilde\lambda(b)}{b}\tilde P(b)\in\mathcal{R}(1/t^{p-})$.
Hence
$\BIG\frac{1-\tilde\lambda(b)}{b}u(\tilde P(b))\in\mathcal{R}(1/t^{p-})$ for any
bounded linear functional 
$u:F_\theta(Y)\to\R$.  Choose $u$ so that
$u(\tilde P(0))\neq0$.  
By Proposition~\ref{prop-tildeR}(c), we can ensure
that $u(\tilde P(b))$ is bounded away from zero for all $b$.
Then $\chi/u(\tilde P)$ is compactly supported
and $C^{p-}$, so $\chi/u(\tilde P)\in\mathcal{R}(1/t^{p-})$.
By Proposition~\ref{prop-tildeR}(b), 
\[
\frac{1-\tilde\lambda(b)}{b}=
\frac{1-\tilde\lambda(b)}{b}\chi(b)=
\frac{1-\tilde\lambda(b)}{b}u(\tilde P(b))\,\frac{\chi(b)}{u(\tilde P(b))}.
\]
Hence $\BIG\frac{1-\tilde\lambda(b)}{b}\in\mathcal{R}(1/t^{p-})$ as required.
\end{proof}

\begin{pfof}{Lemma~\ref{lem-fml}}
By Proposition~\ref{prop-tildeR}(a), $\psi\hat B=\psi\tilde B$
where $\tilde B(b)=b(I-\tilde R_0(b))^{-1}$.
Write
\[
\tilde B(b)=((1-\tilde\lambda(b))/b)^{-1}\tilde P(b)
+b(I-\tilde R_0(b))^{-1}(I-\tilde P(b)).
\]
The second term is $C^{p-}$ and so
lies in $\mathcal{R}(1/t^{p-})$ when multiplied by $\psi$.
Hence it remains to show that
$\psi(b)((1-\tilde\lambda(b))/b)^{-1}\tilde P(b)\in\mathcal{R}(1/t^{p-})$.

Let $\chi$ be a compactly supported $C^\infty$ function with $\chi\equiv 1$ on the support of $\psi$.  Then 
\[
\psi(b)((1-\tilde\lambda(b))/b)^{-1}\tilde P(b)=
\psi(b)((1-\tilde\lambda(b))/b)^{-1}\chi(b)\tilde P(b),
\]
and $\chi\tilde P\in \mathcal{R}(1/t^{p-})$.
Hence it remains to show that
$\psi(b)((1-\tilde\lambda(b))/b)^{-1}\in\mathcal{R}(1/t^{p-})$.

Now $\psi$ is a compactly supported element of $\mathcal{R}(1/t^{p-})$.
By Proposition~\ref{prop-lambda},
$(1-\tilde\lambda(b))/b\in\mathcal{R}(1/t^{p-})$.
Moreover, 
$(1-\tilde\lambda(b))/b$ is bounded away from zero on the support of $\psi$
by Proposition~\ref{prop-tildeR}(e).
By Lemma~\ref{lem-W}, 
there exists $g\in\mathcal{R}(1/t^{p-})$ such that
$\psi(b)=g(b)(1-\tilde\lambda(b))/b$.
Hence
$\psi(b)((1-\tilde\lambda(b))/b)^{-1}=g(b)\in\mathcal{R}(1/t^{p-})$,
as required.~
\end{pfof}

\section{Proof of Lemma~\ref{lem-T4}}
\label{sec-T4}

Finally, we deal with the term 
\begin{align*}
& \hat\rho_4  =(1/\bv)\int_{\tilde Y}\hat U\hat T_4v\,w\,d\mu, \qquad
 (\hat T_4v)(y,u)=(\hat T_{0,4}v^u)(y), \\
& \hat T_{0,4}= (I-P_\varphi \hat D^*)^{-1}(I-P_\varphi \hat C^*(0)) \hat K^*, \qquad
 \hat K^*= \hat H^*(I-\hat C^*\hat B).
\end{align*}

In Section~\ref{sec-smooth}, we introduced the notation $d_p\hat S$.
We recall the following basic result.

\begin{prop} \label{prop-S}
\begin{itemize}
\item[(a)]
Suppose that the family $b\mapsto \hat S(ib)$ is $C^p$ for some $p>0$ and that
there is a constant $C>0$ such that
$d_p\hat S(ib)\le C|b|^{-2}$ for $|b|>1$.
Then $\hat S\in\mathcal{R}(1/t^p)$.
\item[(b)]  Suppose that $g:\R\to\R$ is $C^\infty$, such that $g\equiv0$ in a neighborhood of $0$, and $g(b)\equiv1$ for $|b|$ sufficiently large.
Let $m\ge1$.  Then $g(b)/b^m\in\mathcal{R}(1/t^p)$ for all $p>0$.
\end{itemize}
\end{prop}

\begin{proof} 
(a) 
For $p$ an integer,
$S(t)=\int_{-\infty}^\infty e^{ibt} \hat S(ib)\,db=
t^{-p}\int_{-\infty}^\infty e^{ibt} \hat S^{(p)}(ib)\,db$ so that
$|S(t)|\ll t^{-p}\int_{-\infty}^\infty (1+|b|^{-2})\,db\ll t^{-p}$.

For $p$ not an integer, we still have
$S(t)= 
t^{-[p]}\int_{-\infty}^\infty e^{ibt} \hat S^{([p])}(ib)\,db
=-t^{-[p]}\int_{-\infty}^\infty e^{ibt} \hat S^{([p])}(i(b+\pi/t))\,db$
and so 
\begin{align*}
2|S(t)| & \le 
t^{-[p]}\int_{-\infty}^\infty 
|\hat S^{([p])}(i(b+\pi/t)) -\hat S^{([p])}(ib)| \,db \\ & 
\ll t^{-[p]}\int_{-\infty}^\infty d_p\hat S(ib) |t|^{-(p-[p])}\,db
\ll |t|^{-p}\int_{-\infty}^\infty (1+|b|)^{-2}\ll |t|^{-p},
\end{align*}
as required.

\vspace{1ex}
\noindent(b) If $m\ge2$, then this is immediate from part (a).
For $m=1$, note that
\begin{align*}
\Bigl|\lim_{L\to\infty}\int_{-L}^L e^{ibt}g(b)/b\,db\Bigr|=t^{-1}\Bigl|\int_{-\infty}^\infty e^{ibt}h(b)\,db\Bigr|,
\end{align*}
where $h(b)=(g(b)/b)'=(bg'(b)-g(b))/b^2$ satisfies the conditions of part (a).  (Recall that $g'\equiv0$ for $b$ large.)
\end{proof}

\begin{prop}  \label{prop-H}
$\hat H^*:F_\theta(Y)\to F_\theta(Y)$ is analytic on a neighborhood of $\overline\H$.
\end{prop}

\begin{proof}
This is standard since the flow under the truncated roof function
$\varphi^*$ is uniformly expanding:
$\hat T_0^*$ has a meromorphic extension across the imaginary axis with a simple pole at zero, and
$\hat H^*(s)=\hat T_0^*(s)-s^{-1}P_\varphi^*$ is analytic on a neighborhood of $\overline\H$.~
\end{proof}

Let $\psi:\R\to[0,1]$ be as in Lemma~\ref{lem-fml}.
Recall that $\psi$ is $C^\infty$ with $\supp\psi\in[-1,1]$, and
$\psi\equiv1$ on a neighborhood of $0$.  We have the following consequence of
Lemma~\ref{lem-hatR}.

\begin{cor} \label{cor-C}
Suppose that $\varphi\in L^p$ for some $p>0$, and let $\epsilon>0$.
Then $\psi\hat C^*\in \mathcal{R}(1/t^{p-\epsilon})$.
\end{cor}

\begin{proof}  
By Lemma~\ref{lem-hatR}, 
$\psi(b)\BIG\frac{\hat R_0(ib)-\hat R_0(0)}{b}\in\mathcal{R}(1/t^{p-})$.
Since $\varphi^*\in L^\infty$, it follows from Lemma~\ref{lem-hatR}
that
$\psi(b)\BIG\frac{\hat R_0^*(ib)-\hat R_0^*(0)}{b}\in\mathcal{R}(1/t^q)$ for all $q$.
But
$\hat C^*(ib)=
\BIG\frac{\hat R_0^*(ib)-\hat R_0^*(0)}{b}- \frac{\hat R_0(ib)-\hat R_0(0)}{b}$
so the result follows.
\end{proof}

\begin{prop} \label{prop-a1}
$\psi(b)^3\hat\rho_4(ib)\in\mathcal{R}(\|v\|_\theta |w|_\infty 1/t^{\beta-})$.
\end{prop}

\begin{proof} 
Regard the operators $\hat H^*$, $\hat C^*$, $\hat B$ in the expression $\hat K^*=\hat H^*(I-\hat C^*\hat B)$ as operators 
on $F_\theta(Y)$.
By Lemma~\ref{lem-fml}, $\psi\hat B\in\mathcal{R}(1/t^{\beta-})$.   
Also, $\psi\hat C^*\in\mathcal{R}(1/t^{\beta-})$ by Corollary~\ref{cor-C}.
By Proposition~\ref{prop-H}, $\psi\hat H^*$ is $C^\infty$ and this together with Proposition~\ref{prop-S}(a) implies that
$\psi\hat H^*\in\mathcal{R}(1/t^p)$ for all $p$.
Hence
$\psi^3\hat K^*=\psi^3\hat H^*-(\psi\hat H^*)(\psi\hat C^*)(\psi\hat B)\in\mathcal{R}(1/t^{\beta-})$.  
It follows that
$\psi^3\hat K^*:F_\theta(Y)\to L^\infty(Y)$ lies in
$\mathcal{R}(1/t^{\beta-})$.

By Proposition~\ref{prop-C}, $\hat C^*(0)$ is a bounded operator on $L^\infty(Y)$.
Hence by Proposition~\ref{prop-Dinv},
$\psi^3\hat T_{0,4}=\psi^3(I-P_\varphi\hat D^*)^{-1}(I-\hat C^*(0))\hat K^*:F_\theta(Y)\to L^\infty(Y)$ lies in 
$\mathcal{R}(1/t^{\beta-})$.
Hence $\psi^3\hat T_4:F_\theta(\tilde Y)\to L^\infty(\tilde Y)$
lies in
$\mathcal{R}(1/t^{\beta-})$.

By Proposition~\ref{prop-U2}, $\hat U:L^\infty(\tilde Y)\to L^1(\tilde Y)$ lies in
$\mathcal{R}(1/t^\beta)$ and the result follows.~
\end{proof}

\begin{prop} \label{prop-a2}
$(1-\psi(b)^3)(\hat\rho(ib)-\hat\rho_4(ib))\in\mathcal{R}(|v|_\infty|w|_\infty 1/t^{\beta-})$.
\end{prop}

\begin{proof}
Using~\eqref{eq-rhoi}, write
$(1-\psi^3)(\hat\rho-\hat\rho_4)=
(1-\psi^3)(\hat\rho_1+\hat\rho_2+\hat\rho_3)=b^{-1}(1-\psi^3)\int_{\tilde Y}\hat U\hat Qv\,w\,d\tilde\mu$ where
$\hat Q(s)=s(\hat T_{0,1}+\hat T_{0,2}+\hat T_{0,3})=
(I-P_\varphi\hat D^*)^{-1}$.

Now $\hat U:L^\infty(\tilde Y)\to L^1(\tilde Y)$ lies in $\mathcal{R}(1/t^\beta)$ by Proposition~\ref{prop-U2}.
Hence $\hat U\hat Q
\in \mathcal{R}(1/t^{\beta-})$ by Proposition~\ref{prop-Dinv}
and Remark~\ref{rmk-Dinv}.   
Also $b^{-1}(1-\psi^3)\in \mathcal{R}(1/t^p)$ for all $p$ by Proposition~\ref{prop-S}(b), so the result follows.
\end{proof}

\begin{prop} \label{prop-a3}   
For $w\in L^{\infty,m}(\tilde Y)$, $m$ sufficiently large, we have that
$(1-\psi(b)^3)\hat\rho(ib)\in \mathcal{R}(\|v\|_\theta |w|_{\infty,m} 1/t^{\beta-})$.
\end{prop}

\begin{proof}
Write $\rho_{v,w}$ to stress the dependence on $v,w$ and similarly
for $\hat\rho_{v,w}$.  
By Proposition~\ref{prop-m},
$\hat\rho_{v,w}(s)=\hat P_m(s)+\hat H_m(s)$, where
$\hat P_m(s)$ is a linear combination of $s^{-j}$, $j=1,\dots,m$, and 
$\hat H_m(s)=s^{-m}\hat \rho_{v,\partial_t^mw}(s)$.

By Proposition~\ref{prop-S}(b), $(1-\psi(b)^3)\hat P_m(ib)\in\mathcal{R}(1/t^p)$
for all $p>0$.  
Next,
\[
\hat\rho_{v,\partial_t^mw}=\int_{\tilde Y}\hat U(I-\hat R)^{-1}v\;\partial_t^mw\,d\tilde\mu,
\]
where $\hat U:L^\infty(\tilde Y)\to L^1(\tilde Y)$ lies in
$\mathcal{R}(1/t^\beta)$ by Proposition~\ref{prop-U2}.
It remains to show that
$Z(b)=b^{-m}(1-\psi(b)^3)(I-\hat R_0(ib))^{-1}:F_\theta(Y)\to F_\theta(Y)$ lies in
$\mathcal{R}(1/t^{\beta-})$.

By Proposition~\ref{prop-smoothR}(b), $\hat R_0$ is $C^{\beta-}$.   
Hence $(I-\hat R_0)^{-1}$ is $C^{\beta-}$ on $\R\setminus\{0\}$ and
$Z$ is $C^{\beta-}$ on $\R$.
Moreover, 
by Proposition~\ref{prop-DolgR} and~Lemma~\ref{lem-approx},
there exists $C,A>0$ such that
$d_{\beta-}(I-\hat R_0(ib))^{-1}\le C|b|^A$ for $|b|>1$.
Hence for $m$ sufficiently large,
$d_{\beta-}Z(ib)\ll |b|^{-2}$ for $|b|>1$.
We conclude from
 Proposition~\ref{prop-S}(a) that $Z\in\mathcal{R}(1/t^{\beta-})$ as required.~
\end{proof}

\begin{pfof}{Lemma~\ref{lem-T4}}
This is immediate by Propositions~\ref{prop-a1},~\ref{prop-a2} and~\ref{prop-a3}.
\end{pfof}

\section{Proof of Theorem~\ref{thm-finite}(b)}
\label{sec-zero}

In this section, we complete the proof of Theorem~\ref{thm-finite}(b).
For this it suffices to replace the estimate for $\rho_3(t)$ in Lemma~\ref{lem-T3} by the improved estimate in Lemma~\ref{lem-zero2} below.

Define $\hat\zeta_j(s)=s^{-1}(\int_Y\hat D^*(s)1_Y\,d\mu)^j$.
Then $\zeta_0\equiv1$ and it follows from the proof of
Proposition~\ref{prop-DD}(a) that $\zeta_1(t)=\zeta(t)$ for $t>k$.

\begin{prop} \label{prop-G}
Let $\beta>1$, $j\ge0$.
Then $\hat\zeta_j\in\mathcal{R}(1/t^{j(\beta-1)})$
if $j(\beta-1)<\beta$ and
$\hat\zeta_j\in\mathcal{R}(1/t^{\beta})$
if $j(\beta-1)>\beta$.
\end{prop}

\begin{proof}  This is proved by continuing inductively the argument in Proposition~\ref{prop-DD}(b).   The details are the same as in~\cite[Lemma~5.1]{Gouezel-sharp}
(with the simplification that there are no noncommutativity issues).
\end{proof}

\begin{lemma} \label{lem-zero2}
Let $1<\beta<2$ and choose $m\ge3$ least such that $m(\beta-1)>\beta$.
Then
\[
\rho_3(t)=\bar v\bar w\sum_{j=2}^{m-1}(1/\bv)^j\zeta_j(t) +O(|v|_\infty|w|_\infty 1/t^{\beta-}).
\]
\end{lemma}

\begin{proof}
Recall that 
\begin{align*}
\hat\rho_j & =(1/\bv)\int_{\tilde Y}\hat U\hat T_jv\,w\,d\tilde\mu,
\quad
(\hat T_jv)(y,u)=(\hat T_{0,j}v^u)(y),
\end{align*}
where
\begin{align*}
\hat T_{0,1}=(1/\bv)s^{-1}P(0), \quad
\hat T_{0,3}=(1/\bv)^3s^{-1}
\Bigl(1-(1/\bv)\int_Y\hat D^*1_Y\,d\mu_Y\Bigl)^{-1}
\Bigl(\int_Y\hat D^*1_Y\,d\mu\Bigr)^2P(0).
\end{align*}
Hence
\begin{align*}
\hat\rho_3 & =(1/\bv)^2\hat\rho_1\Bigl(1-(1/\bv)\int_Y\hat D^*1_Y\,d\mu_Y\Bigl)^{-1}\Bigl(\int_Y\hat D^*1_Y\,d\mu\Bigr)^2
\\ & =\sum_{j=2}^{m-1}(1/\bv)^j\Bigl(\int_Y\hat D^*1_Y\,d\mu\Bigr)^j\hat\rho_1+(1/\bv)^m\hat q\Bigl(\int_Y\hat D^*1_Y\,d\mu\Bigr)^m\hat\rho_1,
\end{align*}
where $\hat q=\bigl(1-(1/\bv)\int_Y\hat D^*1_Y\,d\mu_Y\bigr)^{-1}$.
By Lemma~\ref{lem-T12}(a),
 $\hat\rho_1(s)=s^{-1}\bar v\bar w+\hat h(s)$
where $\hat h\in\mathcal{R}(1/t^\beta)$.
Hence
\begin{align*}
\hat\rho_3 
& =\bar v\bar w\sum_{j=2}^{m-1}(1/\bv)^j\hat\zeta_j+\bar v\bar w(1/\bv)^m\hat q\hat\zeta_m \\
& \qquad  +\sum_{j=2}^{m-1}(1/\bv)^j\Bigl(\int_Y\hat D^*1_Y\,d\mu\Bigr)^j\hat h
+(1/\bv)^m\hat q\Bigl(\int_Y\hat D^*1_Y\,d\mu\Bigr)^m \hat h.
\end{align*}
Now apply Corollary~\ref{cor-D} and 
 Propositions~\ref{prop-Dinv} and~\ref{prop-G}.
\end{proof}

\appendix

\section{Wiener lemma}

This appendix contains material about a version of the Wiener lemma that is
required in Section~\ref{sec-fml}.
We have chosen the notation here to conform with standard conventions in Fourier analysis.  
(In the application of this material, the roles of $f:\R\to\C$ and its Fourier transform $\hat f$ is reversed, with $b$ and $t$ playing the role of $x$ and $\xi$ respectively.)

% Throughout we consider periodic continuous functions $f:[-\pi,\pi]\to\C$ and
% continuous functions $f:\R\to\C$.   
% Functions supported on a closed
% subset of $(-\pi,\pi)$ may be given either interpretation.

Let $\A$ be the Banach algebra of $2\pi$-periodic continuous functions
$f:\R\to\C$ such that their Fourier coefficients $\hat f_n$ are absolutely summable, with norm $\|f\|_{\A}=\sum_{n\in\Z}|\hat f_n|$.
Similarly, let $\mathcal{R}$ be the Banach algebra of continuous functions
$f:\R\to\C$ such that their Fourier transform $\hat f:\R\to\C$ lies in $L^1(\R)$, with norm $\|f\|_{\mathcal{R}}=\int_{-\infty}^\infty|\hat f(\xi)|\,d\xi$.

Given $\beta>1$, we define the Banach algebra $\A_\beta=\{f\in\A:\sup_{n\in\Z}|n|^\beta|\hat f_n|<\infty\}$ with norm
$\|f\|_{\A_\beta}=\sum_{n\in\Z}|\hat f_n|+\sup_{n\in\Z}|n|^\beta|\hat f_n|$.
Similarly, we define the Banach algebra $\mathcal{R}_\beta=\{f\in\mathcal{R}:\sup_{\xi\in\R}|\xi|^\beta|\hat f(\xi)|<\infty\}$ with norm
$\|f\|_{\mathcal{R}_\beta}=\int_{-\infty}^\infty|\hat f(\xi)|\,d\xi+\sup_{\xi\in\R}|\xi|^\beta|\hat f(\xi)|$.

The following Wiener lemmas are standard.

\begin{lemma} \label{lem-Wdiscrete}
Let $\beta>1$ and
let $f,f_1\in\mathcal{A}_\beta$.
Suppose that $f$ is bounded away from zero on the support of $f_1$.

Then there exists $g\in\mathcal{A}_\beta$ such that
$f_1=fg$.
\end{lemma}

\begin{lemma} \label{lem-W}
Let $\beta>1$ and let $f,f_1\in\mathcal{R}_\beta$.
Suppose $f_1$ is compactly supported and that $f$ is bounded away from zero on the support of $f_1$.

Then there exists $g\in\mathcal{R}_\beta$ such that
$f_1=fg$.
\end{lemma}

A statement and proof of Lemma~\ref{lem-Wdiscrete} can be found 
in~\cite[Theorem~1.2.12]{Frenk}.
In this paper, we require Lemma~\ref{lem-W}, but we could not find it stated in the literature.
Hence we provide here a proof of Lemma~\ref{lem-W}, using a standard argument to
reduce to Lemma~\ref{lem-Wdiscrete}.

\begin{lemma}  \label{lem-AR}
Let $\epsilon>0$.  Suppose that $f:\R\to\C$ is a continuous function with
$\supp f\subset[-\pi+\epsilon,\pi-\epsilon]$.
Let $h:\R\to\C$ denote the $2\pi$-periodic continuous function such that
$h|_{[-\pi,\pi]}=f|_{[-\pi,\pi]}$.
Then $f\in\mathcal{R}_\beta$ if and only if $h\in\A_\beta$.
\end{lemma}

\begin{proof}
(cf.\ \cite[Theorem 6.2, Ch.\ VIII, p.~242]{Katzn})
Fix a $C^\infty$ function $\psi:\R\to\R$ supported in $[-\pi+\epsilon/2,\pi-\epsilon/2]$ and such that $\psi\equiv1$ on $[-\pi+\epsilon,\pi-\epsilon]$.
For $\alpha\in[-1,1]$ let $\psi_\alpha(x)=e^{i\alpha x}\psi(x)$.
Then there is a constant $K_0>0$ such that 
\[
|(\widehat\psi_\alpha)_n|\le K_0n^{-\beta},\quad\text{for all $\alpha\in[-1,1]$, $n\in\Z$}.
\]
In particular, $\psi_\alpha\in \A_\beta$ for all $\alpha$
and $\sup_{|\alpha|\le1}\|\psi_\alpha\|_{\A_\beta}<\infty$.

Define $h_\alpha(x)=e^{i\alpha x}h(x)$.
If $h\in\A_\beta$,
then $h_\alpha=h\psi_\alpha\in \A_\beta$ and there is a constant $K>0$ such that $\|h_\alpha\|_{\A_\beta}\le K\|h\|_{\A_\beta}$ for all $\alpha\in[-1,1]$.

Now,
\[
(\widehat{h_\alpha})_n=\frac{1}{2\pi}\int_{-\pi}^\pi e^{i\alpha x}h(x)e^{-inx}\,dx=
\frac{1}{2\pi}\int_{-\infty}^\infty f(x)e^{-i(n-\alpha) x}\,dx=
\frac{1}{2\pi}\hat f(n-\alpha).
\]
%where in the final expression we are viewing $f$ as a function on $\R$.
Hence $\int_{n-1}^n|\hat f(\xi)|\,d\xi=\int_0^1|\hat f(n-\alpha)|\,d\alpha
=2\pi \int_0^1|(\widehat{h_\alpha})_n|\,d\alpha$.
It follows that
\begin{align} \label{eq-A1}
\|f\|_{\mathcal{R}}=2\pi\sum_{n=-\infty}^\infty\int_0^1 |(\widehat{h_\alpha})_n|\,d\alpha=2\pi\int_0^1\|h_\alpha\|_{\A}\,d\alpha\le 2\pi K\|h\|_{\A_\beta}.
\end{align}

Next, we observe that any $\xi\in\R$ can be expressed as $\xi=(n-\alpha)\sgn\xi$
where $n\ge1$, $\alpha\in[0,1]$.
Hence 
\begin{align} \label{eq-A2} \nonumber
\sup_{\xi\in\R}|\xi|^\beta|\hat f(\xi)|
& =\sup_{n\ge1,\,\alpha\in[0,1]} (n-\alpha)^\beta|\hat f((n-\alpha)\sgn\xi)|
\le \sup_{n\ge1,\,\alpha\in[0,1]} n^\beta|\hat f((n-\alpha)\sgn\xi)|
\\ \nonumber  & 
\le \sup_{n\in\Z,\,\alpha\in[-1,1]} |n|^\beta|\hat f(n-\alpha)|
 = 2\pi\sup_{n\in\Z,\,\alpha\in[-1,1]} |n|^\beta|(\widehat{h_\alpha})_n|
\\ & \le  2\pi\sup_{\alpha\in[-1,1]} \|h_\alpha\|_{A_\beta}
\le  2\pi K\|h\|_{A_\beta}.
\end{align}

Combining~\eqref{eq-A1} and~\eqref{eq-A2}, we obtain that $\|f\|_{\mathcal{R}_\beta}\le 4\pi K\|h\|_{A_\beta}$.  Hence we have shown that
$h\in\A_\beta$ implies that $f\in\mathcal{R}_\beta$.

Conversely, suppose $f\in\mathcal{R}_\beta$.
Then $\sum_{n\in\Z}\int_0^1|\hat f(n-\alpha)|\,d\alpha
=\int_{-\infty}^\infty|\hat f(\xi)|\,d\xi<\infty$ and it follows from Fubini that 
$\sum_{n\in\Z}|\hat f(n-\alpha)|<\infty$ for almost every $\alpha$.
Fix such an $\alpha$.  Then 
$\sum_{n\in\Z}|(\widehat{h_\alpha})_n|=(1/2\pi)
\sum_{n\in\Z}|\hat f(n-\alpha)|<\infty$ so that $h_\alpha\in\A$.    Hence
$h=(h_\alpha)_{-\alpha}\in\A$.   Moreover,
\[
\sup_{n\in\Z}|n|^\beta|\hat h_n|
=(1/2\pi) \sup_{n\in\Z}|n|^\beta|\hat f(n)|
\le (1/2\pi) \sup_{\xi\in\R}|\xi|^\beta|\hat f(\xi)|<\infty,
\]
so that $h\in\A_\beta$.
\end{proof}

\begin{pfof}{Lemma~\ref{lem-W}}
(cf.\ \cite[Lemma 6.3, Ch.\ VIII, p.~242]{Katzn})
We make the standard abuse of notation that functions on $\R$ supported on a closed subset of $(-\pi,\pi)$ can be identified with $2\pi$-periodic functions on $\R$.
In particular, the conclusion of Lemma~\ref{lem-AR} becomes
 $f\in\mathcal{R}_\beta$ if and only if $f\in\A_\beta$.

Without loss, we can suppose that $\supp f_1\subset[-2,2]$.
By Lemma~\ref{lem-AR}, $f_1\in\A_\beta$.

Choose a $C^\infty$ function $\chi:\R\to\R$
such that $\supp\chi\subset[-3,3]$
and $\chi\equiv1$ on $[-2,2]$.
Then $\chi\in\mathcal{A}_\beta$ and $\chi\in\mathcal{R}_\beta$.
In particular $\chi f\in\mathcal{R}_\beta$, and by 
Lemma~\ref{lem-AR} $\chi f\in\A_\beta$.

Moreover $\chi f=f$ on $\supp f_1$ and hence is bounded away from zero
on $\supp f_1$.   By Lemma~\ref{lem-Wdiscrete}, there exists
$g_0\in\A_\beta$ such that $f_1=g_0(\chi f)=(g_0\chi)f$.

Since $g_0,\chi\in\A_\beta$, we deduce that $g=g_0\chi\in\A_\beta$.
By Lemma~\ref{lem-AR}, $g\in\mathcal{R}_\beta$.
Hence $f_1=gf$ with $g\in\mathcal{R}_\beta$ as required.
\end{pfof}

\paragraph{Acknowledgements}
The research of IM was supported in part by EPSRC Grant EP/F031807/1 (held at the University of Surrey) and by the European Advanced Grant StochExtHomog (ERC AdG 320977).
The research of DT was supported in part by the European Advanced Grant MALADY (ERC AdG 246953).
IM and DT are grateful to the {\em Centre International de Rencontres Math\'ematiques} for funding the Research in Pairs topic ``Infinite Ergodic Theory'', Luminy, August 2012, where part of this research was carried out.

\end{document}